\documentclass[psamsfonts]{amsart}

\usepackage{amssymb,amsfonts}
\usepackage[all,arc]{xy}
\usepackage{tikz-cd}
\usepackage{enumitem}
\usepackage{mathrsfs}
\usepackage{xcolor}
\usepackage{hyperref}
\usepackage{mathtools}
\usepackage{tensor} 
\usepackage{bbm}

\input xy
\xyoption{all}

\newtheorem{thm}{Theorem}[section]
\newtheorem{cor}[thm]{Corollary}
\newtheorem{prop}[thm]{Proposition}
\newtheorem{lem}[thm]{Lemma}

\newtheorem*{thmcont}{Theorem \continued}
\newcommand{\continued}{??}
\newenvironment{continuethm}[1]
  {\renewcommand{\continued}{\ref{#1}}\thmcont[continued]}
  {\endthmcont}

\theoremstyle{definition}
\newtheorem{defn}[thm]{Definition}

\newtheorem{con}[thm]{Construction}
\newtheorem{exmp}[thm]{Example}

\newtheorem{notn}[thm]{Notation}

\theoremstyle{remark}
\newtheorem{rem}[thm]{Remark}
\newtheorem{rems}[thm]{Remarks}

\newtheorem{claim}[thm]{Claim}

\makeatletter
\let\c@equation\c@thm
\makeatother
\numberwithin{equation}{section}

\setlist[enumerate]{label = \roman*.}

\def\F{\mathbb{F}}
\def\N{\mathbb{N}}

\def\Z{\mathbb{Z}}

\def\1{\mathbbm{1}}

\def\a{\alpha}
\def\b{\beta}
\def\g{\gamma}
\def\d{\delta}
\def\e{\varepsilon}

\def\s{\sigma}
\def\t{\tau}
\def\th{\theta}

\def\vp{\varphi}

\def\vp{\varphi}

\def\S{\Sigma}

\def\O{\Omega}

\def\cA{\mathcal{A}}

\def\cC{\mathcal{C}}
\def\cD{\mathcal{D}}

\def\cS{\mathcal{S}}

\def\fC{\mathsf{G}} 
\def\fG{\mathsf{G}}
\def\mycat{\operatorname{\mathsf{Ch}_{\F}}}
\def\myset{U}
\def\bitimes{\circ} 
\def\biact{\boxtimes}
\def\bifvp{\F_{\vp}^{\text{bm}}}

\def\H{H}
\def\Fvp{\F_{\vp}}

\def\gr{\operatorname{gr}}
\def\colim{\operatorname{colim}}
\def\coeq{\operatorname{coeq}}
\def\Sym{\operatorname{Sym}}
\def\Hom{\operatorname{Hom}}
\def\End{\operatorname{End}}
\def\Aut{\operatorname{Aut}}
\def\Nat{\operatorname{Nat}}
\def\Fun{\operatorname{\mathsf{Fun}}}
\def\Vect{\operatorname{\mathsf{Vect}}}
\def\GrVect{\operatorname{\mathsf{GrVect}}}
\def\Ch{\operatorname{\mathsf{Ch}}}
\def\Ho{\operatorname{\mathsf{Ho}}}
\def\Mod{\operatorname{\mathsf{Mod}}}
\def\Mon{\operatorname{\mathsf{Mon}}}
\def\CMon{\operatorname{\mathsf{CMon}}}
\def\LMod{\operatorname{\mathsf{LMod}}}
\def\Alg{\operatorname{\mathsf{Alg}}}
\def\CAlg{\operatorname{\mathsf{CAlg}}}
\def\Bimod{\operatorname{\mathsf{Bimod}}}
\def\Set{\operatorname{\mathsf{Set}}}
\def\hofib{\operatorname{hofib}}
\def\GL{\operatorname{GL}}
\def\SL{\operatorname{SL}}
\def\sgn{\operatorname{sgn}}
\def\im{\operatorname{im}}
\def\id{\operatorname{id}}
\def\ob{\operatorname{ob}}
\def\tr{\operatorname{tr}}
\def\dc{\circledast}

\def\cop{\operatorname{op}}
\newcommand{\op}[4]{I_{{#1},{#2},{#3}}^{#4}}

\newcommand{\bigdc}{%
  \mathop{\scalebox{1.5}{\raisebox{-0.2ex}{$\circledast$}}
  }
}

\newcommand{\arity}[2]{{#1}\langle{#2}\rangle}
\newcommand{\gpower}[2]{{#1}^{#2}}

\newcommand{\h}[1]{\operatorname{\mathsf{Ho}({#1})}}

\newcommand{\W}[2]{W({#2},{#1})}
\newcommand{\newW}{\mathbf{W}}

\newcounter{lowerindex}

\title{Twisted homology operations for $E_{\infty}$-algebras}

\author{Calista Bernard}
\date{2 April 2023}

\begin{document}


\begin{abstract}
We develop a theory of operations on the twisted homology of
$E_{\infty}$-algebras, generalizing a classical theory developed by J.P.~May.
First we describe a framework suitable for discussing twisted coefficients,
which requires working with $E_{\infty}$-algebras in certain categories of functors.
 In this context, we define twisted versions of the
classical Dyer--Lashof operations, as well as a product.

Moreover, we prove that these distinguished operations generate all operations on
twisted homology by giving a (non-canonical) explicit basis for the homology of free
$E_{\infty}$-algebras in terms of these operations. 
We also make this statement functorial
by proving that the homology of a free $E_{\infty}$-algebra is
a free object in an appropriate category of objects equipped with an action of
the Dyer--Lashof operations.

This theory has applications to the study of $E_{\infty}$-spaces with local
coefficients, though these are not discussed in detail here.
\end{abstract}

\maketitle
\tableofcontents

\section{Introduction}
It is a well-known phenomenon that the homology of an $E_n$-algebra carries
extra structure in the form of homology operations. Such operations arise
directly from
the $E_n$-structure and are natural in maps of $E_n$-algebras. The most
ubiquitous of these operations are a product and a bracket, but there are also additional
operations defined only on mod $p$ homology that exhibit a very rich theory,
and that have proved vital for the study of
$E_n$-algebras.

Classically, these mod $p$ homology operations have been widely studied, and
they have been used to classify all operations on the mod $p$ homology of
$E_n$-algebras \cite{AK,DL,IteratedLoops}.
However, this classical theory applies only to ordinary mod $p$ homology. Given
that homology with twisted coefficients arises naturally in a number of
settings, a version of this theory for twisted coefficients is
foundational work and will likely have a
number of applications. Developing such a theory for the homology of
$E_{\infty}$-algebras with certain twisted coefficient
systems is the goal of this work.

Although potential applications of this theory
are not discussed in detail here, Section \ref{subsection:examples} briefly outlines
how to apply our setup to a few concrete examples.

\subsection{Overview of classical results}
Much like the cup product's failure to be commutative on the chain level gives
rise to the Steenrod operations, the failure of an $E_{\infty}$-multiplication to be
strictly commutative gives rise to operations on the mod $p$ homology of
an $E_{\infty}$-algebra. These operations, usually known as Dyer--Lashof or
Dyer--Lashof--Kudo--Araki operations, were originally defined by Kudo and Araki
for $p=2$ and by Dyer and Lashof for $p>2$ \cite{AK,DL}, and they behave very
similarly to Steenrod operations.

In this work, we predominanly deal with odd primes, so we summarize here the
classical story for $p>2$. Let $\F_p$ denote the field with $p$ elements, and
let $X$ be an $E_{\infty}$-algebra in spaces. The classical Dyer--Lashof
operations are linear maps
\begin{align*}
Q^s&\colon H_*(X;\F_p)\to H_{*+2s(p-1)}(X;\F_p)\\
\b Q^s &\colon H_*(X;\F_p)\to H_{*+2s(p-1)-1}(X;\F_p)
\end{align*}
defined for integers $s\geq 0$ that are natural in maps of $E_{\infty}$-algebras.
These operations satisfy a number of relations, such as the Adem relations and
Cartan formulas, and they are stable in that they commute with
suspension.

In \cite{IteratedLoops}, J.P.~May gives a complete description of the $\F_p$-homology
of free $E_{\infty}$-algebras in terms of the Dyer--Lashof operations together
with a product, a result that is not only useful for computation but also
classifies all homology operations in terms of these particular
operations.
May gives two different versions of this result, one of which gives a functorial
description of the homology of a free $E_{\infty}$-algebra, and the other of
which gives an explicit (non-canonical) basis.

\subsection{Idea of setup}\label{subsection:setup}
In this paper we work with $E_{\infty}$-algebras in categories of functors
$\Fun(\fC,\Ch_{\F})$ from a Picard groupoid $\fG$ to the category $\Ch_{\F}$ of
$\Z$-graded chain
complexes over a field $\F$. A more rigorous discussion of this functor category is
given in Section \ref{section:preliminaries}. Here the
motivating example of $\fC$ is the fundamental groupoid $\pi_{\leq 1}(Y)$ of an
infinite loop
space $Y$.

We can take homology of such functors after evaluating on objects of $\fG$. That
is, given a functor $T\in \Fun(\fG,\Ch_{\F})$, an object $g\in \fG$, and
$m\in\Z$, the $(g,m)$ homology of $T$ is the $\F$--vector space
$$H_{g,m}(T)\coloneqq H_m(T(g)).$$
Thus, our homology inherits an extra grading
from the objects of $\fG$ and carries an action of the morphisms of $\fG$.
There is a monoidal structure known as
\emph{Day convolution} on this functor category, and this allows us to define
homology with coefficients.
\begin{defn}
Let $S,T\in \Fun(\fG,\Ch_{\F})$. For each $g\in \fG$, $m\in \Z$, the $(g,m)$
homology of $T$ with coefficients in $S$ is the $\F$--vector space
$$H_{g,m}(T;S)\coloneqq H_{g,m}(T\dc S),$$
where $\dc$ denotes the Day convolution of functors.
\end{defn}

\subsubsection{Choice of coefficient systems}\label{subsect:coeffs}
We do not expect to be able to develop a reasonable theory of homology operations on
homology with \emph{any} twisted coefficients.
Even the classical theory
due to May is developed only for $\F_p$ coefficients, rather than
for coefficients in $\Z$ or a more general abelian group. The reason for this is
computational: although it is possible to define operations more generally, it
is difficult or impossible to classify all operations in a reasonable way.
Therefore, our first restriction in this work is to consider only coefficient
systems valued in a field $\F$ of characteristic $p$. We do not include a
discussion of operations for fields of characteristic 0, since it is easy to
show that such operations are rather trivial.

The particular coefficient systems we consider, denoted $\Fvp$, are defined
precisely in
Section \ref{section:homology}. When described as a functor
$$\Fvp\colon \fG\to \Ch_{\F},$$
the coefficient system $\Fvp$ vanishes outside the isomorphism class of the
identity object,
and to the identity
$\1_{\fG}\in\fG$, $\Fvp$ assigns the chain complex $\F[0]$ containing a copy of $\F$ in
degree 0 and zeros in other degrees. Finally, $\Aut_{\fG}(\1_{\fG})$ acts on
$\F[0]$ as multiplication along a
homomorphism $\vp\colon \Aut_{\fG}(\1_{\fG})\to \F^{\times}$.

\subsection{Overview of operations}
As discussed in \ref{subsect:coeffs}, we will consider homology with
coefficients in a system $\Fvp$,
which is determined by a field $\F$ and an action of the morphisms of $\fC$ on $\F$.
 This work contains a description of a complete theory of
operations on $\Fvp$-homology for $E_{\infty}$-algebras.
Here we assume that $\F$ has odd characteristic; as we discuss in Section
\ref{section:DLops}, the theory in our setting is nearly identical to the
classical one for characteristic 2.

In Section \ref{section:DLops}, we define versions of the Dyer--Lashof
operations on $\Fvp$-homology. In the literature there are two standard ways of
writing these operations: using lower indexing (which emphasizes how the
operations are defined), and using upper indexing (which makes the operations
stable). Since the upper-indexed operations are more commonly used, we
discuss them here. The lower-indexed operations are defined in
\ref{subsection:lower},
and Theorem \ref{thm:lower index} gives a list of the relations they satisfy. 
There are two classes of upper-indexed operations:
\begin{itemize}
\item the ``untwisted'' Dyer--Lashof operations
\begin{align*}
Q^s&\colon H_{g,n}(-;\Fvp)\to H_{\gpower{g}{p},n+2s(p-1)}(-;\Fvp)\\
\b Q^s&\colon H_{g,n}(-;\Fvp)\to H_{\gpower{g}{p},n+2s(p-1)-1}(-;\Fvp),
\end{align*}
which are defined on $H_{g,n}$ for $g\in \fG$, $n\in\Z$, $s\in \Z$ when the braiding morphism $g\oplus g\to g\oplus
g$ acts
trivially on $\F$;
\item the ``twisted'' Dyer--Lashof operations
\begin{align*}
Q^{s+\tfrac12}&\colon H_{g,n}(-;\Fvp)\to H_{\gpower{g}{p},n+(2s+1)(p-1)}(-;\Fvp)\\
\b Q^{s+\tfrac12}&\colon H_{g,n}(-;\Fvp)\to H_{\gpower{g}{p},n+(2s+1)(p-1)-1}(-;\Fvp),
\end{align*}
which are defined on $H_{g,n}$ for $g\in\fG$, $n\in\Z$, $s\in \Z$ when the braiding morphism $g\oplus g\to g\oplus
g$ acts nontrivially.
\end{itemize}
These operations satisfy versions of the usual relations, such as the Adem
relations and Cartan formulas, a complete list of which is stated in
Theorem \ref{thm:upper index}.

Since $\Fvp$-homology is a representable functor when defined suitably,
we show in Proposition \ref{prop:yoneda} that all $\Fvp$-homology operations are
classified by the homology of certain free $E_{\infty}$-algebras. We prove in
two ways that the homology of free $E_{\infty}$-algebras may be described in
terms of the Dyer--Lashof operations and the product. First, in Section
\ref{section:basis}, we give a non-canonical basis for this homology in terms of
certain ``admissible'' compositions of Dyer--Lashof operations.

\begin{thm}[Informal statement]
If $S$ is a basis for the $\Fvp$-homology of $X$, then the $\Fvp$-homology of
the free $E_{\infty}$-algebra on $X$ is isomorphic as an algebra to the free commutative algebra on
$\{Q^I x\}$, where $x\in S$, and $Q^I$ ranges over all ``admissible''
compositions of Dyer--Lashof operations subject to a degree condition.
\end{thm}
Of course, the notions of ``basis,'' ``free commutative algebra,'' and
``admissible composition'' must all be made precise. The rigorous version of
this statement appears as Theorem \ref{thm:W}.

The above theorem does not capture the full structure of the homology of free
$E_{\infty}$-algebras, since it does not remember the action of the
operations. Moreover, given that the operations are natural in maps of
$E_{\infty}$-algebras, there ought to be a functorial version of this
description. In Section \ref{section:allowable}, we make precise the notion of
an \emph{allowable Dyer--Lashof algebra}. Informally, such an object has a
product and an action of the Dyer--Lashof operations subject to all of the
relations that these operations on homology of $E_{\infty}$-algebras
satisfy.

Our notion of allowable Dyer--Lashof algebras is inspired by that of May
\cite{IteratedLoops}. However, we make two modifications. First, we
rephrase all definitions in the language of modules over $\Fvp$, rather than for
ordinary algebras over $\F_p$. Second, our definition holds for any field of
characteristic $p$, not only for $\F_p$. The main difference here is that the
Dyer--Lashof operations are \emph{Frobenius} linear, rather than linear. We
must therefore define the Dyer--Lashof algebra in a category of
$\F$-$\F$-bimodules, rather than in a category of $\F$--vector spaces to allow
for different right and left vector space structures.

\begin{thm}[Informal statement.]
The $\Fvp$-homology of the free $E_{\infty}$-algebra on $X$ is naturally isomorphic
to the free allowable Dyer--Lashof algebra on $H_{*,*}(X;\Fvp)$.
\end{thm}
The rigorous version of this statement appears as Corollary \ref{cor:free}.
Since this theorem describes the homology of a free $E_{\infty}$-algebra in terms of
the Dyer--Lashof operations and the product, we have thus shown that
these operations generate all $\Fvp$-homology operations for
$E_{\infty}$-algebras.

\subsection{Motivating examples}\label{subsection:examples}
The remainder of this work is concerned with developing a theory of operations
for $E_{\infty}$-algebras; it does not discuss applications of this theory. We
therefore here briefly discuss some motivating examples in order to give the
reader an idea of how this theory may be applied.

First we must describe how the ordinary twisted homology of a space relates to
the setup outlined in \ref{subsection:setup}.

Given a space $X$ with an $E_{\infty}$-algebra structure, we may consider the
\emph{group completion} of $X$, which is a map $X\to Y$ of
$E_{\infty}$-algebras, where $Y$ is \emph{grouplike} (that is, $\pi_0(Y)$ is a
group with respect to the monoid structure inherited from the
$E_{\infty}$-algebra structure). According to the \emph{recognition principle},
$Y$ has the homotopy type of an infinite loop space \cite{M}. Let $f$ denote the
induced map $X\to Y_{\leq 1}$ to the 1-truncation of $Y$, and let
$T\in \Fun(\pi_{\leq 1}(Y),\Ch_{\F})$ be such that for any $y\in Y$,
\begin{equation}\label{equation:hofiber}
T(y)\coloneqq C_*(\hofib_y(f);\F).
\end{equation}
Paths in $Y$ act on the homotopy fiber in the usual way. The functor $T$
inherits the structure of an $E_{\infty}$-algebra in $\Fun(\pi_{\leq
1}(Y),\Ch_{\F})$, and thus our results describe operations on the homology
$H_{*,*}(T;\Fvp)$ for $\Fvp$ any coefficient system as in \ref{subsect:coeffs}.

Suppose we wish to study the ordinary twisted homology $H_*(X;\cA)$ of the space
$X$ with coefficients in a functor
$$\cA\colon \pi_{\leq 1}(X)\to \Vect_{\F}$$
to the category $\Vect_{\F}$ of $\F$--vector spaces. Equip $\Vect_{\F}$ with the
tensor product monoidal structure.
If $\cA$ is a strong (but not necessarily symmetric!) monoidal functor valued in $1$-dimensional vector spaces, then
$\cA$ factors canonically through a strong monoidal functor $\cA'\colon
\pi_{\leq 1}(Y)\to \Vect_{\F}$, and
\begin{equation}\label{eq:twisted-homology}
H_*(X;\cA)\cong \bigoplus_{[a]\in\pi_0 X} H_{f(a),*}(T;\Fvp)
\end{equation}
as graded vector spaces. Here the homomorphism $\vp$ is obtained by restricting
$\cA'$ to $\pi_1(Y,y)$, where $y\in Y$ is the basepoint.

Thus, when presented with a suitable coefficient system, we may use the methods of
this paper to study twisted homology.
Perhaps the simplest example of such a situation is that of symmetric groups with coefficents in the sign
representation.

\begin{exmp}\label{exmp:sgnrep}
Let $X=\bigsqcup BS_n,$ where $S_n$ denotes the symmetric group on $n$ letters.
Then $X$ is equivalent to the free $E_{\infty}$-algebra on a point. If we
wish to compute the homology of the symmetric groups with coefficients in the
sign representation, we may equivalently compute the homology of
$X$ with certain twisted coefficients.

By either the Barratt--Priddy--Quillen theorem \cite{BP,Q} or the recognition
principle \cite{M}, we see that $Y=QS^0=\O^{\infty} S^{\infty}$ is a group
completion of $X$. We therefore work in the category of functors
$\Fun(\fG,\Ch_{\F})$, where $\fG=\pi_{\leq 1}(Y)$. Note that $\pi_{\leq 1}(Y)$
is equivalent
to a skeletal category with objects the integers and with the endomorphisms of
each object isomorphic to $\Z/2\Z$, so for simplicity we replace $\fG$ by its
skeleton.

The coefficient system $\Fvp\in
\Fun(\fG,\Ch_{\F})$ we wish to
consider is the one determined by the homomorphism $\vp\colon \Z/2\Z\to \F^{\times}$
sending the nontrivial element of $\Z/2\Z$ to $-1$. Explicitly, $\Fvp$
 takes on the value $\F$ on the identity element $0$ in $\Z=\ob{\fG}$, and the
 automorphism group $\Z/2\Z$ of 0 acts on $\F$ as multiplication along $\vp$.

If $T\in \Fun(\fG,\Ch_{\F})$ is the functor described in \ref{equation:hofiber}, then the homology of $X$ with
coefficients in the sign representation is isomorphic to the homology of $T$
with coefficients in $\Fvp$, as in \ref{eq:twisted-homology}. That is,
$$\bigoplus_{n\in\N} H_*(S_n;\F^{\text{sgn}})\cong \bigoplus_{n\in\N}
H_{n,*}(T;\Fvp).$$ Moreover, this homology is isomorphic to that of a free $E_{\infty}$-algebra (in the notation
of Section \ref{section:bigradedsusp}, $H_{*,*}(T\dc \Fvp)\cong H_{*,*}(E_{\infty}(\S^{1,0}\Fvp))$).
Theorem
\ref{thm:W} gives an explicit basis for this homology in terms of certain compositions of
Dyer--Lashof operations applied to a single generating class $x$, which should be
thought of as coming from
$H_0(S_1;\F)$. We note that all Dyer--Lashof operations
in this example are ``twisted,'' so we have for instance elements
$$Q^{\tfrac12}(x)\in H_{p-1}(S_p;\F^{\sgn}) \text{  and  } \b Q^{\tfrac12}(x)\in
H_{p-2}(S_p;\F^{\sgn}).$$ Moreover, $x^2=0$. 
\end{exmp}

The previous example of symmetric groups with coefficients in the sign
representation can also be used to study the homology of alternating groups.

\begin{exmp}
If $A_n$ denotes the alternating group on $n$ elements, then for each $n$ there
is a fibration
$$\Z/2\Z\to B\! A_n\to B S_n$$
whose Serre spectral sequence collapses, so that
$$H_*(A_n;\F)\cong H_*(S_n;\F[\Z/2\Z]),$$
where the left-hand side is ordinary $\F$-homology and the right-hand side is
homology with twisted coefficients. A permutation $\s\in S_n$ acts on the group
ring $\F[\Z/2\Z]$ as the identity
if $\sigma$ is even and as multiplication by the nontrivial element of $\Z/2\Z$
if $\sigma$ is odd. In particular, this means that when the characteristic of
$\F$ is not 2, $\F[\Z/2\Z]$ splits as a direct sum of two one-dimensional
representations: a trivial representation
and a sign representation.
Hence, $$H_*(A_n;\F)\cong H_*(S_n;\F)\oplus H_*(S_n;\F^{\sgn}).$$

The ordinary $\F$-homology of symmetric groups is well-known, and in particular
may be computed as the homology of the free $E_{\infty}$-algebra on a point
using May's result \cite{IteratedLoops}. As outlined in the previous example,
the homology of symmetric groups with coefficients in the sign representation
may be computed as the homology of a free $E_{\infty}$-algebra with twisted
coefficients using the methods of this work. 
This then gives a complete description of the homology of
alternating groups in terms of homology operations.

As an example, an elementary computation shows that the abelianizations of $A_3$
and $A_4$ are both $\Z/3\Z$, so $H_1(A_3;\F_3)\cong H_1(A_4;\F_3)\cong \F_3$.
The class $\b Q^{\tfrac12} (x)$ described in the previous example then generates $H_1(A_3;\F_3)$, and $x\b
Q^{\tfrac12} (x)$ generates $H_1(A_4;\F_3)$.
\end{exmp}

\begin{rem}
In the category of spaces, the space $\bigsqcup B\! A_n$ of classifying spaces of
alternating groups does \emph{not} form an $E_{\infty}$-algebra in an obvious
way. However, if we work in the category $\Fun(\fG,\Ch_{\F})$, where again $\fG$
is the fundamental groupoid of
$Y=QS^0$, we \emph{can} consider the alternating groups as an
$E_{\infty}$-algebra. That is,
there is a free $E_{\infty}$-algebra in this functor category whose homology is isomorphic to that of
the alternating groups. Hence, we could directly study operations on the
homology of the alternating groups as an $E_{\infty}$-algebra. However, doing so
requires working with a two-dimensional coefficient system, which does not
fit into the theory developed in this work, and therefore we used the approach
of the previous example instead. It is nonetheless interesting to note
that by working in a different category we can obtain a wider class of
$E_{\infty}$-algebras.
\end{rem}

The previous two examples were very close to the free $E_{\infty}$-algebra on a
point. A more complicated example of a similar flavor is that of general and special linear groups
of a ring $R$.

\begin{exmp}
The space
$$X\coloneqq \bigsqcup B\!\GL_k(R)$$
of classifying spaces of general linear groups has a well-known $E_{\infty}$-structure, although it is not a \emph{free}
$E_{\infty}$-algebra. Galatius--Kupers--Randal-Williams have used
this $E_{\infty}$-structure to study homological stability of general linear
groups for $R$ a field \cite{E3,E4}, and Kupers--Miller--Patzt have used similar
methods for certain other rings $R$, including $R=\Z$ \cite{KMP}. In these works
the general strategy is to build a model out of free $E_{\infty}$-algebras that
approximates the homology of $X$. The homology of this model can then be easily
understood in terms of classical homology operations.

It is tempting to use a similar method to study the space obtained from special linear groups,
$$\bigsqcup B\!\SL_k(R),$$
but as in the example of alternating groups, this space does not inherit an
$E_{\infty}$-algebra structure from that on general linear groups. Once again,
however, the Serre spectral sequence of the fibrations
$$R^{\times}\to B\!\SL_k(R)\to B\!\GL_k(R)$$
yields an isomorphism
$$H_*(\SL_k(R);\F)\cong H_*(\GL_k(R);\F[R^{\times}]).$$
Again the right-hand side is homology with twisted coefficients, where
$\GL_k(R)$ acts on $R^{\times}$ by multiplication by the determinant.

If $R=\Z$ and if the characteristic is not 2, the $\GL_k(\Z)$-representation
$\F[\Z^{\times}]$ splits as a direct sum of two one-dimensional representations, a trivial representation and a
determinant representation. Hence, in order to understand the homology of
$\SL_k(\Z)$, it suffices to understand both the ordinary homology of $\GL_k(\Z)$
and the homology with coefficients in the determinant representation. The former
has already been studied by Kupers--Miller--Patzt \cite{KMP}. The latter could
also be studied using similar methods in combination with the theory of twisted homology operations for
$E_{\infty}$-algebras developed in this work. Here, one should work in the
functor category $\Fun(\pi_{\leq 1}(Y),\Ch_{\F})$, where $Y$ is the group
completion of $\sqcup B\!\GL_k(\Z)$, using a setup as in Example \ref{exmp:sgnrep}.
\end{exmp}

\subsection*{Acknowledgements}
I am deeply grateful to S\o{}ren Galatius for guidance throughout this project. I
would also like to thank Tyler Lawson for a number of very helpful
conversations. My understanding of cellular $E_n$-algebras and of potential
applications of this work benefitted from conversations with Alexander Kupers,
Jeremy Miller, and Oscar Randal-Williams.

Over the course of this project I received partial support from the Danish National Research
Foundation (CPH-GEOTOP-DNRF151) during a stay at the University of Copenhagen
and from the Swedish
Research Council (grant no. 2016-06596) during a stay at the Institut
Mittag-Leffler.

\section{Preliminaries}\label{section:preliminaries}

We begin by recalling some notions about Picard groupoids and Day convolution,
and we introduce key notation for functor categories that will be used
throughout this work. 

Let
$\mathsf{G}=(\mathsf{G},\oplus_{\fG},\1_{\fG},\a^{\fG},\b^{\fG})$ be a Picard groupoid; that is, a symmetric monoidal groupoid in
which every object of $\mathsf{G}$ is invertible. Here, $\a^{\fG}$ denotes the
associator and $\b^{\fG}$ the braiding. 
We will often omit the $\fG$ superscripts when
the category is clear from context.

Since $\fG$ is a Picard groupoid, $\Aut_{\fG}(\1_{\fG})$ is an \emph{abelian}
group, and for every object $g$ in $\fG$, there is a
canonical isomorphism
\begin{equation}\label{eq:trace}
\tr_g\colon \Aut_{\fG}(g)\to \Aut_{\fG}(\1_{\fG}),
\end{equation}
which we refer to as the \emph{trace} of an automorphism.
Moreover, the map
$$\ob{\fG}\to \Aut_{\fG}(\1_{\fG})$$
sending $g\mapsto \tr_{g\oplus g}(\b_{g,g}^{\fG})$ factors through the group
$\pi_0(\fG)$ of isomorphism classes of objects of $\fG$, and this induced map
\begin{equation}\label{eq:trhomo}
\pi_0(\fG)\to \Aut_{\fG}(\1_{\fG})
\end{equation}
is a group homomorphism.

For the remainder of this work, we will utilize the functor category
$$\Ch_{\F}^{\fG}\coloneqq \Fun(\fG,\mycat),$$
where $\fG$ is a Picard groupoid as above, and
$\mycat$ denotes the category of $\Z$-graded chain complexes over a field $\F$.
We endow this functor category with the \emph{Day convolution} monoidal
structure.

\begin{defn}
The \emph{Day convolution} of $F,G\in \Ch_{\F}^{\fC}$, denoted by $F\dc G$, is
the left Kan extension of $\otimes_{\mycat} \circ (F\times G)$ along the product
in $\fC$: 
$$
\begin{tikzcd}
\fC\times \fC \arrow[r, "F\times G"] \arrow[dr, "\oplus_{\fC}"'] & \mycat\times\mycat
\arrow[r,"\otimes_{\mycat}"] \arrow[Rightarrow, d] & \mycat\\
& \fC. \arrow[ur, dotted, "F\dc G"'] & 
\end{tikzcd}
$$
\end{defn}
This endows $\Ch_{\F}^{\fC}$ with a closed symmetric monoidal structure
(see \cite{Day}). 
Note that in the literature $\fG$ is usually taken to be enriched in $\Ch_{\F}$.
Here we will find it useful to distinguish between $\fG$ and a
$\Ch_{\F}$-enriched version of $\fG$. 

\begin{defn}
Let $\F\fG$ denote the category with
objects the objects of $\fG$, and where
$\Hom_{\F\fG}(g,h)$ is the free $\F$--vector space on the set
$\Hom_{\fG}(g,h)$ considered as a chain complex
concentrated in $\Z$-degree 0 with trivial differentials.
\end{defn}

With this notation, the unit for the Day
convolution monoidal structure on $\Ch_{\F}^{\fG}$ is the functor
$$\Hom_{\F\fG}(\1_{\fG},-).$$

We may describe the associators and braiding of $\Ch^{\fG}_{\F}$ concretely
using the universal property of the Day convolution. Let $X,Y,Z\in
\Ch^{\fG}_{\F}$ and $g_1,g_2,g_3\in \fG$. The associator in $\Ch_{\F}$ gives a
map
$$(X(g_1)\otimes Y(g_2))\otimes Z(g_3)\to X(g_1)\otimes (Y(g_2)\otimes
Z(g_3)),$$
and the universal property of Day convolution gives a map
$$X(g_1)\otimes (Y(g_2)\otimes Z(g_3))\to (X\dc (Y\dc Z))(g_1\oplus (g_2\oplus
g_3)).$$
Postcomposition with $X\dc (Y\dc Z)$ applied to the inverse associator
$\a^{-1}_{g_1,g_2,g_3}$ in $\fG$ yields a map
$$(X(g_1)\otimes Y(g_2))\otimes Z(g_3)\to (X\dc (Y\dc Z))((g_1\oplus g_2)\oplus
g_3)$$
that gives the associator
$$(X\dc Y)\dc Z\to X\dc (Y\dc Z)$$
in $\Ch^{\fG}_{\F}$.
Similarly, the braiding is constructed from the map
\begin{equation}\label{eq:braiding}
X(g_1)\otimes Y(g_2)\to Y(g_2)\otimes X(g_1)\to (Y\dc X)(g_2\oplus g_1)\to
(Y\dc X)(g_1\oplus g_2),
\end{equation}
where the first map uses the braiding in $\Ch_{\F}$ and the last map is $(Y\dc
X)(\b^{-1}_{g_1,g_2})$.

Recall that $\mycat$ can be given a cofibrantly-generated monoidal model structure where
weak equivalences are quasi-isomorphisms, and fibrations are degreewise
epimorphisms, and where the unit is cofibrant (see, for example, \cite{Hovey}
Proposition 4.2.13). Since $\fC$ is a small category, this then gives
$\Ch_{\F}^{\fC}$ a cofibrantly generated model structure whose weak
equivalences and fibrations are the natural transformations that are objectwise
weak equivalences and fibrations, respectively (Theorem 11.6.1,
\cite{modelcats}). Moreover, this is a \emph{monoidal} model structure, and the
unit is cofibrant (see the
proof of Lemma 7.9, \cite{SOS}).

\section{Twisted homology}\label{section:homology}
As in Section \ref{section:preliminaries}, we fix a Picard groupoid
$\fG$. In this section we make explicit what we mean by
twisted homology in the context of $\Ch_{\F}^{\fC}$.

Using ordinary homology of chain complexes, we may define the twisted homology of
functors in $\Ch_{\F}^{\fC}$. Our homology now carries an extra grading from the
objects of $\fC$.

\begin{defn}
Let $X,Y\in \Ch_{\F}^{\fC}$. For any object $g$ in $\fG$ and $n\in \Z$, the
\emph{$(g,n)$
homology of $X$ with coefficients in $Y$} is the $\F$--vector space
$$H_{g,n}(X)\coloneqq H_n((X\dc Y)(g)).$$
\end{defn}

\subsection{The coefficient systems $\Fvp$}\label{subsection:Fvp}
In order to obtain a reasonable theory of operations, we work only with
a particularly well-behaved class of coefficient systems.

Fix a group homomorphism
$$\vp\colon \Aut_{\fG}(\1_{\fG})\to \F^{\times}.$$
In Notation \ref{notn:Fvp}, we will define a functor $\Fvp\in \Ch_{\F}^{\fG}$ that is isomorphic to
$$\Hom_{\F\fG}(\1_{\fG},-)\otimes_{\Aut_{\fG}(\1_{\fG})}\F[0],$$ where $\F[0]$ denotes the unit in $\Ch_{\F}$ (that is, $\F[0]$ is the chain
complex with a copy of $\F$ in degree 0 and and zeros elsewhere), and where
$\Aut_{\fG}(\1_{\fG})$ acts on $\F[0]$ as multiplication along $\vp$.
We will then consider homology with coefficients in $\Fvp$.

Although we can give the above explicit description of $\Fvp$, it is
preferable to make a more formal definition that makes it easier to define extra
structure on $\Fvp$.

Let $\mathsf{B}$ denote the category with a single object,
$\1_{\mathsf{B}}$, with
$$\End_{\mathsf{B}}(\1_{\mathsf{B}},\1_{\mathsf{B}})=\Aut_{\fG}(\1_{\fG}).$$ For
ease of notation, we will sometimes write $B$ for
$\Aut_{\mathsf{B}}(\1_{\mathsf{B}})$.
The inclusion $$\iota\colon \mathsf{B}\hookrightarrow \fG$$ is
strong symmetric monoidal if $\mathsf{B}$ is given the obvious symmetric
monoidal structure, and it thus induces a lax symmetric monoidal
functor
$$\iota^*\colon \Ch_{\F}^{\fG}\to \Ch_{\F}^{\mathsf{B}}$$
if $\Ch_{\F}^{\mathsf{B}}$ is also equipped with the Day convolution monoidal
structure. Left Kan extension along $\iota$
defines a left adjoint
$$\iota_*\colon \Ch_{\F}^{\mathsf{B}}\to \Ch_{\F}^{\fG}$$
to $\iota^*$ that is strong monoidal (see Lemma 2.12 of \cite{SOS}).

Let $\Fvp\in \Ch_{\F}^{\mathsf{B}}$ be the functor with
$\Fvp(\1_{\mathsf{B}})=\1_{\Ch_{\F}}=\F[0]$. A morphism $f$ in
$B=\Aut_{\mathsf{B}}(\1_{\mathsf{B}})$ acts
on $\Fvp(\1_{\fG})$ as multiplication by $\vp(f)$.

\begin{notn}\label{notn:Fvp}
In this section (Section \ref{subsection:Fvp}) and in the following section
(Section \ref{subsect:modules}), we will write
$\Fvp$ for the functor we have just defined in $\Ch_{\F}^{\mathsf{B}}$, but
throughout this work we wish to work with coefficients in the functor
$\iota_*(\Fvp)\in \Ch_{\F}^{\fG}$. Therefore, outside of these two sections, we will
by slight abuse of notation write $\Fvp$ for $\iota_*(\Fvp)$.
\end{notn}

\subsection{Modules over $\Fvp$}\label{subsect:modules}
Rather than working with functors in $\Ch_{\F}^{\fC}$ and needing to Day convolve
with $\iota_*(\Fvp)$ before taking homology, it is preferable to work instead with
modules over $\iota_*(\Fvp)$. For this to make sense, we need to endow
$\iota_*(\Fvp)$ with the
structure of a commutative monoid object.

A functor $\mathsf{B}\to \Ch_{\F}$ carries the same
data as a chain complex with an action of the group $B= \Aut_{\fG}(\1_{\fG})$, and
the Day convolution monoidal structure coincides with the tensor product
$\otimes_{B}$ over the group ring $\F[B]$. Thus, there is a canonical multiplication map
$$\Fvp(\1_{\mathsf{B}})\otimes_{B} \Fvp(\1_{\mathsf{B}})\to
\Fvp(\1_{\mathsf{B}})$$
that determines a multiplication map $\Fvp\dc\Fvp\to \Fvp$ in $\Ch_{\F}^{\mathsf{B}}$. The
ring quotient map $\F[B]\to \F$ sending $b\in B$ to $\vp(b)$ determines a unit map
that, together with the above multiplication map, makes $\Fvp$ into a monoid
object in $\Ch_{\F}^{\mathsf{B}}$.

Moreover, we claim that $\Fvp$ is a \emph{commutative} monoid object.
Recall from \ref{eq:braiding}
that the braiding on the Day convolution consists of applying the braiding in
$\Ch_{\F}$ and then acting by the braiding morphism in $\mathsf{B}$. Since
$\Fvp(\1_{\mathsf{B}})$ is concentrated in $\Z$-degree 0, the
braiding in $\Ch_{\F}$ is trivial; moreover, $\1_{\mathsf{B}}$ braids trivially with
itself, since $\mathsf{B}$ is a skeletal category.
 Hence,
the braiding $\Fvp\dc \Fvp\to \Fvp\dc\Fvp$ is trivial, so $\Fvp$ is commutative.

Since $\iota_*$ is strong symmetric monoidal, $\iota_*(\Fvp)$ inherits the structure of a
commutative monoid in $\Ch_{\F}^{\fG}$, and therefore it makes sense to consider
the category $\Mod_{\iota_*(\Fvp)}$ of modules over $\iota_*(\Fvp)$ in
$\Ch_{\F}^{\fG}$. Recall from Notation \ref{notn:Fvp} that we will write $\Mod_{\Fvp}$ for
$\Mod_{\iota_*(\Fvp)}$ outside of this section.

The monoidal structure on $\Mod_{\iota_*(\Fvp)}$ is given by $\dc_{\iota_*(\Fvp)}$, which is defined via the
coequalizer
$$X\dc \Fvp\dc Y \rightrightarrows X\dc Y \to X\dc_{\iota_*(\Fvp)} Y,$$
where $X,Y\in \Mod_{\iota_*(\Fvp)}$, and the two arrows are given by the module
structures
on $X$ and $Y$.

In fact, we claim that being a module over $\iota_*(\Fvp)$ is a \emph{property}, rather
than extra structure.

\begin{prop}\label{prop:property}
Any $X\in \Ch_{\F}^{\fG}$ may be endowed with an
$\iota_*(\Fvp)$-module structure if and only if for 
every object $g\in \fG$ and $f\in \Aut_{\fG}(g)$, the morphism
$$X(f)\colon X(g)\to X(g)$$
agrees with multiplication by $\vp(\tr_g(f))$, where $\tr_g$ is the isomorphism
of Equation \ref{eq:trace}. If such a structure exists, it is
unique, and the map
$$\iota_*(\Fvp)\dc X\to X$$
is an isomorphism.
\end{prop}

\begin{proof}
Note that the unit map $\1_{\Ch^{\mathsf{B}}}\to \Fvp$ may be
described as a quotient: after evaluating on $\1_{\mathsf{B}}$, it is the
coequalizer
$$\bigoplus_{b\in B} \F[B]\rightrightarrows \F[B]\to \F,$$
where as before $B=\Aut_{\fG}(\1_{\fG})$, and the two arrows on the component
indexed by $b$ are multiplication by $b$ and multiplication by $\vp(b)$.
(The vector spaces in this coequalizer should be interpreted as chain complexes
concentrated in $\Z$-degree 0.)

Since left Kan extension preserves colimits, we see that similarly
$\iota_*(\Fvp)$ is a
quotient of $\1_{\Ch^{\fG}}$. Moreover, the map
$$\1_{\Ch^{\fG}}\dc X\to \iota_*(\Fvp)\dc X$$
is also a quotient, since Day convolution is obtained as a left Kan extension.
Thus, a map $\iota_*(\Fvp)\dc X\to X$ making the diagram
$$
\begin{tikzcd}
\1_{\Ch_{\F}^{\mathsf{G}}}\dc X \ar[r] \ar[dr] &
\iota_*(\Fvp)\dc X \ar[d]\\
& X
\end{tikzcd}
$$
commute exists if and only if the relations imposed by the quotient map
$$\1_{\Ch^{\fG}}\to \iota_*(\Fvp)$$ already hold in $X$. If this is the case, the map
$\iota_*(\Fvp)\dc X\to X$ must be an isomorphism.

Explicitly, the canonical isomorphism
$$\1_{\Ch^{\fG}}\dc X=\Hom_{\F\fG}(\1_{\fG},-)\dc X\to X$$
is defined from maps
$$\Hom_{\F\fG}(\1_{\fG},g)\otimes X(h)\to X(g\oplus_{\fG} h)$$
for all $g,h\in \fG$. These maps take a morphism $\1_{\fG}\to g$ to obtain a
morphism $h\to \1_{\fG}\oplus_{\fG} h\to g\oplus_{\fG} h$ in $\fG$, which then induces a map
$X(h)\to X(g\oplus_{\fG} h)$.
We can define maps
$$\iota_*(\Fvp)(g)\otimes X(h)\cong
(\Hom_{\F\fG}(\1_{\fG},g)\otimes_{\Aut_{\fG}(\1_{\fG})}\F[0])\otimes X(h)\to X(g\oplus_{\fG}
h)$$
in exactly the same way, but such maps are well-defined if and only if
precomposing a morphism $\1_{\fG}\to g$ by an automorphism $b\colon \1_{\fG}\to
\1_{\fG}$ results in multiplication by $\vp(b)$. Tracing through the definition
of $\tr_h$, we see that this is equivalent to saying that any morphism $f\colon
h\to h$ acts on $X(h)$ as multiplication by $\vp(\tr_h(f))$.

\end{proof}

In fact, this proposition implies that the natural map
$$X\dc Y\to X\dc_{\iota_*(\Fvp)} Y$$
is an isomorphism: the two arrows
$$X\dc\iota_*(\Fvp)\dc Y\rightrightarrows X\dc Y$$
given by the the module structures on $X$ and $Y$ already commute, and therefore
$X\dc Y$ satisfies the universal property of the coequalizer. In light of this,
we will typically write $\dc$ for the monoidal structure in $\Mod_{\iota_*(\Fvp)}$ rather
than $\dc_{\iota_*(\Fvp)}$.

Note that for any $X\in \Mod_{\iota_*(\Fvp)}$, the isomorphism $X\dc
\iota_*(\Fvp)\to X$
induces an isomorphism
$$H_{g,n}(X;\iota_*(\Fvp))\cong H_{g,n}(X).$$
For any $Y\in \Ch_{\F}^{\fC}$, we can therefore express $H_{*,*}(Y;\iota_*(\Fvp))$ as the
homology of the free $\iota_*(\Fvp)$-module on $Y$. We thus choose to consider the
category $\Mod_{\iota_*(\Fvp)}$ for the purposes of discussing $\iota_*(\Fvp)$-homology.

\subsubsection{Model structure on $\Mod_{\Fvp}$}
We wish to endow $\Mod_{\iota_*(\Fvp)}$ with a cofibrantly-generated monoidal model structure.

Let $\pi_0\fG$ denote the set of iosmorphism classes of objects in $\fG$, and
choose a representative $g_0\in \fG$ of each class $[g]\in\pi_0\fG$.
Consider the forgetful functor
$$\Mod_{\iota_*(\Fvp)}\to \prod_{ \pi_0\fG} \Ch_{\F}$$
taking a functor $X$ to $\prod X(g_0)$. In light of Proposition
\ref{prop:property}, this has a left adjoint given by taking $\prod
X_{g_0}$ to $$\bigoplus_{g_0\in\pi_0\fG}
\Hom_{\F\fG}(g_0,-)\otimes_{\Aut_{\fG}(g_0)}X_{g_0},$$ where an automorphism $b\in
\Aut_{\fG}(g_0)$ acts on $X_{g_0}$ by multiplication by
$\vp(\tr_{g_0}(b))$.

We claim that this is in fact an \emph{equivalence} of categories. It is easy to
see that $\prod \Hom_{\F\fG}(g_0,g_0)\otimes_{\Aut_{\fG}(g_0)} X_{g_0}$ is naturally
isomorphic to $\prod X_{g_0}$. For the other composition, note that for any
$X\in \Mod_{\iota_*(\Fvp)}$ there is a map
$$\Hom_{\F\fG}(g_0,g)\otimes_{\Aut_{\fG}(g_0)} X(g_0)\to X(g)$$
obtained by using the morphism $g_0\to g$ to map $X(g_0)\to X(g)$. This has
inverse
$$X(g)\to \Hom_{\F\fG}(g_0,g)\otimes_{\Aut_{\fG}(g_0)} X(g_0)$$
obtained by choosing any $\g\colon g_0\to g$ and sending $x\in X(g)$ to
$\g\otimes \g^{-1}(x)$. This map is independent of the choice of $\g$, and these maps are part of a natural isomorphism between $X$
and $\bigoplus_{\pi_0\fG}\Hom_{\F\fG}(g_0,-)\otimes_{\Aut_{\fG}(g_0)} X(g_0).$

Since the category $\Ch_{\F}$ has a cofibrantly-generated model
structure with weak equivalences the quasi-isomorphisms and fibrations the levelwise
surjections (see \cite{Hovey} Section 2.3), it follows that $\prod \Ch_{\F}$
may be given
the ``product'' model structure. That is, a morphism $f\in \prod \Ch_{\F}$ is a
weak equivalence/fibration/cofibration precisely when $f_g$ is a weak
equivalence/fibration/cofibration in $\Ch_{\F}$ for all $g\in \fG$.
Thus, $\Mod_{\iota_*(\Fvp)}$ inherits a cofibrantly-generated model structure whose weak
equivalences and fibrations are precisely those that are weak
equivalences/fibrations in $\Ch_{\F}$ after evaluating at each object $g\in \fG$.
Note that in this model structure all objects are fibrant. Since we are working
over a field, every chain complex in $\Ch_{\F}$ is cofibrant, and hence every
object of $\Mod_{\iota_*(\Fvp)}$ is cofibrant as well.

Checking that this is a monoidal model structure amounts to checking the
pushout-product axiom and that the unit is cofibrant (see \cite{Hovey} Section
4.2). In our case the unit $\iota_*(\Fvp)$ is clearly cofibrant since all objects of
$\Mod_{\iota_*(\Fvp)}$ are cofibrant. Checking the
pushout-product axiom reduces to checking that the pushout-product axiom holds
in $\Ch_{\F}$, since pushouts in $\Mod_{\iota_*(\Fvp)}$ are computed levelwise. This may
be checked easily by hand (see \cite{Hovey} Proposition 4.2.13). Thus, we have
proven the following.

\begin{prop}
There exists a cofibrantly-generated monoidal model structure on
$\Mod_{\iota_*(\Fvp)}$
whose fibrations are level-wise surjections of chain complexes and whose weak
equivalences are level-wise quasi-isomorphisms.
\end{prop}

It will be useful to relate this model structure to that of $\Ch_{\F}^{\fC}$. There is a free-forgetful adjunction
$$\begin{tikzcd}
\iota_*(\Fvp)\dc- \colon\Ch^{\fG}_{\F}\arrow[r,shift left] & \Mod_{\iota_*(\Fvp)}\colon U
\arrow[l,shift left].
\end{tikzcd}$$
If a morphism in $\Mod_{\iota_*(\Fvp)}$ is a
fibration, and therefore levelwise surjective, then its image under $U$ is also
levelwise surjective and therefore a fibration in $\Ch^{\fG}_{\F}$.
Similarly, $U$ preserves quasi-isomorphisms. Thus, this is a Quillen adjunction.

\begin{rem}
We could instead have defined the model structure on $\Mod_{\iota_*(\Fvp)}$ via right
transfer along this adjunction. However, since $\iota_*(\Fvp)$ is not cofibrant in
$\Ch^{\fC}_{\F}$, checking that this is indeed a model structure requires more
work. While this is not impossible, it was faster to use $\prod \Ch_{\F}$ as
above.
\end{rem}

\subsubsection{Bigraded suspension of functors}\label{section:bigradedsusp}

Now we define the bigraded suspension $\S^{g,n}$ of modules over $\iota_*(\Fvp)$, since these will be useful later in
discussing homology. On the level of objects, 
\begin{equation}\label{eq:informalsusp}(\S^{g,n}X)(h)=\S^nX(h\oplus
g^{-1}),\end{equation} but we
will need to make this precise (also note that we have not chosen unique
inverses for each object $g\in \fG$, so the right-hand expression does not make
sense).

For $n\in \Z$, let $\S^n\colon \Ch\to \Ch$ denote the usual $n$th suspension
functor. Given any category of functors into $\Ch$, postcomposition with $\S^n$
defines an endofunctor. In light of Proposition
\ref{prop:property}, we see that it also makes sense to consider $\S^n$ as an
endofunctor of $\Mod_{\iota_*(\Fvp)}$.

We also wish to consider ``suspension'' in the $\fG$ direction.
For each $g\in\fG$, we can consider the functor
$$\iota_g\colon \mathsf{B}\to \mathsf{G}$$
taking $\1_{\mathsf{B}}\mapsto g$. An automorphism of $\1_{\mathsf{B}}$ is sent
to its image under
$$\tr^{-1}_g\colon \Aut_{\fG}(\1_{\fG})\to \Aut_{\fG}(g).$$ As before,
precomposition with $\iota_g$ induces a functor
$$\iota^*_{g}\colon \Ch^{\fG}_{\F}\to \Ch^{\mathsf{B}}_{\F}$$
that has a left adjoint
$$(\iota_g)_*\colon \Ch^{\mathsf{B}}_{\F}\to \Ch^{\fG}_{\F}.$$

\begin{defn}\label{defn:Gsus}
For $g\in\fG$ and $n\in\Z$, let $\S^{g,n}\iota_*(\Fvp)\coloneqq (\iota_g)_*(\Fvp)$.
If $X$ is an object of $\Mod_{\iota_*(\Fvp)}$, let $$\S^{g,n}X\coloneqq \S^{g,n}\iota_*(\Fvp)\dc X\in
\Mod_{\iota_*(\Fvp)}.$$
\end{defn}

Note that when $g=\1_{\fG}$, $\S^{g,0}\iota_*(\Fvp)=\iota_*(\Fvp)$.
Moreover,
$$\S^{g,n}\iota_*(\Fvp)\cong
\Hom_{\F\fG}(g,-)\otimes_{\Aut_{\fG}(g)}\F[n],$$
where $\Aut_{\fG}(g)$ acts on $\F[n]=\S^n\F[0]$ as multiplication along $\vp$ after
applying the trace homomorphism. From this description, combined with
Proposition \ref{prop:property}, we see that $\S^{g,n}X$ is indeed a module
over $\iota_*(\Fvp)$, and that this definition is consistent with the informal
description of \ref{eq:informalsusp}.

We remark that the $\S$ notation is here ambiguous; when the
superscript is an element $n\in \Z$, $\S^n$ will always denote the suspension
with respect to the chain complex grading, and when the superscript is bigraded,
$\S^{g,n}$ will
be as in Definition \ref{defn:Gsus}.

We claim that this bigraded suspension behaves well with respect to the monoidal
structures on both gradings.

\begin{prop}
Let $\fG\times \Z$ denote the product of $\fG$ with the category whose objects
are the integers and that has no non-identity morphisms. The functor
$$\fG\times\Z\to \Mod_{\iota_*(\Fvp)}$$
sending $(g,n)\mapsto \S^{g,n}\iota_*(\Fvp)$ is strong monoidal. In particular,
$$\S^{g,n}\iota_*(\Fvp)\dc \S^{h,m}\iota_*(\Fvp)\cong \S^{g\oplus h,
n+m}\iota_*(\Fvp).$$
\end{prop}

\begin{proof}
Since $(\iota_g)_*$ and $(\iota_h)_*$ are left Kan extensions, we have a map
$$\S^n\Fvp(\1_{\mathsf{B}})\otimes \S^m\Fvp(\1_{\mathsf{B}})\to \S^{g,n}\Fvp(g)\otimes
\S^{h,m}\Fvp(h)$$
that is easily seen to be an isomorphism. Using the inverse of this isomorphism,
we obtain the composition
$$\S^{g,n}\Fvp(g)\otimes \S^{h,m}\Fvp(h)\to \S^n\Fvp(\1_{\mathsf{B}})\otimes
\S^m\Fvp(\1_{\mathsf{B}})\to \S^{n+m}\Fvp(\1_{\mathsf{B}})\to \S^{g\oplus
h,n+m}\Fvp(g\oplus h).$$
Given any morphisms $\gamma\colon g\to k$ and $\delta\colon h\to \ell$, we
obtain the composition
$$
\begin{tikzcd}[column sep=large]
\S^{g,n}\Fvp(k)\otimes \S^{h,m}\Fvp(\ell) \arrow[r, "\gamma^{-1}\otimes
\delta^{-1}"] \arrow[d, dashed] & \S^{g,n}\Fvp(g)\otimes \S^{h,m}\Fvp(h)\ar[d]\\
\S^{g\oplus h,n+m}\Fvp(k\oplus \ell) &
\S^{g\oplus h,n+m}\Fvp(g\oplus h)\ar[l, "\gamma\oplus \delta"]
\end{tikzcd}
$$
Proposition \ref{prop:property} implies that this map is independent of the
choice of morphisms $\gamma,\delta$, and therefore this map is natural with
respect to morphisms in $\fG$. Noting that the left-hand side is zero when $k\not\cong
g$ or $\ell\not\cong h$, we thus have maps
$$(\S^{g,n}\iota_*(\Fvp))(k)\otimes (\S^{h,m}\iota_*(\Fvp))(l)\to (\S^{g\oplus
h,n+m}\iota_*(\Fvp))(k\oplus l)$$
for all $k,\ell\in \fG$
that assemble together into a natural transformation
$$\S^{g,n}\iota_*(\Fvp)\dc \S^{h,m}\iota_*(\Fvp)\to \S^{g\oplus
h,n+m}\iota_*(\Fvp).$$
Checking that this map is an isomorphism is elementary.
\end{proof}

\subsection{A K{\"u}nneth theorem for $\Fvp$-homology}
In this section, and in all subsequent sections, we write $\Fvp$ for
$\iota_*(\Fvp)\in \Ch_{\F}^{\fG}$ as described in Notation \ref{notn:Fvp}.

Note that for any $X\in \Mod_{\Fvp}(\Ch_{\F}^{\fC})$, we may regard $H_{*,*}(X)$ as a functor
$$H_{*,*}(X)\colon \mathsf{G}\to \GrVect_{\F}$$
to the category of $\Z$-graded $\F$--vector spaces.
By a slight abuse of notation, we again write $\Fvp$ 
for $H_{*,*}(\Fvp)\in \GrVect_{\F}^{\fC}$. 

We may make $\GrVect_{\F}^{\fC}$
into a symmetric monoidal category using Day convolution and the symmetries in
$\fC$ and $\GrVect_{\F}$. 
Once again, $\Fvp$ is a commutative
monoid object in $\GrVect_{\F}^{\fC}$, and being a monoid over $\Fvp$ is a
property rather than extra structure as in Proposition \ref{prop:property}.
Moreover, the monoidal structure $\dc_{\Fvp}$ of
$\Mod_{\Fvp}(\GrVect_{\F}^{\fC})$ again coincides with the monoidal structure
$\dc$ of $\GrVect_{\F}^{\fC}$. Note that the unit
object in $\Mod_{\Fvp}(\GrVect_{\F}^{\fC})$ is $\Fvp$.

\begin{rem}
We could instead have chosen to work with the equivalent category
$\Vect_{\F}^{\fC\times\Z}$, where $\Z$ denotes the symmetric monoidal category
of the natural numbers under addition with no non-identity morphisms. In either
case, we are essentially working with $(\ob{\fG}\times \Z)$-graded vector spaces with
extra structure.
\end{rem}

We conclude this discussion by proving a K{\"u}nneth theorem for
$\Fvp$-homology.

\begin{prop}\label{prop:Kunneth}
Suppose $X,Y\in \Mod_{\Fvp}(\Ch^{\fC}_{\F})$. Then $$H_{*,*}(X\dc Y)\cong
H_{*,*}(X)\dc H_{*,*}(Y)$$
in $\GrVect_{\F}^{\fC}$. In particular, $$H_{*,*}(X)\cong H_{*,*}(\Fvp\dc X)\cong \Fvp\dc
H_{*,*}(X),$$
so $H_{*,*}(X)$ inherits the structure of an $\Fvp$-module in
$\GrVect^{\fC}_{\F}$.
\end{prop}

In light of this, it is more natural to view homology as a functor
$$H_{*,*}\colon \Mod_{\Fvp}(\Ch_{\F}^{\fC})\to
\Mod_{\Fvp}(\GrVect_{\F}^{\fC})$$
landing in the category of modules over $\Fvp$ in $\GrVect_{\F}^{\fC}$.

\begin{proof}
Recall from Section \ref{subsect:modules} that
$X\dc_{\Fvp}Y\cong X\dc Y.$ For an object $k$ of $\fG$, the Day convolution can
be described as a colimit over the category $\fG\times \fG/k$ whose objects are
triples $(g,h,f)$, where $g,h\in \fG$ and $f\colon g\oplus h\to k$, and whose
morphisms
$$(g,h,f)\to (g',h',f')$$ consist of morphisms $\a\colon g\to
g'$ and $\b\colon h\to h'$ such that $f=f'\circ (\a\oplus \b)$. Explicitly,
$$(X\dc Y)(k)\cong \colim_{\fG\times \fG/k} X(g)\otimes Y(h),$$
where a morphism $(\a,\b)$ from $(g,h,f)\to (g',h',f')$ induces the map
$$X(\a)\otimes Y(\b)\colon X(g)\otimes Y(h)\to X(g')\otimes Y(h').$$
If we choose one representative $g_0$ of each isomorphism class in $\pi_0\fG$,
we can express this colimit as
$$\colim_{\fG\times \fG/k} X(g)\otimes Y(h)\cong \bigoplus_{[g_0]\in \pi_0\fG}
\colim_{\{(g,h,f)\mid g\cong g_0\}} X(g)\otimes Y(h).$$
Note that if $g\cong g_0$, this forces $h\cong g_0^{-1}\oplus k$, where
$g_0^{-1}$ is any inverse of $g_0$.
The $g_0$ summand of this coproduct is isomorphic to
$$X(g_0)\otimes_{\Aut_{\fG}(\1_{\fG})}Y(h_0),$$
where $h_0$ is any choice of representative of the class $[g_0^{-1}\oplus k]$.
Here $\Aut_{\fG}(\1_{\fG})$ acts on $X(g_0)$ and $Y(h_0)$ as multiplication along
$\vp$, which is consistent with the actions of automorphisms of $g_0$ and $h_0$,
respectively, in light of Proposition \ref{prop:property}.
Since the action of $\Aut_{\fG}(\1_{\fG})$ factors through $\F^{\times}$,
$$X(g_0)\otimes_{\Aut_{\fG}(\1_{\fG})}Y(h_0)\cong X(g_0)\otimes_{\F} Y(h_0).$$

Since we are working over a field,
$$H_*(X(g_0)\otimes_{\F}Y(h_0))\cong H_*(X(g_0))\otimes_{\F} H_*(Y(h_0)).$$
A similar argument to that above shows that
$$(H_{*,*}(X)\dc H_{*,*}(Y))(k)\cong \bigoplus_{\substack{[g_0],[h_0]\in\pi_0\fG\\
[g_0\oplus h_0]=[k]}} H_*(X(g_0))\otimes_{\F}
H_*(Y(h_0)).$$
This shows that the desired isomorphism holds abstractly.

For any $g,h\in \fG$, the usual K{\"u}nneth isomorphism together with the
universal property of Day convolution give a map
$$H_*(X(g))\otimes_{\F} H_*(Y(h))\to H_*(X(g)\otimes_{\F} Y(h))\to H_*((X\dc
Y)(g\oplus h))$$
that induces a map
$$H_{*,*}(X)\dc H_{*,*}(Y)\to H_{*,*}(X\dc Y).$$ This map is compatible with the
above isomorphisms, so it is an isomorphism.
\end{proof}

Note that the K{\"u}nneth theorem need not hold for arbitrary functors in
$\Ch_{\F}^{\fG}$; it was necessary that $X$ and $Y$ be modules over $\F_{\vp}$
so that the action of automorphisms in $\fG$ can be expressed as multiplication
by an element of $\F$.

\section{$E_{\infty}$-algebras in $\Mod_{\Fvp}$}\label{section:einfty}

Given a functor $F\in\Ch_{\F}^{\fC}$, we may tensor a chain complex $V$ with
$F$ to obtain a functor $V\otimes F$. Explicitly,
$$(V\otimes F)(g)=V\otimes_{\F} F(g),$$
and a morphism $b\in \Aut(g)$ acts as $\text{id}_{V}\otimes b$ on $V\otimes
F(g)$. This defines a functor $-\otimes-\colon  \Ch_{\F}\times \Ch_{\F}^{\fC}\to
\Ch_{\F}^{\fC}$ satisfying unit an associativity axioms, so that
$\Ch_{\F}^{\fC}$ is tensored over $\Ch_{\F}$.
This
induces a levelwise tensor product
$$-\otimes-\colon \Ch_{\F}\times \Mod_{\Fvp}\to \Mod_{\Fvp}$$
so that $\Mod_{\Fvp}$ is tensored over $\Ch_{\F}$ as well, and the induced
map $\Ch_{\F}\to \Mod_{\Fvp}$ given by tensoring with the unit $\Fvp$ is strong symmetric monoidal. 

For any operad $\mathscr{O}$ in chain complexes, we can now make sense of
algebras over $\mathscr{O}$ in $\Mod_{\Fvp}$.

\begin{defn}
Let $\mathscr{O}$ be an operad in chain complexes. An
\emph{$\mathscr{O}$-algebra} in $\Mod_{\Fvp}$ consists of an object
$X\in\Mod_{\Fvp}$ together with maps
$$\mathscr{O}(n)\otimes X^{\dc n}\to X$$
for each $n$ that satisfy the usual associativity, unit, and equivariance
relations.
Here $\mathscr{O}(n)$ denotes the $n$th
arity of the operad.

A morphism of $\mathscr{O}$-algebras is a map in $\Mod_{\Fvp}$ that preserves
the action of $\mathscr{O}$. We write $\Alg_{\mathscr{O}}(\Mod_{\Fvp})$ for the
category of $\mathscr{O}$-algebras in $\Mod_{\Fvp}$.
\end{defn}

Of course, we are interested in the case where $\mathscr{O}$ is an
$E_{\infty}$-operad. In order to obtain a reasonable homotopy theory for
$E_{\infty}$-algebras, we need to make some technical assumptions on
$E_{\infty}$-operads.

First of all, the category of symmetric sequences in
$\Mod_{\Fvp}$ inherits a model structure where a map is a weak equivalence
(respectively fibration) if and only if its $n$th arity map is a weak
equivalence (respectively fibration) in $\Ch_{\F}$ for all $n$ (see, e.g.,
\cite{Fresse}). We say an operad
is \emph{$\S$-cofibrant} if its underlying symmetric sequence is cofibrant with
respect to this model structure. 

When an operad $\mathscr{O}$ is
$\S$-cofibrant, the category $\Alg_{\mathscr{O}}(\Mod_{\Fvp})$
inherits a ``semi-model'' structure (in the sense of \cite{Fresse}), where a map
of $\mathscr{O}$-algebras
is a weak equivalence (respectively, fibration) if and only if its underlying
map in $\Mod_{\Fvp}$ is a weak equivalence (respectively, fibration)
(Theorem 12.3.A, \cite{Fresse}). 
Moreover, a weak equivalence of $\S$-cofibrant operads induces a (semi-model) Quillen
equivalence of categories of algebras (Theorem 12.5.A, \cite{Fresse}). Thus, we
consider the following definition of $E_{\infty}$-operads.

\begin{defn}
An \emph{$E_{\infty}$-operad} consists of a $\S$-cofibrant operad $\cC_{\infty}$
in chain complexes together with a zig-zag of weak equivalences to singular
chains on the little
$\infty$-disks operad. 
An \emph{$E_{\infty}$-algebra} is an algebra over an $E_{\infty}$-operad.
\end{defn}

In particular, this means that for any $E_{\infty}$-operad $\cC_{\infty}$, the
category $\Alg_{\cC_{\infty}}(\Mod_{\Fvp})$ inherits a (semi-)model structure,
and any two $E_{\infty}$-algebras have Quillen-equivalent categories of algebras.

We may describe the free
$\cC_{\infty}$-algebra on $X\in\Mod_{\Fvp}$ by
$$\cC_{\infty}(X)\coloneqq \bigoplus_{k\geq 0} \cC_{\infty}(k)\otimes_{\S_k}
X^{\circledast k},$$
and we write
$$\Sym_{\cC_{\infty}}^k(X)\coloneqq \cC_{\infty}(k)\otimes_{\S_k} X^{\circledast
k}$$
for the $k$th symmetric power of $X$.
The functor $\cC_{\infty}(-)$ is left adjoint to the forgetful functor
$$\Alg_{\cC_{\infty}}(\Mod_{\Fvp})\to \Mod_{\Fvp}.$$

The following lemma will allow us to simplify some computations by replacing functors in $\Mod_{\Fvp}$ by their
homology.

\begin{lem}\label{lem:preservewes}
If $\cC_{\infty}$ is any $E_{\infty}$-operad, the functors $\cC_{\infty}(-)$ and $\Sym^k_{\cC_{\infty}}(-)$ preserve weak
equivalences in $\Mod_{\Fvp}$.
\end{lem}

\begin{proof}
Since 
$$\cC_{\infty}(X)\cong \bigoplus \Sym^k_{\cC_{\infty}}(X),$$
we have that $\cC_{\infty}(-)$ preserves weak equivalences if
$\Sym^k_{\cC_{\infty}}(-)$ does for all $k$.
The $k$th power functor $(-)^{\dc
k}\colon \Mod_{\Fvp}\to \Mod_{\Fvp}$ preserves weak equivalences by the K\"{u}nneth
Theorem \ref{prop:Kunneth}. Since the operad $\cC_{\infty}$ is 
$\S$-cofibrant, we have in particular that $\cC_{\infty}(k)$ is cofibrant as a
chain complex with a $\S_k$ action; that is, $\cC_{\infty}(k)$ is a chain
complex of projective $\F[\S_k]$-modules. Hence, $\cC_{\infty}(k)\otimes_{\S_k}
-\colon \Ch_{\F[\S_k]}\to \Ch_{\F[\S_k]}$ preserves quasi-isomorphisms.

Since a map in
$\Mod_{\Fvp}$ is a weak equivalence precisely when all of the underlying maps of chain
complexes are quasi-isomorphisms after evaluating at objects in $\fG$, it follows
that a weak equivalence $X\to Y$ in $\Mod_{\Fvp}$ gives quasi-isomorphisms
$$X^{\dc k}(g)\to Y^{\dc k}(g).$$ We may view this as a quasi-isomorphism in
$\Ch_{\F[\S_k]}$ using the $\S_k$-action inherited from the braiding, and hence,
we have that $\Sym^k(X)\to \Sym^k(Y)$ is a quasi-isomorphism after evaluating on
each $g\in\fG$, which is equivalent to being a weak equivalence in $\Mod_{\Fvp}$. 
\end{proof}

For constructing homology operations, it is convenient to work with a choice of
$E_{\infty}$-operad, so we conclude this discussion by showing that the theory of homology
operations is independent of the choice of $\cC_{\infty}$.

\begin{lem}\label{lem:homindep}
If $\cC_{\infty}$ and $\cD_{\infty}$ are two $E_{\infty}$-operads equipped with
a weak equivalence $\cC_{\infty}\to \cD_{\infty}$, then for any $X\in
\Mod_{\Fvp}$,
$$H_{*,*}(\cC_{\infty}(X))\cong H_{*,*}(\cD_{\infty}(X))$$
as objects in $\GrVect^{\fC}_{\F}$.
\end{lem}

\begin{proof}
The weak equivalence $\cC_{\infty}\to \cD_{\infty}$ induces a map
$$\Alg_{\cD_{\infty}}(\Mod_{\Fvp})\to \Alg_{\cC_{\infty}}(\Mod_{\Fvp})$$
that is the identity on underlying objects, and where the
$\cC_{\infty}$-structure on a $\cD_{\infty}$-algebra $X$ is given by the
composition
$$\Sym^k_{\cC_{\infty}}(X)\to \Sym^k_{\cD_{\infty}}(X)\to X.$$
This functor has a left adjoint that sends free $\cC_{\infty}$-algebras to free
$\cD_{\infty}$-algebras. 
In particular, we claim this induces an isomorphism on homology
$$H_{*,*}(\cC_{\infty}(X))\to H_{*,*}(\cD_{\infty}(X)).$$
It suffices to check this on $\Sym^k$ for each $k$ separately. 
Since $\cC_{\infty}(k)\to \cD_{\infty}(k)$ is a
quasi-isomorphism and $X^{\dc k}$ is cofibrant in $\Mod_{\Fvp}$ (as all objects
of $\Mod_{\Fvp}$ are cofibrant), we have a quasi-isomorphism
$$\cC_{\infty}(k)\otimes_{\S_k} X^{\dc k}\to \cD_{\infty}(k)\otimes_{\S_k}
X^{\dc k}.$$
\end{proof}

We will see later that operations on the $\Fvp$-homology of algebras over an
$E_{\infty}$-operad $\cC_{\infty}$ are classified by homology of free
$\cC_{\infty}$-algebras. In light of the above proposition, therefore, the
theory of homology operations is independent of the choice of
$E_{\infty}$-operad. (See also Remark \ref{rem:anyoperad}.) For convenience, we thus fix $\cC_{\infty}$ to be singular
chains on the
little $\infty$-disks operad, and write
$$\Alg_{E_{\infty}}(\Mod_{\Fvp}) \text{, } E_{\infty}(X)\text{, and } \Sym^k(X)$$
for $\Alg_{\cC_{\infty}}(\Mod_{\Fvp})$, $\cC_{\infty}(X)$, and
$\Sym^k_{\cC_{\infty}}(X)$, respectively.

\subsection{Monoidal structure on $\Alg_{E_{\infty}}(\Mod_{\Fvp})$}

Since the little $\infty$-disks operad $\cC_{\infty}$ arises as singular chains
of an operad in spaces, it inherits a coproduct
$$\cC_{\infty}(k)\to \cC_{\infty}(k)\otimes_{\F} \cC_{\infty}(k)$$
in each arity from the diagonal map of underlying spaces. Using this structure,
we may define a monoidal structure on $E_{\infty}$-algebras.

If $X,Y\in \Alg_{E_{\infty}}(\Mod_{\Fvp})$, we may endow the product $X\dc Y$ in
$\Mod_{\Fvp}$ with an $E_{\infty}$-algebra structure. First we use the coproduct on
$\cC_{\infty}$ to obtain
$$\cC_{\infty}(k)\otimes (X\dc Y)^{\dc k}\to (\cC_{\infty}(k)\otimes
\cC_{\infty}(k))\otimes (X\dc Y)^{\dc k}.$$
By coherence of the axioms of symmetric monoidal categories (see, for instance,
section VII.1 of \cite{CftWM} for coherence between different parenthesizations;
see section 4 of \cite{JoyalStreet} for coherence of the braiding), there is a
unique shuffle map
$$(X\dc Y)^{\dc k}\to X^{\dc k}\dc Y^{\dc k}$$
preserving the ordering of each of the $X$ factors and each of the $Y$ factors.
(Implicit in all tensor powers in a monoidal category is a choice of
parenthesization for $k$-fold products; for instance, we may take $X^{\dc
k}=X\dc (X\dc (\dots \dc X))$.)
Using this shuffle map in $\Mod_{\Fvp}$ along with the lax monoidality
of the functor $\Ch_{\F}\to \Mod_{\Fvp}$, we obtain a map $$ (\cC_{\infty}(k)\otimes
\cC_{\infty}(k))\otimes (X\dc Y)^{\dc k}\to (\cC_{\infty}(k)\otimes X^{\dc
k})\dc (\cC_{\infty}(k)\otimes Y^{\dc k}).$$
Putting these maps together and using the $E_{\infty}$-structures on $X$ and $Y$, we therefore obtain a map
\begin{equation}\label{eq:Einftyprod}\cC_{\infty}(k)\otimes (X\dc Y)^{\dc k}\to
X\dc Y.\end{equation}
It is standard to check that these maps are compatible with the operad
composition and the symmetric group actions.

The commutative monoid $\Fvp$ may be given an $E_{\infty}$-algebra
structure: as usual there is a map of operads from $\cC_{\infty}$ to the
commutative operad (which in $\Ch_{\F}$ consists of $\F[0]$ in each arity
equipped with the trivial symmetric group action), and
the commutative monoid structure on $\Fvp$ endows it with the structure of an algebra
over the commutative operad. With this $E_{\infty}$-algebra structure, $\Fvp$ is
the monoidal unit for the product of $E_{\infty}$-algebras.

\begin{prop}
The $E_{\infty}$-algebra structure on the Day convolution \ref{eq:Einftyprod}
equips $\Alg_{E_{\infty}}(\Mod_{\Fvp})$ with a monoidal structure with $\Fvp$ as
the monoidal unit. Moreover, the
forgetful functor $\Alg_{E_{\infty}}(\Mod_{\Fvp})\to \Mod_{\Fvp}$ is strong
monoidal.
\end{prop}

The proof of this proposition is standard and therefore omitted.

If we choose any 0-chain $\mu\in
\cC_{\infty}(2)$ coming from a point in the arity 2 space of the little disks
operad, then there is an induced map
$$X\dc X\to \F[0]\otimes X\dc X\to \cC_{\infty}(2)\otimes X\dc X\to X,$$
where $\F[0]\to \cC_{\infty}(2)$ picks out $\mu$. We will refer to such a
map as a \emph{chain-level product}, since on the homology level it gives the
usual product. Although such a map is not exactly a map of
$E_{\infty}$-algebras, it is after passing to homology.

\begin{lem}\label{lem:chain prod}
Let $X\in \Alg_{E_{\infty}}(\Mod_{\Fvp})$. The diagram
$$
\xymatrix{
H_{*,*}(E_{\infty}(X\dc X)) \ar[r] \ar[d] & H_{*,*}(X\dc X) \ar[d]\\
H_{*,*}(E_{\infty}(X)) \ar[r] & H_{*,*}(X)
}
$$
induced by any chain-level product $X\dc X\to X$ commutes.
\end{lem}

\begin{proof}
This is a generalization of Lemma 1.9 in \cite{M}, which proves a similar
statement for $E_{\infty}$-algebras in spaces. We define two maps
$$f,g\colon \cC_{\infty}(k)\to \cC_{\infty}(2k)$$
for each $k$. Let $\theta$ denote the operad composition map. The map $f$ is
given by $\theta(-\otimes \mu^k)$. The map $g$ is given by $\theta(\mu\otimes
\psi(-))\s$, where $\psi$ is the coproduct on $\cC_{\infty}(k)$, and $\s$
denotes the shuffle permutation in $\S_{2k}$ that takes $(S\times T)^k\to
S^k\times T^k$ for sets $S,T$.

Let $\S_k$ act on $\cC_{\infty}(2k)$ by permuting pairs; that is,
partition $\{1,2,3,4,\dots,2k\}$ into $k$ sets of 2 as
$\{\{1,2\},\{3,4\},\dots,\{2k-1,2k\}\}$, and then an element of $\S_k$ permutes
these 2-element sets. Then both $f$ and $g$ are $\S_k$-equivariant. Note that
$f$ and $g$ both preserve the canonical augmentation maps of $\cC_{\infty}(k)$
and $\cC_{\infty}(2k)$---for this to be true for $f$ it was essential that $\mu$
is the class of a point in the underlying space. Thus, since $f$ and $g$ are
both maps of $\F[\S_k]$ complexes extending the identity map on $\F$ (in degree
$-1$), since $\cC_{\infty}(k)$ is free over $\F[\S_k]$, and since
$\cC_{\infty}(2k)$ is acyclic when considered as an augmented complex, we have
that $f$ and $g$ agree on homology.

Using the associativity of the $E_{\infty}$-multiplication map on $X$ together
with the definition of the $E_{\infty}$-structure on $X\dc X$, it is easy to
check that the map $$\Sym^k(X\dc X)\to X\dc X\to X$$ is equal to
$$\Sym^k(X\dc X)\xrightarrow{g\otimes \text{id}} \cC_{\infty}(2k)\otimes_{\S_k} (X\dc
X)^{\dc k} \to X,$$
and the map
$\Sym^k(X\dc X)\to \Sym^k(X)\to X$ is equal to
$$\Sym^k(X\dc X) \xrightarrow{f\otimes \text{id}} \cC_{\infty}(2k)\otimes_{\S_k} (X\dc
X)^{\dc k}\to X.$$
Since $f$ and $g$ induce identical maps in homology, we see that the desired
diagram commutes.
\end{proof}

Note also that for $X\in \Alg_{E_{\infty}}(\Mod_{\Fvp})$, the $E_{\infty}$
multiplication map
$$E_{\infty}(X)\to X$$
in particular restricts to a map
$$\cC_{\infty}(0)\otimes \Fvp\cong \Sym^0(X)\to X$$
which means that the obvious map
$$\Fvp\to \Sym^0(X)\to E_{\infty}(X)$$
is a map of $E_{\infty}$-algebras, and therefore each $X\in
\Alg_{E_{\infty}}(\Mod_{\Fvp})$ comes equipped with a unit map
\begin{equation}\label{eq:unit}
\Fvp\to X
\end{equation} of
$E_{\infty}$-algebras.
Thus, for any $X,Y\in \Alg_{E_{\infty}}(\Mod_{\Fvp})$, there are canonical maps
$$X\cong X\dc \Fvp\to X\dc Y$$
and
$$Y\cong \Fvp\dc Y\to X\dc Y$$ that together induce a
natural map
$$X\sqcup_{E_{\infty}} Y\to X\dc Y$$
of $E_{\infty}$-algebras, where $\sqcup_{E_{\infty}}$ denotes the coproduct in
$\Alg_{E_{\infty}}(\Mod_{\Fvp})$.

In particular, when $X$ and $Y$ are free $E_{\infty}$-algebras, say
$X=E_{\infty}(U)$, $Y=E_{\infty}(V)$, we obtain a natural map
$$E_{\infty}(U)\sqcup_{E_{\infty}} E_{\infty}(V)\cong E_{\infty}(U\oplus V)\to E_{\infty}(U)\dc E_{\infty}(V).$$
Here we have used the fact that the left adjoint $E_{\infty}(-)$ preserves
coproducts.
In fact, this map is a weak equivalence.
The following statement appears in the proof of Proposition 16.5 in
\cite{SOS}. Note that there $X$ and $Y$ are required to be cofibrant, but we may
omit this assumption as every module over $\Fvp$ is cofibrant.

\begin{prop}\label{prop:sum to prod} 
For any 
$X,Y\in \Mod_{\Fvp}$, the map
$$E_{\infty}(X\oplus Y)\to E_{\infty}(X)\dc E_{\infty}(Y)$$
induces an isomorphism
$$H_{*,*}(E_{\infty}(X\oplus Y))\cong H_{*,*}(E_{\infty}(X))\dc
H_{*,*}(E_{\infty}(Y))$$
of objects in $\GrVect^{\fC}_{\F}.$
\end{prop}

\section{Homology operations}\label{section:homops}
Now we explain how to construct operations on $\Fvp$-homology. In essence, a
homology operation should be a natural transformation of homology functors.
In previous sections, we considered $(g,n)$ homology as a functor
$$H_{g,n}\colon\Mod_{\Fvp}\to \Vect_{\F},$$
but here we need to make two modifications. First, we wish to consider
operations that are natural in maps of $E_{\infty}$-algebras. Second, we do not
wish to require that our operations are linear. Therefore, in this section, we
consider the compositions
$$\Alg_{E_{\infty}}(\Mod_{\Fvp})\to \Mod_{\Fvp} \xrightarrow{H_{g,n}} \Vect_{\F}\to
\Set,$$
where $\Set$ is the category of sets, and the final map is the forgetful
functor. This map factors through the
homotopy category $\Ho(\Alg_{E_{\infty}}(\Mod_{\Fvp}))$, and hence we will
consider the functors
$$\H_{g,n}\colon \Ho(\Alg_{E_{\infty}}(\Mod_{\Fvp}))\to \Set$$
for each $(g,n)\in \ob{\fG}\times \Z$. By slight abuse of notation, we continue to call
these functors $H_{g,n}$. We will always make clear when we are considering
natural transformations of $\Set$-valued functors.

Note that taking levelwise Cartesian products defines a product $\prod$ on the
category $\Fun(\Ho(\Alg_{E_{\infty}}(\Mod_{\Fvp}),\Set))$.

\begin{defn}
A \emph{$j$-ary $\Fvp$-homology operation} is a natural transformation
$$\prod_{i=1}^j \H_{g_i,n_i}\to \H_{h,m}$$
of functors from $\Ho(\Alg_{E_{\infty}}(\Mod_{\Fvp}))$ to $\Set$, where
$(g_i,n_i),(h,m)\in \ob{\fG}\times \Z$ and $j\in \N$.
\end{defn}

\begin{rem}
Note that this definition and the following discussion do not require anything special about
$E_{\infty}$-operads. In fact, for any operad $\mathscr{O}$ in chain complexes,
we can define $\Fvp$-homology operations of $\mathscr{O}$-algebras to be natural
transformations of homology functors from $\Ho(\Alg_{\mathscr{O}}(\Mod_{\Fvp}))$
to $\Set$. See \cite{Lawson} for a similar discussion (but for spectra rather
than chain complexes). However, for a general operad, we are unlikely to be able
to reasonably classify all operations. The theory for $E_{\infty}$-algebras is both
computable and useful.
\end{rem}

\subsection{The multicategory of homology operations}
In fact, the $\Fvp$-homology operations fit together into a set-enriched symmetric multicategory $\mathsf{Op}_{\Fvp}$. The objects of $\mathsf{Op}_{\Fvp}$ are the elements
of $\ob{\fG}\times \Z$. For objects $(g_i,n_i),(h,m)\in \ob{\fG}\times\Z$, the morphisms from
$(g_1,n_1),\dots,(g_j,n_j)$ to $(h,m)$ are given by
$$\Hom_{\mathsf{Op}_{\Fvp}}((g_1,n_1),\dots,(g_j,n_j);(h,m))=\Nat\left(\prod_{i=1}^j
\H_{g_i,n_i},\H_{h,m}\right),$$
where natural transformations are taken to be in the category of functors from
$\h{\Alg_{E_{\infty}}(\Mod_{\Fvp})}$ to $\mathsf{Set}$. Composition in this
multicategory is given by composition of natural transformations.
The symmetric group $\S_j$ acts on the morphisms
$$\Hom((g_1,n_1),\dots,(g_j,n_j);(h,m))$$
via permutation of inputs. Explicitly, we precompose with the natural
transformation of homology functors that applies the appropriate permutation
levelwise in $\mathsf{Set}$.

Now we may use the Yoneda lemma to obtain a relationship between
\emph{all} $\Fvp$-homology operations and the homology of free
$E_{\infty}$-algebras.

\begin{prop}\label{prop:yoneda}
The set of homology operations
$$\Nat\left(\prod_{i=1}^j
\H_{g_i,n_i},\H_{h,m}\right)$$
of functors from $\Ho(\Alg_{E_{\infty}}(\Mod_{\Fvp}))$ to $\Set$
is in bijective correspondence with
$$\H_{h,m}\left(E_{\infty}(\oplus_{i=1}^j \S^{g_i,n_i}\Fvp)\right).$$
\end{prop}

Here, the notation $\S^{g_i,n_i}\Fvp$ refers to the bigraded suspension of
\ref{section:bigradedsusp}.

\begin{proof}
We have already described free-forgetful Quillen adjunctions
$$\Alg_{E_{\infty}}(\Mod_{\Fvp})\rightleftarrows \Mod_{\Fvp}\rightleftarrows
\Ch_{\F}^{\fC}.$$
Given $g\in\fG$, the functor $\Ch^{\fC}_{\F}\to \mycat$ that evaluates a
functor on $g$ has a left adjoint defined by left Kan extension that sends a
chain complex $D$ to $$\Hom_{\F\fG}(g,-)\otimes_{\F} D.$$ This is also a Quillen adjunction (in fact the model
structure on $\Ch_{\F}^{\fC}$ is defined via transfer along the product of
these maps over all $g\in\fG$.) Moreover, the image
of
$\F[n]=\S^n \F[0]$ under the composition of left adjoints $\mycat\to
\Ch_{\F}^{\fG}\to \Mod_{\Fvp}$ is isomorphic to $\S^{g,n}\Fvp$. 

The homology functor
$$\H_{g,n}\colon \h{\Alg_{E_{\infty}}(\Mod_{\Fvp})}\to \mathsf{Set}$$
factors as a composition
$$\h{\Alg_{E_{\infty}}(\Mod_{\Fvp})}\to \h{\Mod_{\Fvp}}\to
\h{\Ch_{\F}^{\fC}}\to\h{\mycat}\to\mathsf{Set},$$
where the final functor is induced from the $n$th homology of a chain complex. Let $X\in
\Alg_{E_{\infty}}(\Mod_{\Fvp})$. Using the notation $[-,-]_{\mathsf{D}}$ to
denote morphisms in the homotopy category of a category $\mathsf{D}$, we then
have the following chain of isomorphisms (where by abuse of notation, we omit
forgetful functors):
\begin{align*}
[\F[n],X(g)]_{\mycat} 
&\cong
[\S^{g,n}\Fvp,X]_{\Mod_{\Fvp}}\\
&\cong [E_{\infty}(\S^{g,n}\Fvp),X]_{\Alg_{E_{\infty}}(\Mod_{\Fvp})}.
\end{align*}
Since $$\H_{g,n}(X)\cong H_n(X(g))\cong [\F[n],X(g)]_{\mycat},$$ we see that
$\H_{g,n}$ is representable with representing object
$E_{\infty}(\S^{g,n}\Fvp)$. Similarly, we can factor $\prod_{i=1}^j
\H_{g_i,n_i}$ as
$$\h{\Alg_{E_{\infty}}(\Mod_{\Fvp})}\to \h{\Mod_{\Fvp}}\to
\h{\Ch_{\F}^{\fC}}\to\prod_{i=1}^j\h{\mycat}\to\mathsf{Set}.$$
Here we have that
$$\prod_{i=1}^j \H_{g_i,n_j}(X)\cong \prod_{i=1}^n H_{n_i}(X(g_i))\cong
\prod_{i=1}^j[\F[n_i],X(g_i)]_{\mycat},$$
so using the same argument as before, we see that $\prod_{i=1}^j
\H_{g_i,n_i}$ is representable with representing object
$E_{\infty}(\oplus_{i=1}^j \S^{g_i,n_i}\Fvp).$ The statement then follows from the
Yoneda lemma. 

\end{proof}

\begin{rem}\label{rem:anyoperad}
A weak equivalence $\cC_{\infty}\to \cD_{\infty}$ of $E_{\infty}$-operads
induces a (semi) Quillen equivalence 
$$\Alg_{\cD_{\infty}}(\Mod_{\Fvp})\to \Alg_{\cC_{\infty}}(\Mod_{\Fvp}),$$
which in turn induces a multifunctor
$\mathsf{Op}_{\Fvp}^{\cC_{\infty}}\to \mathsf{Op}^{\cD_{\infty}}_{\Fvp}.$ In
light of Proposition \ref{prop:yoneda} and Lemma \ref{lem:homindep}, this is an
equivalence of multicategories. This is the sense in which the theory of
homology operations does not depend on the choice of $E_{\infty}$-operad.
\end{rem}

Although this is standard, given the ubiquity of this bijective correspondence
in this work, it is worth
giving an explicit description of how such a homology class determines a
homology operation. Suppose $X\in \Alg_{E_{\infty}}(\Mod_{\Fvp})$. Classes
$x_i\in \H_{g_i,n_i}(X)$ may be represented by maps
$\S^{g_i,n_i}\Fvp\to X$
in $\Mod_{\Fvp}$, which together determine a map
$$\bigoplus \S^{g_i,n_i}\Fvp\to X.$$ Applying the functor $E_{\infty}(-)$, such a map gives rise to
a map
$$E_{\infty}(\oplus \S^{g_i,n_i}\Fvp)\to E_{\infty}(X)\to X.$$
Thus, the image of a class $\g\in \H_{b,m}(E_{\infty}(\oplus \S^{g_i,n_i}\Fvp))$
under this map determines a class in $\H_{b,m}(X)$; the homology
operation $\g$ applied to the tuple $(x_i)\in \prod \H_{g_i,n_i}(X)$ is given by
this class.

\subsection{Charge and weight of operations}
Using the explicit construction of free $E_{\infty}$-algebras, we can develop
terminology to describe more explicitly how an operation arises. First we define
a grading on all free $E_{\infty}$-algebras, which we will use extensively in
giving a description of the homology of free $E_{\infty}$-algebras. The
terminology is inspired by \cite{jeremysander}, although the definition is
slightly different.

\begin{defn}
The isomorphism $$E_{\infty}(X)\cong \bigoplus_{k\geq 0}\Sym^k(X)$$ gives
$E_{\infty}(X)$ an $\N$ grading, which we refer to as the \emph{charge} grading.
\end{defn}

 Thus, Proposition
\ref{prop:yoneda} gives us a charge grading on homology operations as well.

\begin{defn}
A homology operation determined by a class in $$\H_{h,m}(\Sym^k(\oplus_{i=1}^j
\S^{g_i,n_i}\Fvp))$$ has \emph{charge $k$}.
\end{defn}

It is convenient (perhaps more so for operations for $E_n$-algebras for finite
$n$) to have a further refinement on the notion of the charge of an operation.
Notice that there is a further splitting
\begin{align*}
\Sym^k(\oplus_{i=1}^j \S^{g_i,n_i}\Fvp) &= \cC_{\infty}(k)\otimes_{\S_k}
\left(\oplus_{i=1}^j \S^{g_i,n_i}\Fvp\right)^{\dc k}\\
&\cong \bigoplus_{\substack{(r_1,\dots,r_{j})\\ \sum r_i=k}}
\cC_{\infty}(k)\otimes_{\S_{r_1}\times\dots\times\S_{r_j}}\bigdc_{i=1}^j
(\S^{g_i,n_i}\Fvp)^{\dc r_i}.
\end{align*}

\begin{defn}
For any tuple $(r_1,\dots,r_j)$ of natural numbers, a \emph{weight
$(r_1,\dots,r_j)$} homology operation is one determined by a class in
$$\H_{h,m}\left(
\cC_{\infty}(k)\otimes_{\S_{r_1}\times\dots\times\S_{r_j}}\bigdc_{i=1}^j
(\S^{g_i,n_i}\Fvp)^{\dc r_i}\right),$$
where $k=\sum r_i$.
\end{defn}

Intuitively, the weight tells us to what
``power'' each input is taken in the operation. Note that for unary operations, the notions of
weight and charge coincide, and in general the charge can be read off from the
weight by summing the indices $r_i$.
The following lemma tells us how weight (and therefore also charge) behaves under composition of
operations.

\begin{lem}\label{lemma:comp}
Let $\g\in \Nat(\prod_{i=1}^j \H_{g_i,n_i},\H_{h,m})$ have
weight $(r_1,\dots,r_j)$, and let $$\eta_i\in
\Nat(\prod_{\ell=1}^{k_i}\H_{c^{(i)}_{\ell},q^{(i)}_{\ell}},\H_{g_i,n_i})$$
have weight $(s^{(i)}_1,\dots,s^{(i)}_{k_i})$. Then the composition
$$\g\circ (\eta_1,\dots,\eta_j)\in
\Nat(\prod_{i,{\ell}}\H_{c^{(i)}_{\ell},q^{(i)}_{\ell}}, \H_{h,m})$$
has weight
$(r_1s^{(1)}_1,\dots,r_1s^{(1)}_{k_1},\dots,r_js^{(j)}_1,\dots,r_js^{(j)}_{k_j}).$
\end{lem}

\begin{proof}
Since $\g$ has weight $(r_1,\dots,r_j)$, it is an element of
$$\H_{h,m}(\cC_{\infty}(\sum r_i)\otimes_{\S_{r_1}\times\dots\times
\S_{r_j}}\bigdc_{i=1}^j (\S^{g_i,n_i}\Fvp)^{\dc r_i}).$$
We can represent each $\eta_i$ as a homotopy class of maps
$$\S^{g_i,n_i}\Fvp\to \cC_{\infty}(s^{(i)})\otimes_{\S_{s_1^{(i)}}\times
\dots\times\S_{s^{(i)}_{k_i}}}\bigdc_{\ell}(\S^{c_{\ell}^{(i)},q_{\ell}^{(i)}}\Fvp)^{\dc
s^{(i)}_{\ell}},$$
where $s^{(i)}=\sum s^{(i)}_{\ell}.$
Tracing through the Yoneda correspondence, we see that  
composition of these operations is obtained by taking the image of $\g$ under
the map
$$E_{\infty}(\oplus_i \S^{g_i,n_i}\Fvp)\to
E_{\infty}(E_{\infty}(\oplus_{i,\ell}\S^{c^{(i)}_{\ell},q^{(i)}_{\ell}}\Fvp))\to
E_{\infty}(\oplus_{i,\ell}\S^{c^{(i)}_{\ell},q^{(i)}_{\ell}}\Fvp).$$
Restricting to the appropriate summands according to our assumptions on the
weight of $\g$ and $\eta_i$, we may replace the first map by
$$
\begin{tikzcd}
\cC_{\infty}(\sum r_i)\otimes_{\S_{r_1}\times\dots\times\S_{r_j}}
\bigdc_{i=1}^j (\S^{g_i,n_i}\Fvp)^{\dc r_i}
\arrow{d}\\
\cC_{\infty}(\sum r_i)\otimes_{\S_{r_1}\times\dots\times\S_{r_j}} \bigdc_i
\left(\cC_{\infty}(s^{(i)})\otimes_{\S_{s_1^{(i)}}\times
\dots\times\S_{s^{(i)}_{k_i}}}\bigdc_{\ell}(\S^{c^{(i)}_{\ell},q^{(i)}_{\ell}}\Fvp)^{\dc
s^{(i)}_{\ell}}\right)^{\dc r_i}.
\end{tikzcd}
$$
However, under the $E_{\infty}$ composition, the second map now has
target
$$
\cC_{\infty}(\sum_{i,\ell} r_is^{(i)}_{\ell})\otimes_{\S_{s^{(1)}_1}\times
\dots\times \S_{s^{(j)}_{k_j}}} \bigdc_{i,\ell}
(\S^{c^{(i)}_{\ell},q^{(i)}_{\ell}}\Fvp)^{\dc s^{(i)}_{\ell}r_i}.
$$
Thus, the composition $\g\circ (\eta_1,\dots,\eta_j)$ lands in the weight
$(r_1s^{(1)}_{\ell},\dots,r_js^{(j)}_{k_j})$ summand.
\end{proof}

In particular, charge is multiplicative under composition of unary operations.
Moreover, if an input has weight 1, then the operation is linear in that input.

\begin{lem}\label{lem:weight1}
Let $\g\in \Nat(\prod_{i=1}^j\H_{g_i,n_i},\H_{h,m})$ be an operation of
weight $(r_1,\dots,r_j)$. If $r_i=1$ for some $i$, then the operation $\g$ is
linear in the $i$th input. That is, if $X\in\Alg_{E_{\infty}}(\Mod_{\Fvp})$,
$x_k\in
H_{g_k,n_k}(X)$ for $1\leq k\leq j$, $x_i'\in H_{g_i,n_i}(X)$, and
$\a\in\F$, then
$$\g(x_1,\dots,\a
x_i+x_i',\dots,x_k)=\a\g(x_1,\dots,x_i,\dots,x_k)+\g(x_1,\dots,x_i',\dots,x_k).$$
\end{lem}

\begin{proof}
Let $x_k\colon \S^{g_k,n_k}\Fvp\to X$ be maps representing the classes $x_k\in
H_{g_k,n_k}(X)$.
Recall that we can evaluate $\g$ on
$(x_1,\dots,x_j)$ by taking its image under the composition
$$E_{\infty}(\oplus \S^{g_k,n_k}\Fvp)\xrightarrow{E_{\infty}(\sum x_k)}
E_{\infty}(X)\to X.$$
Since $\g$ has weight $(r_1,\dots,r_j)$, it is an element of
$$\H_{h,m}(\cC_{\infty}(r)\otimes_{\S_{r_1}\times\dots\times\S_{r_j}}
\bigdc_k (\S^{g_k,n_k}\Fvp)^{\dc r_k}),$$
where $r=\sum r_k$.
Restricting the above composition to this summand, we obtain the composition
$$\cC_{\infty}(r)\otimes_{\S_{r_1}\times\dots\times\S_{r_j}}
\bigdc_k (\S^{g_k,n_k}\Fvp)^{\dc r_k} \xrightarrow{id\otimes (\dc x_k^{\dc r_k})}
\cC_{\infty}(r)\otimes_{\S_{r_1}\times\dots\times\S_{r_k}} \bigdc X^{\dc r_k} \to
X.$$
But Day Convolution of maps is linear; that is, $$y\dc (\a x_i+x_i')\dc y'= \a
y\dc x_i \dc y' + y\dc x_i'\dc y'$$ for any maps $y,y'$. Hence, if $r_i=1$, the
image of $\g$ under the map determined by $(x_1,\dots,\a x_i+x_i',\dots,x_j)$
coincides with that under the map determined by
$$\a(x_1,\dots,x_i,\dots,x_j)+(x_1,\dots,x_i',\dots,x_j).$$
\end{proof}

Proposition \ref{prop:yoneda} tells us that in order to classify \emph{all} $\Fvp$-homology operations, it suffices to
understand the homology of free $E_{\infty}$-algebras together with its weight
grading. In Theorem
\ref{thm:W}, we will show that the homology of free $E_{\infty}$-algebras can
be expressed entirely in terms of certain unary, charge $p$ operations---known as
the \emph{Dyer--Lashof} operations---and a product, a binary weight $(1,1)$ operation.
Thus, these operations generate all $\Fvp$-homology operations.

\subsection{The product}\label{section:product}
We provide a brief note on the product, a collection of binary weight $(1,1)$
operations in $\Hom_{\mathsf{Op}_{\Fvp}}((g_1,n_1),(g_2,n_2);(g_1\oplus g_2,n_1+n_2)).$ By
Lemma \ref{lem:weight1}, such operations are linear in both inputs, and so can
be thought of as operations
$$H_{g_1,n_1}\otimes H_{g_2,n_2}\to H_{g_1\oplus g_2,n_1+n_2}$$
that are natural transformations of functors from
$\Ho(\Alg_{E_{\infty}}(\Mod_{\Fvp}))$ to $\Vect_{\F}$.
Moreover,
such operations are in bijection with classes in
$$\H_{g_1\oplus g_2,n_1+n_2}(\cC_{\infty}(2)\otimes_{\S_1\times \S_1}
\S^{g_1,n_1}\Fvp\dc
\S^{g_2,n_2}\Fvp).$$
Since $\S^{g_1,n_1}\Fvp\dc \S^{g_2,n_2}\Fvp\cong \S^{g_1\oplus g_2,n_1+n_2}\Fvp$, we have
\begin{align*}
\H_{g,n}(\cC_{\infty}(2)\otimes \S^{g_1,n_1}\Fvp\dc
\S^{g_2,n_2}\Fvp) &\cong \H_{g,n}(\cC_{\infty}(2)\otimes
\S^{g_1\oplus g_2,n_1+n_2}\Fvp)\\
&\cong H_{n_1+n_2}(\cC_{\infty}(2)\otimes \F[n_1+n_2])\\
&\cong H_0(\cC_{\infty}(2))\\
&\cong \F,
\end{align*}
where $g=g_1\oplus g_2$ and $n=n_1+n_2$.
Defining a product thus amounts to choosing a generator of this one-dimensional
vector space. There is a canonical choice of generator of $H_0(\cC_{\infty}(2))$, namely the homology class
 determined by any point in the second arity space of the little disks operad. 
We use this distinguished class to define the product.

\begin{prop}\label{prop:product}
 The product is associative and commutative. That is, for homology classes $x_i\in
H_{g_i,n_i}(X)$, the canonical map
$$H_{(g_1\oplus g_2)\oplus g_3,n_1+n_2+n_3}(X)\to H_{g_1\oplus (g_2\oplus g_3),n_1+n_2+n_3}(X)$$
induced by the associator $(g_1\oplus g_2)\oplus g_3\to g_1\oplus (g_2\oplus
g_3)$ sends
$$(x_1x_2)x_3\mapsto x_1(x_2x_3),$$
and the map
$$H_{g_1\oplus g_2,n_1+n_2}(X)\to H_{g_2\oplus g_1,n_1+n_2}(X)$$
induced by the braiding $g_1\oplus g_2\to g_2\oplus g_1$ sends
$$x_1x_2\mapsto (-1)^{n_1n_2}x_2x_1.$$
\end{prop}

\begin{rem}\label{rem:product}
In general, it does not make sense to ask whether $(x_1x_2)x_3$ \emph{equals}
$x_1(x_2x_3)$ or whether $x_1x_2$ \emph{equals} a multiple of $x_2x_1$, since
these classes live in different homology groups.

However, if $g_1=g_2=g$, we can use the description of the braiding given in
Section \ref{section:preliminaries} to obtain
$$x_1x_2=(-1)^{n_1n_2}\vp(\tr_{g\oplus g}(\b_{g,g}^{-1}))x_2x_1\in H_{g\oplus
g,n_1+n_2}(X),$$
where $\b_{g,g}$ is the braiding $g\oplus g\to g\oplus g$.
In particular, if $x_1=x_2=x$, this says that $x^2=0$ when
$(-1)^{n_1}\vp(\tr_{g\oplus g}(\b_{g,g}^{-1}))=-1$.
\end{rem}

\begin{proof}
Note that $x_1(x_2x_3)$ may be described as the product of the identity
operation applied to $x_1$ with the product operation applied to $x_2\otimes x_3$.
The identity operation on $H_{g_1,n_1}$ is represented by a map
$$\S^{g_1,n_1}\Fvp\to \cC_{\infty}(1)\otimes \S^{g_1,n_1}\Fvp,$$
and the product on $H_{g_2,n_2}\otimes H_{g_3,n_3}$ is represented by a map
$$\S^{g_2\oplus g_3,n_2+n_3}\Fvp\to \cC_{\infty}(2)\otimes \S^{g_2,n_2}\Fvp\dc
\S^{g_3,n_3}\Fvp.$$
Thus, the product on $H_{g_1,n_1}\otimes H_{g_2\oplus g_3,n_2+n_3}$ is given by the image
of the product class in 
$$H_{g_{1(23)},n}(
\cC_{\infty}(2)\otimes \S^{g_1,n_1}\Fvp\dc \S^{g_2\oplus g_3,n_2+n_3}\Fvp),$$
where $g_{1(23)}=g_1\oplus (g_2\oplus g_3)$ and $n=n_1+n_2+n_3$, under the composition
$$
\begin{tikzcd}
\cC_{\infty}(2)\otimes \S^{g_1,n_1}\Fvp\dc \S^{g_2\oplus g_3,n_2+n_3}\Fvp\arrow{d}\\
\cC_{\infty}(2)\otimes (\cC_{\infty}(1)\otimes \S^{g_1,n_1}\Fvp)\dc
(\cC_{\infty}(2)\otimes \S^{g_2,n_2}\Fvp\dc \S^{g_3,n_3}\Fvp)
\arrow{d}\\
\cC_{\infty}(3)\otimes \S^{g_1,n_1}\Fvp\dc (\S^{g_2,n_2}\Fvp\dc \S^{g_3,n_3}\Fvp).
\end{tikzcd}
$$
Explicitly, if $c\in \cC_{\infty}(2)$ is any chain determined by a point in
arity 2 of the little discs operad, and if we write $\iota_i$ for $1\in
\S^{g_i,n_i}\Fvp(g_i)_{n_i}$, then the product operation on $H_{g_1,n_1}\otimes
(H_{g_2,n_2}\otimes H_{g_3,n_3})$ is represented by the chain $$\th(c;e,c)\otimes
\iota_1\otimes (\iota_2\otimes \iota_3).$$ Here, $e$ is the class of any point in
the arity 1 space of the little discs operad, and $\th$ denotes the operad
composition.

Similarly, the product on $(H_{g_1,n_1}\otimes H_{g_2,n_2})\otimes H_{g_3,n_3}$ is represented by the chain $$\th(c;c,e)\otimes
(\iota_1\otimes \iota_2)\otimes \iota_3$$
in
$$\cC_{\infty}(3)\otimes (\S^{g_1,n_1}\Fvp\dc \S^{g_2,n_2}\Fvp)\dc \S^{g_3,n_3}\Fvp.$$

Since $c$ and $e$ both come from points in a space, the classes $\th(c;e,c)$ and
$\th(c;c,e)$ both come from points in the arity 3 space of the operad, and thus
are homologous since this space is connected.

The canonical map comparing the two parenthesizations is obtained from the
composition
$$
\begin{tikzcd}
H_{g_{(12)3},n}(\cC_{\infty}(3)\otimes (\S^{g_1,n_1}\Fvp\dc \S^{g_2,n_2}\Fvp)\dc
\S^{g_3,n_3}\Fvp) \ar[d]\\
H_{g_{(12)3},n}(\cC_{\infty}(3)\otimes \S^{g_1,n_1}\Fvp\dc (\S^{g_2,n_2}\Fvp\dc
\S^{g_3,n_3}\Fvp))\ar[d]\\
H_{g_{1(23)},n}(\cC_{\infty}(3)\otimes \S^{g_1,n_1}\Fvp\dc (\S^{g_2,n_2}\Fvp\dc
\S^{g_3,n_3}\Fvp)),
\end{tikzcd}
$$
where the first map comes from applying the associator in $\Mod_{\Fvp}$, and the
second map is induced by the associator $\a_{g_1,g_2,g_3}$ in $\fC$.
Recall from the description given in Section \ref{section:preliminaries} that
the
associator in $\Mod_{\Fvp}$ is defined using a composition whose final map is
induced by $\a^{-1}_{g_1,g_2,g_3}$, so these associators cancel. Since the
associators in $\Ch_{\F}$ are trivial,
we see that
this composition sends $$\th(c;c,e)\otimes ((\iota_1\otimes \iota_2)\otimes
\iota_3)\mapsto \th(c;c,e)\otimes (\iota_1\otimes (\iota_2\otimes \iota_3))$$
on the chain level. Passing to homology, and using that $\th(c;e,c)$ and
$\th(c;c,e)$ are homologous proves the result.

Comparing the product on $H_{g_1,n_1}\otimes H_{g_2,n_2}$ 
with the product on $H_{g_2,n_2}\otimes H_{g_1,n_1}$ amounts to considering the
image of the class $c\otimes \iota_1\otimes \iota_2$ under the
composition
$$
\begin{tikzcd}
H_{g_1\oplus g_2,n_1+n_2}(\cC_{\infty}(2)\otimes \S^{g_1,n_1}\Fvp\dc
\S^{g_2,n_2}\Fvp)
\ar[d]\\
H_{g_1\oplus g_2,n_1+n_2}(\cC_{\infty}(2)\otimes \S^{g_2,n_2}\Fvp\dc
\S^{g_1,n_1}\Fvp)
\ar[d]\\
H_{g_2\oplus g_1,n_1+n_2}(\cC_{\infty}(2)\otimes \S^{g_2,n_2}\Fvp\dc
\S^{g_1,n_1}\Fvp),
\end{tikzcd}
$$
where the first map is obtained by applying the braiding
$$\S^{g_1,n_1}\Fvp\dc \S^{g_2,n_2}\Fvp\to \S^{g_2,n_2}\Fvp\dc \S^{g_1,n_1}\Fvp$$
in $\Mod_{\Fvp}$,
and the second is induced from the braiding $\b_{g_1,g_2}$ in $\fC$. From the
discussion in Section \ref{section:preliminaries}, we know that the former
braiding morphism multiplies by $(-1)^{n_1n_2}$ and acts by $\b^{-1}_{g_1,g_2}$.
Therefore before taking homology, the composition sends
$$c\otimes \iota_1\otimes \iota_2\mapsto (-1)^{n_1n_2} c\otimes \iota_2\otimes
\iota_1.$$
Passing to homology, we see that the product on $H_{g_1,n_1}\otimes H_{g_2,n_2}$
is mapped to $(-1)^{n_1n_2}$ times the product on $H_{g_2,n_2}\otimes
H_{g_1,n_1}$.
\end{proof}

Recall from \ref{eq:unit} that any $X\in \Alg_{E_{\infty}}(\Mod_{\Fvp})$ comes
equipped with a unit map $\Fvp\to X$ of $E_{\infty}$-algebras. The canonical generator
$1\in H_{\1_{\fG},0}(\Fvp)$ then determines a class $\mathbbm{1}\in H_{\1_{\fG},0}(X)$,
which we claim acts as a unit for the product.

\begin{prop}\label{prop:unit}
If $X\in \Alg_{E_{\infty}}(\Mod_{\Fvp})$, then $\mathbbm{1}\in H_{\1_{\fG},0}(X)$
is the unit for the product on $H_{*,*}(X)$.
\end{prop} 

\begin{proof}
Let $x\colon \S^{g,n}\Fvp\to X$ represent a class in $H_{g,n}(X)$. We can replace any
class with the identity operation applied to it, which amounts to considering
instead the composition
$$\S^{g,n}\Fvp\to \cC_{\infty}(1)\otimes \S^{g,n}\Fvp\to \Sym^1(X)\to X.$$
The class $\mathbbm{1}$ may be represented by a map
$$\S^{\1_{\fG},0}\Fvp\to \cC_{\infty}(0)\otimes \S^{\1_{\fG},0}\Fvp\to \Sym^0(X)\to X.$$
Thus, the product of $x$ with $\mathbbm{1}$ is given by composing the product
operation with the identity and the unit, which is given by the composition
$$
\begin{tikzcd}
\cC_{\infty}(2)\otimes \S^{g,n}\Fvp\dc \S^{\1_{\fG},0}\Fvp \arrow{d}\\
\cC_{\infty}(2)\otimes (\cC_{\infty}(1)\otimes \S^{g,n}\Fvp)\otimes
(\cC_{\infty}(0)\otimes \S^{\1_{\fG},0}\Fvp) \arrow{d}\\
\cC_{\infty}(1)\otimes \S^{g,n}\Fvp.
\end{tikzcd}
$$
The operation we obtain is given by the operad composition
$$\cC_{\infty}(2)\otimes \cC_{\infty}(1)\otimes \cC_{\infty}(0)\to
\cC_{\infty}(1).$$
Since the product class in $\cC_{\infty}(2)$, the identity class in
$\cC_{\infty}(1)$, and the unit class in $\cC_{\infty}(0)$ all come from the
class of a point in a space, their composition also comes from
the class of a point in $\cC_{\infty}(1)$. Any such point is homologous to
``the'' identity element of $\cC_{\infty}(1)$, namely the chain determined by
the identity embedding of a single disc into itself, and thus any point
determines the identity map on homology.

Similarly, multiplying in the other order also gives the identity.
\end{proof}

\section{Dyer--Lashof operations}\label{section:DLops}

We construct certain operations that generalize the classical Dyer--Lashof
operations. In the following section, we will show that these operations,
together with the product, generate all $\Fvp$-homology operations.

The characteristic $p$ Dyer--Lashof operations on $H_{g,n}$ come from certain classes in the
homology $H_*(\S_p;\F)$ of the $p$th symmetric group with twisted coefficients
in $\F$. Here, $\S_p$ acts on $\F$ as multiplication by a unit using the map
$$\S_p\to \Z/2\Z\to \F^{\times},$$
where the nonidentity element in $\Z/2\Z$ maps to $(-1)^n\vp(\tr_{g\oplus
g}(\b_{g,g}^{-1}))$.
For homology with constant coefficients, the value of $\vp(\tr_{g\oplus
g}(\b_{g,g}^{-1}))$ is assumed to be 1, but in our case it
is also possible to have $\vp(\tr_{g\oplus g}(\b_{g,g}^{-1}))=-1$. Thus, if $p=2$, there is no
difference between the classical definition of the operations and the twisted
operations we construct. Therefore, for the remainder of this work, we suppose
$p>2$.

Let $C_p$ denote the cyclic group of order $p$ with generator $\a$, and let $W$ be the resolution of
$\F$ by free $\F[C_p]$-modules that has one generator, $e_i$, in each degree,
and such that
$$d(e_{2i})=(1+\a+\dots +\a^{p-1})e_{2i-1} \text{  and  }
d(e_{2i-1})=(\a-1)e_{2i-2} \text{ for } i>0.$$
Since $\cC_{\infty}(p)$ is, when given the usual augmentation map to $\F$, an
acyclic $\F[\S_p]$-complex (and therefore also an acyclic $\F[C_p]$-complex), we
obtain a map $W\to \cC_{\infty}(p)$ of $\F[C_p]$-complexes that is unique up to
homotopy. In particular, given any $X\in \Mod_{\Fvp}$, we obtain a unique map 
$$H_{*,*}(W\otimes_{C_p} X^{\dc p})\to H_{*,*}(\Sym^p(X)).$$

We also define a coproduct $\psi$ on $W$ (see Definition 1.2 of \cite{Mayalg}):
\begin{equation}\label{eq:Wcoprod}
\psi(e_{2i+1})=\sum_{j+k=i}e_{2j}\otimes e_{2k+1}+e_{2j+1}\otimes \a e_{2k}
\end{equation}
and
$$\psi(e_{2i})=\sum_{j+k=i}e_{2j}\otimes e_{2k}+\sum_{j+k=i-1} \sum_{0\leq
r<s<p}\a^r e_{2j+1}\otimes \a^s e_{2k+1}.$$
This coproduct is a chain map and resolves the canonical isomorphism
$$\F\to \F\otimes_{\F}\F.$$ 

Since $\cC_{\infty}$ arises from a space operad and so we have a chain-level
coproduct $\cC_{\infty}(p)\to \cC_{\infty}(p)\otimes \cC_{\infty}(p)$, we claim that the diagram
$$
\xymatrix{
W\ar[d]^{\psi} \ar[r] & \cC_{\infty}(p)\ar[d]\\
W\otimes W \ar[r] & \cC_{\infty}(p)\otimes \cC_{\infty}(p)
}
$$
commutes (at least up to homotopy).
Note that the diagram commutes in degree 0, since the coproduct of a zero chain
$c$ determined by a point in the arity $p$ space is simply $c\otimes c$, and
$\psi(e_0)=e_0\otimes e_0$. Since $\cC_{\infty}(p)\otimes \cC_{\infty}(p)$ is
acyclic and $W\otimes W$ is a free $\F[C_p]$-complex, maps $W\otimes W\to
\cC_{\infty}(p)\otimes \cC_{\infty}(p)$ lifting this map are unique up to
homotopy. 

\begin{lem}\label{lem:zero}
Let $X=\S^{g,n}\Fvp$ for $g\in \fG$ and $n\in \Z$. Then
$$H_{h,m}(W\otimes_{C_p} X^{\circledast p})\cong
\begin{cases}
\F & \text{if }h\cong \gpower{g}{p} \text{ and } m\geq pn\\
0 & \text{otherwise.}
\end{cases}
$$
The image of $e_i$ under the map
$$H_{\gpower{g}{p},i+pn}(W\otimes_{C_p} X^{\circledast p}) \to H_{\gpower{g}{p},i+pn}(\Sym^p(X))$$
is zero, unless
$$
\begin{cases}
i=2j(p-1) \text{ or } 2j(p-1)-1 & \text{if }(-1)^n\vp(\tr_{g\oplus g}(\b_{g,g}^{-1}))=1\\
i=(2j+1)(p-1) \text{ or } (2j+1)(p-1)-1 & \text{if }(-1)^n\vp(\tr_{g\oplus g}(\b_{g,g}^{-1}))=-1.
\end{cases}
$$
\end{lem}

\begin{proof}
First note that
$$(\S^{g,n}\Fvp)^{\dc p}\cong \S^{\gpower{g}{p},pn}\Fvp,$$ 
so the $(h,m)$ homologies of both $W\otimes_{C_p} X^{\dc p}$ and
$\Sym^p(X)$ are zero unles $h\cong \gpower{g}{p}$. Moreover, any morphism $h\to
\gpower{g}{p}$ induces an isomorphism $H_{h,m}\to
H_{\gpower{g}{p},m}$, so it suffices to consider $h=\gpower{g}{p}$.

The symmetric group action on $(\S^{g,n}\Fvp)^{\dc p}$ comes from
permuting the factors $\S^{g,n}\Fvp$. Recall that the symmetry
$$\S^{g,n}\Fvp\dc \S^{g,n}\Fvp\to \S^{g,n}\Fvp\dc \S^{g,n}\Fvp$$
is given by multiplication by the unit $q=(-1)^n\vp(\tr_{g\oplus g}(\b_{g,g}^{-1}))$. Thus, 
$$H_{\gpower{g}{p},m}(W\otimes_{C_p} X^{\dc p})\cong H_{m}(W_{*-pn}\otimes_{C_p}
\F(q))$$
and
$$H_{\gpower{g}{p},m}(\Sym^p(X))\cong
H_m(\cC_{\infty}(p)_{*-pn}\otimes_{\S_p}\F(q)),$$
where by $\F(q)$, we mean the $\S_p$-representation $\F$, where the generator
$\s_j$ that permutes $j$ and $j+1$ acts by multiplication by $q$. But since $W$
is a free $C_p$-resolution of $\F(q)$, and $\cC_{\infty}(p)$ is a free
$\S_p$-resolution of $\F(q)$, we have that
$$H_{\gpower{g}{p},m}(W\otimes_{C_p} X^{\dc p})\cong H_{m-pn}(C_p;\F(q))$$ and
$$H_{\gpower{g}{p},m}(\Sym^p(X))\cong H_{m-pn}(\S_p;\F(q)).$$
A generator of $C_p$ can be written as a product of $p-1$ transpositions, and
hence it acts trivially on $\F$, since $q^{p-1}=1$. The homology of
$W\otimes_{C_p} \F(q)$ is then isomorphic to $H_*(C_p;\F)$, which is
one-dimensional in each degree. This homology is generated by $e_i\otimes 1\in
W\otimes \F$.

The remainder of the statement follows directly from Lemma 1.4 in \cite{Mayalg}. 
\end{proof}

\subsection{Lower indexing}\label{subsection:lower}

Recall from Proposition \ref{prop:yoneda} that a class in
$$H_{\gpower{g}{p},i+pn}(\Sym^p(\S^{g,n}\Fvp))$$ determines a homology operation from
$$H_{g,n}\to H_{\gpower{g}{n},i+pn}.$$
We will use such classes coming from the homology of $C_p$ to define particular
operations.

\begin{defn}\label{defn:lowerindex}
Let $Q_{i(p-1)}$ denote the homology operation determined by the class $e_{i(p-1)}\in
H_{i(p-1)}(C_p;\F)$. That
is, for any $(g,n)$,
$$Q_{i(p-1)}\colon H_{g,n}\to H_{\gpower{g}{p},i(p-1)+pn}$$
is the operation determined by the image of $$e_{i(p-1)}\in H_{i(p-1)}(C_p;\F)\cong
H_{\gpower{g}{p},i(p-1)+pn}(W\otimes_{C_p} (\S^{g,n}\Fvp)^{\dc p})$$ in
$H_{\gpower{g}{p},i(p-1)+pn}(\Sym^p(\S^{g,n}\Fvp)).$
Similarly, we let $\b Q_{i(p-1)}$ denote the homology operation determined by
$e_{i(p-1)-1}\in H_{i(p-1)-1}(C_p;\F)$.
\end{defn}

These operations are the \emph{Dyer--Lashof operations} written with lower
indexing.
By Lemma \ref{lem:zero}, these operations account for all nontrivial operations
coming from $H_*(C_p;\F)$ (although some of the operations we have defined are
trivial).

\begin{rem}
Note that $\b Q_{i(p-1)}$ is \emph{not} defined to
be the Bockstein homomorphism applied to $Q_{i(p-1)}$. Since our chain complexes
need not come from integral chain complexes, we may not have a Bockstein
homomorphism.
\end{rem}

The Dyer--Lashof operations satisfy versions of all of the classical relations,
listed below (compare
with Theorem  1.1 of \cite{IteratedLoops}). Much of the proof follows the methods of
\cite{Mayalg}, though we take care to rephrase everything in the language of
$\Mod_{\Fvp}$ and sometimes give modified arguments.

\begin{thm}\label{thm:lower index}
The Dyer--Lashof operations on $\Alg_{E_{\infty}}(\Mod_{\Fvp})$ consist of
natural transformations
\begin{align*}
Q_{i(p-1)}&\colon H_{g,n}\to H_{\gpower{g}{p},i(p-1)+pn}  &\text{ for } i\geq 0\\
\b Q_{i(p-1)}&\colon H_{g,n}\to H_{\gpower{g}{p},i(p-1)+pn-1} &\text{ for }i>0
\end{align*}
satisfying the following relations.
\begin{enumerate}
\item (``Frobenius'' linearity) For $x,y\in H_{g,n}(X)$ and $t\in \F$,
$$Q_{i(p-1)}(t x+y)=t^p Q_{i(p-1)}(x)+Q_{i(p-1)}(y)$$ and $$\b Q_{i(p-1)}(t
x+y)= t^p\b
Q_{i(p-1)}(x)+\b Q_{i(p-1)}(y).$$ \label{additivity}
\item If
$(-1)^{n+i}\vp(\tr_{g\oplus g}(\b_{g,g}^{-1}))=-1$, then $Q_{i(p-1)}=0$ and $\b Q_{i(p-1)}=0$ on
$H_{g,n}$.  \label{zero}
\item For any $x\in H_{*,*}(X)$, $Q_0(x)=x^p$. \label{power}
\item If $\mathbbm{1}\in H_{\1_{\fG},0}(X)$ is the identity, then
$Q_{i(p-1)}(\mathbbm{1})=0$ and $\b Q_{i(p-1)}(\mathbbm{1})=0$ for all $i>0$.
\label{unit}
\item (External Cartan formulas) If $X,Y\in \Alg_{E_{\infty}}(\Mod_{\Fvp})$ with
homology classes
$x\in H_{g,n}(X)$, $y\in H_{h,m}(Y)$, then the canonical map
$$H_{\gpower{(g\oplus h)}{p},*}(X\dc Y)\to H_{\gpower{g}{p}\oplus \gpower{h}{p},*}(X\dc Y)$$
induced by the shuffle map $\gpower{(g\oplus h)}{p}\to \gpower{g}{p}\oplus
\gpower{h}{p}$ in $\fC$ sends
\begin{align*}
Q_{i(p-1)}(x\otimes y)\mapsto (-1)^{\tfrac{nm(p-1)}{2}}\sum_{j+k=i}
Q_{j(p-1)}(x)\otimes Q_{k(p-1)}(y)
\end{align*}
and
\begin{align*}
\b Q_{i(p-1)}(x\otimes y)\mapsto (-1)^{\tfrac{nm(p-1)}{2}}\sum_{j+k=i}
&\big(\b Q_{j(p-1)}(x) \otimes Q_{k(p-1)}(y)\\
 &+(-1)^nQ_{j(p-1)}(x) \otimes \b Q_{k(p-1)}(y)\big).
\end{align*}
\label{external cartan}

\item (Internal Cartan formulas) If $x\in H_{g,n}(X)$ and $y\in H_{h,m}(X)$,
then the canonical map
$$H_{\gpower{(g\oplus h)}{p},*}(X)\to H_{\gpower{g}{p}\oplus \gpower{h}{p},*}(X)$$
sends
$$Q_{i(p-1)}(xy)\mapsto (-1)^{\tfrac{nm(p-1)}{2}}\sum_{j+k=i}
Q_{j(p-1)}(x)Q_{k(p-1)}(y)$$ 
and
\begin{align*}
\b Q_{i(p-1)}(xy)\mapsto (-1)^{\tfrac{nm(p-1)}{2}} \sum_{j+k=i}
&\big(\b Q_{j(p-1)}(x) Q_{k(p-1)}(y)\\
 &+(-1)^nQ_{j(p-1)}(x) \b Q_{k(p-1)}(y)\big).
\end{align*}
\label{internal cartan}
\item (Adem relations)
For $x\in H_{g,n}$, if $r>s$, then

$$Q_{r(p-1)}Q_{s(p-1)}(x)= q_{g,n,s}\sum_j
t_{r,s,j}\, Q_{(r+ps-pj)(p-1)}
Q_{j(p-1)}(x)$$ 
and
$$\b Q_{r(p-1)}Q_{s(p-1)}(x)=q_{g,n,s} \sum_j t_{r,s,j}\,
 \b Q_{(r+ps-pj)(p-1)}
Q_{j(p-1)}(x),$$
where $$q_{g,n,s}=[\vp(\tr_{g\oplus g}(\b_{g,g}^{-1}))(-1)^n]^{\tfrac{p-1}{2}}(-1)^{\tfrac{s(p-1)}{2}}$$
and $$t_{r,s,j}=(-1)^{\tfrac{r-j}{2}}
\binom{\tfrac{(j-s)(p-1)}{2}-1}{\tfrac{r-j}{2}-1};$$
if $r\geq s$, then
\begin{align*}
Q_{r(p-1)}\b
Q_{s(p-1)}(x)=q_{g,r}&\sum_j
\tfrac{(j-s)(p-1)}{pj-s(p-1)-r+1} t_{r,s,j}\, \b
Q_{(r-1+ps-pj)(p-1)}Q_{j(p-1)}(x)\\
&-q_{g,n,s}\sum_j t_{r,s,j}\,
Q_{(r+ps-pj)(p-1)}\b Q_{j(p-1)}(x)
\end{align*}
and
\begin{align*}
\b Q_{r(p-1)}\b
Q_{s(p-1)}(x)=-q_{g,n,s}\sum_j
t_{r,s,j}\,
\b Q_{(r+ps-pj)(p-1)}\b Q_{j(p-1)}(x)),
\end{align*}
where
$$q_{g,r}=\vp(\tr_{g\oplus g}(\b_{g,g}^{-1}))^{\tfrac{p-1}{2}}(-1)^{\tfrac{r(p-1)}{2}}\left(\tfrac{p-1}{2}!\right),$$ 
$$q_{g,n,s}=[\vp(\tr_{g\oplus g}(\b_{g,g}^{-1}))(-1)^n]^{\tfrac{p-1}{2}} (-1)^{\tfrac{s(p-1)}{2}},$$ and
$$t_{r,s,j}=(-1)^{\tfrac{r-1-j}{2}}
\binom{\tfrac{(j-s)(p-1)}{2}-1}{\tfrac{r-1-j}{2}}.$$
Here, if an expression in a binomial coefficient is not an integer, the binomial
coefficient is defined to be zero.
\label{item:adem}
\setcounter{lowerindex}{\value{enumi}}
\end{enumerate}
\end{thm}

In the above theorem, we also wish to include a statement about the
compatibility of the Dyer--Lashof operations with suspension in the chain
complex degree. Making a rigorous
statement of this nature requires some additional definitions. We found the perspective outlined in Lecture 3 of
some course notes by Lurie useful \cite{Lurie}. 
Below we write $\S$ and $\O$ for $\S^{\1_{\fG},1}$ and
$\S^{\1_{\fG},-1}$, respectively.

Given $X\in \Mod_{\Fvp}$, we will construct a map 
\begin{equation}\label{eq:loops}
H_{*,*}(E_{\infty}(\O X))\to  H_{*,*}(\O E_{\infty}(X))
\end{equation}
that is natural in $X$. 
 In particular, if we let $\O
X=\S^{g,n}\Fvp$, we obtain a map
$$\s\colon H_{h,m}(E_{\infty}(\S^{g,n}\Fvp))\to H_{h,m}(\O
E_{\infty}(\S^{g,n+1}\Fvp))\xrightarrow{\cong} H_{h,m+1}(E_{\infty}(\S^{g,n+1}\Fvp))$$
that defines the \emph{suspension of an operation} $H_{g,n}\to H_{h,m}$.

In fact, this is a special case of how suspension interacts with operations.
Suppose $X,\O X\in \Alg_{E_{\infty}}(\Mod_{\Fvp})$. Given any homology class
represented by a map $x\colon \S^{g,n}\Fvp\to \O X$, we may consider the diagram
\begin{equation}\label{eq:susp diagram}
\xymatrix{
H_{*,*}(E_{\infty}(\S^{g,n}\Fvp))\ar[r]^-{\s}\ar[d]^{H_{*,*}(E_{\infty}(x))} & H_{*,*}(\O
E_{\infty}(\S^{g,n+1}\Fvp))\ar[r]^{\cong}\ar[d] &
H_{*,*+1}(E_{\infty}(\S^{g,n+1}\Fvp))\ar[d]^{H_{*,*}(E_{\infty}(\S x))} \\
H_{*,*}(E_{\infty}(\O X)) \ar[r] \ar[d] & H_{*,*}(\O E_{\infty}(X))
\ar[r]^{\cong} & H_{*,*+1}(E_{\infty}(X)) \ar[d]\\
H_{*,*}(\O X) \ar[rr]^{\cong} & & H_{*,*+1}(X)
}
\end{equation}
where the suspension maps are the canonical ones, and the map
$$H_{*,*}(E_{\infty}(\O X))\to H_{*,*}(\O E_{\infty}(X))$$ comes from
\ref{eq:loops}. Given a class $Q\in H_{*,*}(E_{\infty}(\S^{g,n}\Fvp))$, which
determines a homology operation by sending $x$ to the image of $Q$ under the two
vertical maps on the left-hand side of the diagram, it makes sense to ask
whether $\S Q(x)$ equals $\s Q(\S x)$.
We then have the following
addition to Theorem \ref{thm:lower index}:

\begin{continuethm}{thm:lower index}
\begin{enumerate}
\setcounter{enumi}{\value{lowerindex}}
\item \label{item:susp1} The suspension map
$$\s\colon \Nat(H_{g,n},H_{\gpower{g}{p},pn+i(p-1)})\to
\Nat(H_{g,n+1},H_{\gpower{g}{p},pn+i(p-1)+1})$$
sends $Q_{i(p-1)}$ to
$$\s Q_{i(p-1)}=(-1)^{\tfrac{(p-1)n}{2}}\left(\tfrac{p-1}{2}\right)!\,
Q_{(i-1)(p-1)}$$
and $\b Q_{i(p-1)}$ to
$$\s \b Q_{i(p-1)}=-(-1)^{\tfrac{(p-1)n}{2}}\left(\tfrac{p-1}{2}\right)!\, \b
Q_{(i-1)(p-1)}.$$
\item If $X, \O X\in \Alg_{E_{\infty}}(\Mod_{\Fvp})$ with compatible
$E_{\infty}$ structures, in the sense that the lower half of diagram
\ref{eq:susp diagram} commutes, then the suspension
isomorphism
$\S\colon H_{g,n}(\O X)\to H_{g,n+1}(X)$
satisfies $$\S
Q_{i(p-1)}(x)=(-1)^{\tfrac{(p-1)n}{2}}\left(\tfrac{p-1}{2}\right)!\,
Q_{(i-1)(p-1)}(\S x)$$
and
$$\S\b Q_{i(p-1)}(x)=-(-1)^{\tfrac{(p-1)n}{2}}\left(\tfrac{p-1}{2}\right)!\,
\b Q_{(i-1)(p-1)}(\S x).$$
\label{item:susp2}
\end{enumerate}
\end{continuethm}

Before providing a proof of Theorem \ref{thm:lower index}, we first construct
the promised zigzag \ref{eq:loops}.
 
\begin{con}\label{con:zigzag}
Let $X\in \Mod_{\Fvp}$. Let $Y$ be the mapping cone of the identity map $\O X\to
\O X$ so that we have a sequence
$$\O X\to Y \to X,$$
where $Y$ is acyclic.
Taking symmetric powers, we obtain maps
$$\Sym^k(\O X)\to \Sym^k(Y)\to \Sym^k(X).$$
Let $Y$ have the filtration $F_iY=0$ if $i<0$, $F_iY=Y$ if $i>0$, and $F_0Y$ is
the image of $\O X$ in $Y$. This induces a filtration on $\Sym^k(Y)$ via
$$F_i \Sym^k(Y)=\im\left(\cC_{\infty}(k)\otimes_{\S_k} \bigoplus_{i_1+\dots+i_k=i}
F_{i_1}Y\dc\cdots \dc F_{i_k}Y\right).$$
Note that the image of $\Sym^k(\O X)$ in $\Sym^k(Y)$ is contained in filtration
$F_{k-1}$ (in fact, it is contained in filtration $F_0$), so we have a map
$\Sym^k(\O X)\to F_{k-1}\Sym^k(Y).$
Since the map $\Sym^k(Y)\to \Sym^k(X)$ is surjective, we may also
define a map
$$H_{*,*}(\O \Sym^k(X))\to H_{*,*}(F_{k-1}\Sym^k(Y))$$
by taking a representative in $\Sym^k(X)$, choosing a lift in $\Sym^k(Y)$, and applying the differential in
$\Sym^k(Y);$ this induces a well-defined map on homology that is independent of
the choice of lift. Moreover, we claim
this map is an isomorphism. 

Since $Y$ is acyclic, $\Sym^k(Y)$ is acyclic as well for $k>0$
by Lemma \ref{lem:preservewes}.
 Thus, any
cycle $x$ in $F_{k-1}\Sym^k(Y)$ is equal to $d(y)$ for some $y\in \Sym^k(Y)$. 
Now, $F_{k-1}\Sym^k(Y)$ is the kernel of the map $\Sym^k(Y)\to \Sym^k(X)$, so
since $d(y)$ is in $F_{k-1}\Sym^k(Y)$, $y$ maps to a cycle $z\in \Sym^k(X)$ under
this map. Then $y$ is a lift of $z$, and hence the homology class of $z$ maps to
the homology class of $x$ under the map
$$H_{*,*}(\O\Sym^k(X))\to H_{*,*}(F_{k-1}\Sym^k(Y)),$$
which gives us that this map is surjective.
If the class of a cycle $z\in H_{*,*}(\O \Sym^k(X))$ is in the kernel of this
map, then $z$ has a lift $y\in \Sym^k(Y)$ with $d(y)=d(y')$ for 
$y'\in F_{k-1}\Sym^k(Y)$. Since $y-y'$ is also a lift of $z$, we may thus assume
that $d(y)=0$. However, if $y$ is a cycle, it is also a boundary, since
$\Sym^k(Y)$ is acyclic. Thus, if $y=d(z')$, then $z$ is given by the
differential applied to the image of $z'$. Hence, the map is injective as well.
\end{con}

Now we are ready to prove Theorem \ref{thm:lower index}. The proof is separated
into different parts for each statement, in order to make it easier to find the
proof of a particular statement.

\begin{proof}[Proof (of Theorem \ref{thm:lower index}, Statement
\ref{additivity})]
Recall that we defined the Dyer--Lashof operations on $H_{g,n}$ using classes in the homology
of $W\otimes_{C_p}
(\S^{g,n}\Fvp)^{\dc p}$. For any map $f\colon \S^{g,n}\Fvp\to X$ representing a
homology class, the diagram
$$
\begin{tikzcd}
W\otimes_{C_p} (\S^{g,n}\Fvp)^{\dc p} \arrow[r, "\text{id}\otimes f^{\dc p}"]
\arrow{d} & W\otimes_{C_p} X^{\dc p} \arrow{d}\\
\Sym^p(\S^{g,n}\Fvp) \arrow[r,"\Sym^p(f)"] & \Sym^p(X)
\end{tikzcd}
$$
commutes. A Dyer--Lashof operation is defined by a class $$e_{i}\otimes \iota^p\in
(W\otimes_{C_p} (\S^{g,n}\Fvp)^{\dc p})(\gpower{g}{p}),$$ where $\iota$ denotes $1\in
\S^{g,n}\Fvp(g)_n$. Applying this operation to $f$ is given by taking the
image of $e_{i}\otimes \iota^p$ under either of these compositions after evaluating on
$\gpower{g}{p}$, and then applying the $E_{\infty}$-multiplication for $X$. Thus, it
suffices to show Frobenius linearity in $\Sym^p(X)$ before applying the
$E_{\infty}$-multiplication. If we write $x$ for $f(\iota)$, then the image of $e_i\otimes \iota^p$ is given by
$$e_{i}\otimes x^{\otimes p}\in
\cC_{\infty}(p)\otimes_{\S_p} X(g)^{\otimes p}.$$
First we wish to show that
$$e_{i}\otimes (x+y)^{\otimes p}= e_{i}\otimes x^{\otimes
p}+e_{i}\otimes y^{\otimes p}$$
after passing to homology $H_{\gpower{g}{p},*}(\Sym^p(X))$.
We prove this on the chain level (see also \cite{Mayalg}). Note that
$$(x+y)^{\otimes p}=x^{\otimes p}+y^{\otimes p}+\sum_{j=1}^{p-1}
(1+\a+\dots+\a^{p-1})x^{\otimes j}y^{\otimes p-j},$$
where $\a$ is our chosen generator of the cyclic group $C_p$, which cyclicly
permutes the monomials $x^{\otimes j}y^{\otimes p-j}$. 
In $W$, if $i$ is even, we have that $$(1+\a+\dots+\a^{p-1})
e_{i}=d((\a-1)^{p-2} e_{i+1}),$$
and if $i$ is odd, we have that
$$(1+\a+\dots+\a^{p-1})e_i=d(e_{i+1}).$$
Thus, 
$$e_{i}\otimes \sum_{j=1}^{p-1}(1+\a+\dots+\a^{p-1})x^{\otimes j}y^{\otimes
p-j}$$
is a boundary, which implies that the Dyer--Lashof operations are additive. 
Since $$e_{i}\otimes (t x)^{\otimes p}=t^p e_{i}\otimes x^{\otimes
p},$$ it follows that Dyer--Lashof operations are Frobenius linear. This proves
\ref{additivity}
\end{proof}
 
Statement \ref{zero} follows directly from Lemma \ref{lem:zero}.

\begin{proof}[Proof (of Theorem \ref{thm:lower index}, Statement
\ref{power})]
The $p$-fold product is a weight $(1,\dots,1)$ $p$-ary operation $\mu$ living in
the $p$th charge summand of
$$H_{g_1\oplus \dots\oplus g_p,n_1+\dots+n_p}(E_{\infty}(\oplus \S^{g_i,n_i}\Fvp)).$$
When all of the inputs are equal, this may be viewed as an operation $\g$ in
$$H_{\gpower{g}{p},pn}(E_{\infty}(\S^{g,n}\Fvp))$$
under the image of the canonical fold map $\oplus \S^{g,n}\Fvp\to\S^{g,n}\Fvp$. In particular,
whenever $x\in H_{g,n}(X)$,
$$\mu(x\otimes x\otimes \dots\otimes x)=\g(x).$$
Thus, it suffices to show that the class in $
H_{\gpower{g}{p},pn}(E_{\infty}(\S^{g,n}\Fvp))\cong \F$ defining $Q_0$ is equal to $\g$.
Since $\mu$ is the iterated product, it is determined by a point in arity $p$ of
a space operad.

Note that the
map $W\to\cC_{\infty}(p)$ must take the chain $e_0$ to a chain whose image under
the augmentation map is $1\in\F$. Thus, since $\cC_{\infty}$ comes from a space
operad, this chain is homologous to any chain determined by a point. Hence, both
$\g$ and $Q_0$ may be represented as a chain of a point tensored with
$1^{\otimes p}$, and therefore they are the same class in homology.
\end{proof}

\begin{proof}[Proof (of Theorem \ref{thm:lower index}, Statement
\ref{unit})]
Recall that the unit map $\Fvp\to X$ is a map of
$E_{\infty}$-algebras, and therefore by naturality it suffices to prove the
statement for $\Fvp$. However, the $E_{\infty}$-structure for $\Fvp$ factors
through the commutative operad:
$$\Sym^p(\Fvp)\to \F[0]\otimes_{\S_p} \Fvp^{\dc p}\to \Fvp.$$
If $e_j\in \cC_{\infty}(p)$ defines a Dyer--Lashof operation, then the image of
$e_j$ in $\F[0]$ is zero if $j>0$, and therefore this operation is identically
zero.
\end{proof}

\begin{proof}[Proof (of Theorem \ref{thm:lower index}, Statements \ref{external cartan}
and \ref{internal cartan})] 

First we show the external version of the formulas. The operation $\b^{\e}
Q_{i(p-1)}$ on $H_{g,n}\otimes H_{h,m}$ is represented by the chain
$$e_{i(p-1)-\e}\otimes (\iota_1\otimes\iota_2)^p\in \Sym^p(\S^{g,n}\Fvp\dc
\S^{h,m}\Fvp)(\gpower{(g\oplus h)}{p}),$$
where $\iota_i$ denotes $1\in \S^{g_i,n_i}\Fvp(g_i)_{n_i}$.
Evaluating this operation on a class $x\otimes y$, where $x$ is represented by a
map $\S^{g,n}\Fvp\to X$ and $y$ is
represented by $\S^{h,m}\Fvp\to Y$, uses either composition in the commutative
diagram
$$
\begin{tikzcd}
\Sym^p(\S^{g,n}\Fvp\dc \S^{h,m}\Fvp) \arrow[r, "T"] 
\arrow{d} &
\Sym^p(\S^{g,n}\Fvp)\dc \Sym^p(\S^{h,m}\Fvp) \arrow{d}\\
\Sym^p(X\dc Y) \arrow[r]
& \Sym^p(X)\dc \Sym^p(Y)
\end{tikzcd}
$$
followed by the $E_{\infty}$-multiplications on $X$ and $Y$. Recall that the
horizontal arrows in this diagram are defined using the coproduct on
$\cC_{\infty}(p)$ followed by a shuffle map taking
\begin{equation}\label{shuffle}
(\S^{g,n}\Fvp\dc \S^{h,m}\Fvp)^{\dc p}\to (\S^{g,n}\Fvp)^{\dc p}\dc
(\S^{h,m}\Fvp)^{\dc p},
\end{equation}
which in turn is followed by braiding $\cC_{\infty}(p)\otimes
(\S^{g,n}\Fvp)^{\dc p}\to (\S^{g,n}\Fvp)^{\dc p}\otimes \cC_{\infty}(p).$

The canonical comparison map $H_{\gpower{(g\oplus h)}{p}}(X\dc Y)\to
H_{\gpower{g}{p}\oplus \gpower{h}{p}}(X\dc Y)$ is
induced by the shuffle morphism $\psi\colon \gpower{(g\oplus h)}{p}\to
\gpower{g}{p}\oplus \gpower{h}{p}$ in $\fC$.
Combining this with the above diagram, we see that it suffices to compute the
image of $e_{i(p-1)-\e}\otimes(\iota_1\otimes\iota_2)^p$ under the composition
$$
\begin{tikzcd}
\Sym^p(\S^{g,n}\Fvp\dc \S^{h,m}\Fvp)(\gpower{(g\oplus h)}{p})
\ar[d,"{T(\gpower{(g\oplus h)}{p})}"]\\
(\Sym^p(\S^{g,n}\Fvp) \dc \Sym^p(\S^{h,m}\Fvp))(\gpower{(g\oplus h)}{p}) \ar[d, "\psi_*"]\\
(\Sym^p(\S^{g,n}\Fvp) \dc\Sym^p(\S^{h,m}\Fvp))(\gpower{g}{p}\oplus \gpower{h}{p}).
\end{tikzcd}
$$

We know that the map $W\to \cC_{\infty}(p)$ preserves the coproduct, so from the
definition of the coproduct on $W$ (see \ref{eq:Wcoprod}), we know that
the coproduct
$$\cC_{\infty}(p)\to \cC_{\infty}(p)\otimes \cC_{\infty}(p)$$
takes $e_{i(p-1)}$ to
$$\sum_{2j+2k=i(p-1)} e_{2j}\otimes e_{2k}+\sum_{2j+2k+2=i(p-1)}\sum_{0\leq r<s<p}
\a^r e_{2j+1}\otimes \a^s e_{2k+1},$$
where as usual $\a$ denotes a chosen generator of the cyclic group $C_p$.
Recall that the shuffle map \ref{shuffle} is obtained by first applying a
shuffle map in $\Ch_{\F}$ (which introduces a sign) and then applying the inverse shuffle morphism
$\psi^{-1}\colon \gpower{g}{p}\oplus \gpower{h}{p}\to \gpower{(g\oplus h)}{p}.$
Thus, the shuffle maps in the composition $\psi_*\circ T(\gpower{(g\oplus
h)}{p})$ cancel, and this composition takes $e_{i(p-1)}\otimes (\iota_1\otimes \iota_2)^p$
to
\begin{align}\label{eq:cartan1}
(-1)^{\tfrac{nmp(p-1)}{2}}&\Big(\sum_{2j+2k=i(p-1)}
(e_{2j}\otimes \iota_1^p)\otimes (e_{2k}\otimes \iota_2^p)\\
&+(-1)^{np}\sum_{2j+2k+2=i(p-1)}\, \sum_{0\leq r<s<p}
(\a^r e_{2j+1}\otimes \iota_1^p)\otimes (\a^s e_{2k+1}\otimes \iota_2^p)\Big).
\nonumber
\end{align}
Again, the signs here arise from the symmetry in $\Ch_{\F}$: we obtain
$(-1)^{\tfrac{nmp(p-1)}{2}}$ from doing $p(p-1)/2$ applications of the symmetry
$$\S^{h,m}\Fvp\dc \S^{g,n}\Fvp\to \S^{g,n}\Fvp\dc \S^{h,m}\Fvp,$$
and the other signs are obtained from $p$ applications of the
symmetry
$$\cC_{\infty}(p)\otimes \S^{g,n}\Fvp\to \S^{g,n}\Fvp\otimes \cC_{\infty}(p).$$

Note that $\a$ acts on $\iota_1^p$ as multiplication by
$((-1)^n\vp(\tr_{g\oplus g}(\b^{-1}_{g,g}))^{p-1}=1$, and similarly for $\iota_2^p$, so that
$$\sum_{2j+2k+2=i(p-1)}\,\sum_{0\leq r<s<p}
(\a^r e_{2j+1}\otimes \iota_1^p)\otimes (\a^s e_{2k+1}\otimes
\iota_2^p)$$
equals
$$\binom{p}{2}\sum_{2j+2k+2=i(p-1)}(e_{2j+1}\otimes \iota_1^p)\otimes
(e_{2k+1}\otimes \iota_2^p),$$
which is 0 in characteristic $p$.
Applying Lemma \ref{lem:zero} to the first term, we thus know that \ref{eq:cartan1} is equal to
$$(-1)^{\tfrac{nmp(p-1)}{2}}\sum_{j+k=i}(e_{j(p-1)}\otimes
\iota_1^p)\otimes (e_{k(p-1)}\otimes \iota_2^p).$$
Using the fact that $p$ is odd, we obtain the constant in
\ref{external cartan}

Similarly, $\psi_*\circ T(\gpower{(g\oplus h)}{p})$ takes $e_{i(p-1)-1}\otimes (\iota_1\otimes
\iota_2)^p$ to
$$(-1)^{\tfrac{nmp(p-1)}{2}}\sum_{2j+2k=i(p-1)} (-1)^{pn}(e_{2j}\otimes
\iota_1^p)\otimes (e_{2k+1}\otimes \iota_2^p)+ (e_{2j+1}\otimes \iota_1^p)\otimes
(\a e_{2k}\otimes \iota_2^p).$$
Once again using the fact that $\a$ acts trivially and applying Lemma
\ref{lem:zero} proves the external Cartan formula for $\b Q_{i(p-1)}$.

By Lemma \ref{lem:chain prod}, we may define a chain-level product map $X\dc
X\to X$ that preserves the $E_{\infty}$-structure after passing to homology.
If $\S^{g,n}\Fvp\to X$, $\S^{h,m}\Fvp\to X$, and $\S^{g\oplus h,n+m}\Fvp\to X$ are maps
representing classes $x,y,$ and $xy$, respectively, then the diagram
$$
\begin{tikzcd}
H_{*,*}(\Sym^p(\S^{g,n}\Fvp\dc \S^{h,m}\Fvp)) \arrow{r}\arrow{d} & H_{*,*}(\Sym^p(X\dc
X)) \arrow{r}\arrow{d} & H_{*,*}(X\dc X) \arrow{d}\\
H_{*,*}(\Sym^p(\S^{g\oplus h,n+m}\Fvp))\arrow{r} & H_{*,*}(\Sym^p(X))\arrow{r} &
H_{*,*}(X)
\end{tikzcd}
$$
commutes, so the right-hand map takes $\b^{\e}Q_{i(p-1)}(x\otimes y)$ to
$\b^{\e}Q_{i(p-1)}(xy)$.
Since the map $X\dc X\to X$ induces the product on homology, the internal Cartan formulas follow from the
external ones by considering the diagram
$$
\begin{tikzcd}
H_{\gpower{(g\oplus h)}{p},*}(X\dc X) \ar[r, "\psi_*"] \ar[d] &
H_{\gpower{g}{p}\oplus \gpower{h}{p},*}(X\dc X) \ar[d]\\
H_{\gpower{(g\oplus h)}{p},*}(X)\ar[r,"\psi_*"] &
H_{\gpower{g}{p}\oplus\gpower{h}{p}}(X).
\end{tikzcd}
$$
\end{proof}

In order to prove the Adem relations, we require a few lemmas. The first is an
elementary statement about binomial coefficients mod $p$.

\begin{lem}\label{lem:binom}
The following identities hold modulo $p$.
\begin{enumerate}
\item Let $j\geq 0$ and $t>0$. Then
$$\binom{j(p-1)}{1+p+\dots+p^t-j}\equiv 0\bmod{p}$$
unless $j=1+p+\dots+p^t$.
Moreover,
$$\binom{j(p-1)-1}{p^t-j}\equiv 0\bmod{p}$$
unless $j=p^t.$

\item If $i,j\in \Z_{\geq 0}$ and $i<p^t$, then $\binom{i+j}{j}\equiv
\binom{i+j+ap^t}{j+ap^t}\bmod{p}$ for any $0\leq a<p$.
\end{enumerate}
\end{lem}

\begin{proof}
All statements follow by elementary arguments from \emph{Lucas's Theorem}, which
states that if $m,n\geq 0$ with base $p$ expansions $$m=\sum_{k=0}^{\ell} m_k
p^k \text{  and  } n=\sum_{k=0}^{\ell} n_kp^k,$$
then
$$\binom{m}{n}\equiv \prod_{k=0}^{\ell} \binom{m_k}{n_k}\bmod{p}.$$

\end{proof}

We also require a version of singular cochains on a sphere in our category
$\Mod_{\Fvp}$. As is described in \cite{Mayalg}, if the singular cochain complex
$C^*(X)$
of a space $X$ is thought of as a nonpositively-graded chain complex, then the Steenrod operations
coincide with the Dyer--Lashof operations, where $C^*(X)$ may be given an
$E_{\infty}$-algebra structure from a chain-level construction of the cup
product. 

\begin{lem}\label{lem:sphere}
For any $k>0$ there exists $Y_k\in \Alg_{E_{\infty}}(\Mod_{\Fvp})$ such that
$$H_{g,n}(Y_k)\cong \begin{cases}
\F & \text{if }g\cong \1_{\fG} \text{ and } n=-k \text{ or } n=0\\
0 & \text{otherwise,}
\end{cases}$$
and there exists $y\in H_{\1_{\fG},-k}(Y_k)$ such that 
\begin{enumerate}
\item $Q_{i(p-1)}(y)=0$ if
$i\not=k$;
\item  the map
$$H_{\gpower{\1_{\fG}}{p},-k}(Y_k)\to H_{\1_{\fG},-k}(Y_k)$$
induced by the canonical map $(\1_{\fG})^{p}\to \1_{\fG}$ sends
$$Q_{k(p-1)}(y)\mapsto y;$$
\item $\b Q_{i(p-1)}(y)=0$ for all $i$.
\end{enumerate}
\end{lem}

\begin{proof}
We may think of singular cochains $C^*(S^k;\F)$ on a $k$-sphere as a
negatively-graded chain complex, which is therefore an object in $\Ch_{\F}$. The
free functor
$$\Ch_{\F}\to \Mod_{\Fvp}(\Ch^{\fC}_{\F})$$
taking a chain complex $X$ to $X\otimes \Fvp$  is strong symmetric monoidal, so it induces a
functor
$$\Alg_{E_{\infty}}(\Ch_{\F})\to \Alg_{E_{\infty}}(\Mod_{\Fvp}).$$
If $X\in \Alg_{E_{\infty}}(\Mod_{\Fvp})$, then $X(\1_{\fG})\in \Ch_{\F}$
inherits an $E_{\infty}$-algebra structure using the fact that there is a
canonical map $(\1_{\fG})^{k}\to \1_{\fG}$ for all $k$. Explicitly, we
have maps
$$\cC_{\infty}(k)\otimes_{\S_k} X(\1_{\fG})^{\otimes k}\to
\Sym^k(X)(\1_{\fG}^{k})\to X(\1_{\fG}^{k})\to X(\1_{\fG})$$
that are easily seen to be compatible with the operad composition maps.

The cochain complex of a space may be given an $E_{\infty}$-algebra structure
such that the induced product on homology coincides with the cup product, and
the Dyer--Lashof operations coincide with the Steenrod operations (see sections
7 and 8 of \cite{Mayalg}). We therefore let $Y_k$ be the image of the
$E_{\infty}$-algebra $C^*(S^k;\F)$ in $\Alg_{E_{\infty}}(\Mod_{\Fvp}).$

Any map $\F[n]\to C^*(S^k;\F)$ in $\Ch_{\F}$ representing a homology class induces a map
$\S^n\Fvp\to Y_k$ in $\Mod_{\Fvp}$, such that the diagram
$$
\begin{tikzcd}
E_{\infty}(\F[n]) \ar[r,"\simeq"] \ar[d] & E_{\infty}(\S^n\Fvp)(\1_{\fG}) \ar[d]\\
C^*(S^k;\F) \ar[r,"\simeq"] & Y_k(\1_{\fG})
\end{tikzcd}
$$
in $\Alg_{E_{\infty}}(\Ch_{\F})$ commutes. Thus, the isomorphism $H^{*}(S^k;\F)\to
H_{\1_{\fG},-*}(Y_k)$ preserves the Dyer--Lashof operations, and therefore $Y_k$
satisfies the desired properties.
\end{proof}

\begin{proof}[Proof (of Theorem \ref{thm:lower index}, Statement
\ref{item:adem})]
 The composition of two Dyer--Lashof operations is an operation
in
$$H_{\gpower{g}{p^2},*}(\Sym^{p^2}(\S^{g,n}\Fvp))$$
determined by the composition 
$$
\begin{tikzcd}
\cC_{\infty}(p)\otimes_{C_p} (\cC_{\infty}(p)\otimes (\S^{g,n}\Fvp)^{\dc p})^{\dc p}
\ar[d]\\
\cC_{\infty}(p)\otimes_{C_p} (\cC_{\infty}(p)^p\otimes ((\S^{g,n}\Fvp)^{\dc p})^{\dc
p})\ar[d]\\
\Sym^{p^2}(\S^{g,n}\Fvp).
\end{tikzcd}
$$
That is, if $\iota$ denotes a chosen generator of $\S^{g,n}\Fvp$, then the
composition of operations determined by $e_i$ and $e_j$ is given by
\begin{equation}\label{eq:first sign}
(-1)^{nj(p-1)/2}\theta(e_i\otimes e_j^p)\otimes (\iota^p)^p,
\end{equation}
where $\theta$ is the composition in $\cC_{\infty}$. Here the sign arises from
braiding $\cC(p)$ with $(\S^{g,n}\Fvp)^{\dc p}$ $p(p-1)/2$ times.
Similar reasoning to that of Lemma \ref{lem:zero} shows that
$$H_{\gpower{g}{p^2},*}(\Sym^{p^2}(\S^{g,n}\Fvp))\cong H_{*-p^2n}(\S_{p^2};\F(q)),$$
where $q=(-1)^n\vp(\tr_{g\oplus g}(\b^{-1}_{g,g}))$ and $\F(q)$ is as in the proof of the same lemma.

Moreover, $$H_{\gpower{g}{p^2},*}(\cC_{\infty}(p)\otimes_{C_p} (\cC_{\infty}(p)\otimes
(\S^{g,n}\Fvp)^{\dc p})^{\dc p})\cong
H_{*-p^2n-p\bullet}(C_p;H_{\bullet}(C_p;\F)^p).$$
If we view $\S_{p^2}$ as acting on tuples $(i,j)$ where $1\leq i,j\leq
p$, let $\a\in \S_{p^2}$ be the permutation sending $(i,j)\mapsto (i+1,j)$, and
let $\a_i$ send $(i,j)\mapsto (i,j+1)$ and $(k,j)\mapsto (k,j)$ if $k\not=i$.
Let $\t\leq \S_{p^2}$ be the subgroup generated by $\a$ and the $\a_i$. Then
$$H_*(\t;\F(q))\cong H_*(C_p;H_*(C_p;\F)^p),$$
and the above isomorphisms can be chosen so that the map
$$H_*(\t;\F(q))\to H_*(\S_{p^2};\F(q))$$
induced by inclusion agrees with the map
$$H_{\gpower{g}{p^2},*}(\cC_{\infty}(p)\otimes_{C_p} (\cC_{\infty}(p)^p\otimes
((\S^{g,n}\Fvp)^{\dc p})^{\dc
p}))\to
H_{\gpower{g}{p^2},*}(\Sym^{p^2}(\S^{g,n}\Fvp)).$$

The Adem relations arise from relations in $H_*(\S_{p^2};\F(q))$ coming from
$H_*(\t;\F(q))$.
We
follow the general method of Section 4 in \cite{Mayalg} to obtain these
relations, making modifications for twisted coefficients and for lower indexing.

View $C_p\times C_p$ as a subgroup of $\S_{p^2}$ by thinking of $(1,0)$ as the permutation
$\a$, and thinking of $(0,1)$ as the permutation
$\b$ sending
$(i,j)$ to $(i,j+1)$. Note that $\a\b=\b\a$, so this indeed embeds $C_p\times C_p$ as a
subgroup of $\S_{p^2}$.

Since $\b=\a_1\cdots \a_p$, we see that this inclusion map factors through $\t$.
Thus, we have an induced map
$$H_*(C_p\times C_p;\F(q))\xrightarrow{f} H_*(\t;\F(q))\to H_*(\S_{p^2};\F(q)).$$
May gives an explicit computation of the map $f$ (Lemma 4.6 (ii),
\cite{Mayalg}):
\begin{align}\label{eq:group rel}
f(e_r \otimes e_s)=&\sum_k t(s,k)
e_{r+(2pk-s)(p-1)}\otimes e^p_{s-2k(p-1)}\\
&-\d(r)\d(s-1)\sum_k t(s-1,k)
 e_{r+p+(2pk-s)(p-1)}\otimes
e^p_{s-2k(p-1)-1},
\nonumber
\end{align}
where $$t(j,k)=(-1)^k v(j)\binom{[j/2]-k(p-1)}{k},$$
$$v(j)=(-1)^{j(j-1)(p-1)/4} \left(\tfrac{p-1}{2}!\right)^j,$$
and $\d(j)$ is 1 if $j$ is odd and 0 if $j$ is even. The constant $v(j)$
will be used throughout the paper, but the constant $t(j,k)$ is only in this
proof.

Now let $\g\in \S_{p^2}$ be the permutation that sends $(i,j)\mapsto (j,i)$.
Then $\g\a\g^{-1}=\b$, so the group homomorphism $\S_{p^2}\to \S_{p^2}$ given by
conjugation by $\g$ restricts to a homomorphism $C_p\times C_p\to C_p\times
C_p$. Together with the map $\F(q)\to \F(q)$ that acts by $\g$ on the left, this
induces maps of group homology making the diagram
$$
\begin{tikzcd}
H_*(C_p\times C_p;\F(q)) \arrow[r, "f"] \arrow[d, "\g_*"] & H_*(\t;\F(q))
\arrow[r,"g"] & H_*(\S_{p^2};\F(q)) \arrow[d, "\g_*"]\\
H_*(C_p\times C_p;\F(q))\arrow[r,"f"] & H_*(\t;\F(q)) \arrow[r,"g"] &
H_*(\S_{p^2};\F(q))
\end{tikzcd}
$$
commute. It is a standard fact in group homology that the map $\g_*$ is the
identity on $H_*(\S_{p^2};\F(q))$. Moreover, it is straightforward to compute
$\g_*$ on $H_*(C_p\times C_p;\F(q))$ (May does this in Lemma 4.4); using the
fact that $\g$ may be written as a product of $p(p-1)/2$ transpositions, we have
that
\begin{equation}\label{eq:second sign}
\g_*(e_i\otimes e_j)=(-1)^{ij}q^{\tfrac{p-1}{2}} e_j\otimes e_i.
\end{equation}
Using the equations \ref{eq:first sign}, \ref{eq:group rel}, and \ref{eq:second
sign}, together with the fact that $gf=gf\g_*$, we have the relation
\begin{align}\label{eq:final}
(-1&)^{ns(p-1)/2}\sum_k t(s,k)
\theta(e_{r+(2pk-s)(p-1)}\otimes e^p_{s-2k(p-1)})\otimes\iota^{p^2}\\
&-(-1)^{n(s-1)(p-1)/2}\d(r)\d(s-1)\sum_k t(s-1,k) \theta(e_{r+p+(2pk-s)(p-1)}\otimes
e^p_{s-2k(p-1)-1})\otimes \iota^{p^2}\nonumber \\
&=
(-1)^{rs+nr(p-1)/2}q^{\tfrac{p-1}{2}}
\sum_{\ell} t(r,\ell)
\theta(e_{s+(2p\ell-r)(p-1)}\otimes e^p_{r-2\ell(p-1)})\otimes \iota^{p^2} \nonumber \\
&
-(-1)^{D}q^{\tfrac{p-1}{2}}
\d(s)\d(r-1)\sum_{\ell}
t(r-1,\ell) \theta(e_{s+p+(2p\ell-r)(p-1)}\otimes
e^p_{r-2\ell(p-1)-1})\otimes \iota^{p^2}
\nonumber
\end{align}
in $H_{\gpower{g}{p^2},*}(\Sym^{p^2}(\S^{g,n}\Fvp))$,
where $t(j,k)$ is as defined in \ref{eq:group rel}, and
$$D=rs+n(r-1)(p-1)/2.$$

Since the homology of $\cC_{\infty}(k)$, as well as the group homologies
$H_*(C_p)$, $H_*(\t)$, and $H_*(\S_k)$ may all be computed from integral chain
complexes that have been reduced mod $p$, these homologies all carry a Bockstein
homomorphism $\b$. May proves that 
$$\b(e_k\otimes e_{\ell}^p\otimes \iota^{p^2})=\d(k-1)e_{k-1}\otimes e_{\ell}^p
\otimes \iota^{p^2};$$
hence, it suffices to show only the first and third Adem relations.

Replacing $r$ by $(r+s(p-1))(p-1)$ and $s$ by $s(p-1)$, followed by $j=s-2k$ and
$i=r+s(p-1)-2\ell$ in \ref{eq:final}, and using that in general when $r$ and $s$ have the same
parity, then $v(s(p-1))\equiv v((r+s(p-1))(p-1))\mod{p}$, we obtain
\begin{align}\label{eq:rewritten}
\sum_j (-1)^{\tfrac{s-j}{2}}& 
\binom{j\tfrac{(p-1)}{2}}{\tfrac{s-j}{2}}
Q_{(r+ps-pj)(p-1)}Q_{j(p-1)}\\
&=
q^{\tfrac{p-1}{2}}\sum_i (-1)^{\tfrac{r+s(p-1)-i}{2}}
\binom{\tfrac{i(p-1)}{2}}{\tfrac{r+s(p-1)-i}{2}}
Q_{(r+ps-pi)(p-1)}Q_{i(p-1)},
\nonumber
\end{align}
where we use the notation $Q_kQ_{k'}$ to mean $\theta(e_k\otimes e_{k'}^{p})\otimes
\iota^{p^2}$.
Note that when introducing $i$ and $j$, we implicitly made some
assumptions on the parities of $i$, $j$, $s$, and $r$. However, we know from Lemma
\ref{lem:zero} that $Q_{k'(p-1)}$ will be zero for either all odd $k'$ or all
even $k'$; moreover, $Q_{k'(p-1)}$ has bidegree $(\gpower{g}{p},pn+k'(p-1))$, and
$$\vp(\tr_{\gpower{g}{p}\oplus
\gpower{g}{p}}(\b_{\gpower{g}{p},\gpower{g}{p}})))(-1)^{pn+k'(p-1)}=\vp(\tr_{g\oplus
g}(\b_{g,g}))^p(-1)^n=\vp(\tr_{g\oplus g}(\b_{g,g}))(-1)^n$$
by Equation \ref{eq:trhomo},
so $Q_{k(p-1)}Q_{k'(p-1)}$ will be zero whenever $k$ and $k'$ do not have the
same parity. Thus, we are justified in assuming that $r$ and $s$ have the same
parity, and thus, we may assume that $s$ and $j$ have the same parity, as do $i$
and $r$. There is therefore no issue with fractions occuring in binomial
coefficients or exponents.

The assumption that $r>s$ ensures that no operation appears on both sides of
\ref{eq:rewritten}. To see this, note that the binomial coefficient on the
left-hand side is zero unless
$$\tfrac{s}{p}\leq j\leq s,$$
and the binomial coefficient on the right-hand side is zero unless
$$\tfrac{r+s(p-1)}{p}\leq i\leq r+s(p-1).$$
If $r>s$, then $s<\tfrac{r+s(p-1)}{p}$, so these ranges are disjoint.

We first show that the relations hold for $s=2(1+p+\dots+p^t)$. As mentioned
above, we may assume that $r,s,j,i$ all have the same parity. Therefore,
replacing $r,j,i$ by $2r,2j,2i$ in \ref{eq:rewritten}, we obtain
\begin{align*}
\sum_j &(-1)^{1+\dots+p^t-j} 
\binom{j(p-1)}{1+\dots+p^t-j}
Q_{(2r+ps-2pj)(p-1)}Q_{2j(p-1)}\\
&=
q^{\tfrac{p-1}{2}}\sum_i (-1)^{r+p^{t+1}-1-i}
\binom{i(p-1)}{r+p^{t+1}-1-i}
Q_{(2r+ps-2pi)(p-1)}Q_{2i(p-1)}.
\end{align*}
On the left-hand side, the binomial coefficients are zero unless $j=1+\dots+p^t$
by Lemma \ref{lem:binom}. On the right-hand side, note that
$Q_{(2r+ps-2pi)(p-1)}$ is zero unless $2r+ps-2pi\geq 0$. This inequality implies
that
$$pi-r-p^{t+1}+1\leq 1+p+\dots+p^t<p^{t+1}.$$
Therefore, by Lemma \ref{lem:binom}, if $r-i-1\geq 0$,
\begin{equation}\label{eq:binomrepl}
\binom{(i-\tfrac{s}{2})(p-1)-1}{pi-r-\tfrac{s(p-1)}{2}}=\binom{i(p-1)-p^{t+1}}{r-i-1}\equiv
\binom{i(p-1)}{r-i+p^{t+1}-1}\bmod{p}.
\end{equation}
If $r-i-1<0$, then $\binom{i(p-1)-p^{t+1}}{r-i-1}=0$, and
similar methods to those mentioned in Lemma \ref{lem:binom} show
that $\binom{i(p-1)}{r-i+p^{t+1}-1}\equiv 0\bmod{p}$ as well when
$r-pi+p+\dots+p^{t+1}\geq 0$. When $r-pi+p+\dots+p^{t+1}<0$, it may \emph{not}
be the case that \ref{eq:binomrepl} holds; however, in
this case $Q_{(2r+ps-2pi)(p-1)}=0$, and thus, we are justified in substituting
the binomial coefficients in \ref{eq:binomrepl} for all $i$.
Hence, when $s=2(1+\dots+p^t)$, we have shown that for any $r>s$ and any $x\in
H_{g,n}(X)$,
$Q_{r(p-1)}Q_{s(p-1)}(x)$ equals
\begin{align*}
q^{\tfrac{p-1}{2}}(-1)^{\tfrac{s(p-1)}{2}}\sum_i (-1)^{\tfrac{r-i}{2}}
\binom{(i-s)\tfrac{(p-1)}{2}-1}{\tfrac{pi-r-s(p-1)}{2}}
Q_{(r+ps-pi)(p-1)}Q_{i(p-1)}(x).
\end{align*}

To prove the formula for a general $s$, let $k=2(1+p+\dots+p^t)-s$ for any $t$
large enough that $k\geq 0$.
Let $Y_k\in \Alg_{E_{\infty}}(\Mod_{\Fvp})$ be defined as in Lemma
\ref{lem:sphere}. In particular, let $y\in H_{\1_{\fG},-k}(Y_k)$ be such that
$Q_{k(p-1)}(y)\mapsto y$ under the map induced from $\gpower{\1_{\fG}}{p}\to
\1_{\fG}$, and $Q_{i(p-1)}(y)=0$ if $i\not=k$.
Define
\begin{align*}
Q_{r,s,g,n}\coloneqq Q_{r(p-1)}&Q_{s(p-1)}\\
-q^{\tfrac{p-1}{2}}&
(-1)^{\tfrac{s(p-1)}{2}} \sum_i (-1)^{\tfrac{r-i}{2}}
\binom{(i-s)\tfrac{p-1}{2}-1}{\tfrac{r-i}{2}-1} Q_{(r+ps-pi)(p-1)}Q_{i(p-1)}
\end{align*}
considered as an operation on $H_{g,n}$. (Recall here that
$q=(-1)^n\vp(\tr_{g\oplus g}(\b_{g,g}))$
depends on $n$ and $g$.)
Using the external Cartan
formulas for $X\dc Y_k$, we have that
under the isomorphism on homology induced by the composition
$$((g\oplus \1_{\fG})^{p})^p\to \gpower{(\gpower{g}{p}\oplus
\gpower{\1_{\fG}}{p})}{p}\to \gpower{(\gpower{g}{p}\oplus \1_{\fG})}{p}\to
\gpower{(\gpower{g}{p})}{p}\oplus \gpower{\1_{\fG}}{p}\to
\gpower{(\gpower{g}{p})}{p}\oplus \1_{\fG}$$
in $\fG$,
$$ Q_{r',s',g,n-k}(x\otimes y)\mapsto Q_{r,s,g,n}(x)\otimes y,$$
where $s'=2(1+\dots+p^t)$ and $r'=r+k$ (and as an intermediate step we have
re-indexed the sum with $i'=i+k$). Note then that $r'-s'=r-s>0$, so we have
already shown that $Q_{r',s',g,n-k}(x\otimes y)=0$. Therefore, since
$H_{*,*}(X\dc Y_k)\cong H_{*,*}(X)\dc H_{*,*}(Y_k)$ and $y\not=0$, we must have
$$Q_{r,s,g,n}(x)=0.$$
This proves the first Adem relation. 

To prove the other Adem relation, we replace $r$ by $(r-1+s(p-1))(p-1)$ and $s$
by $s(p-1)-1$, followed by a change of index $j=s-2k$, $i=r-1+s(p-1)-2\ell$ in
\ref{eq:final}. A simplification of constants yields the formula
\begin{align*}
\sum_j& (-1)^{\tfrac{s-j}{2}} \binom{\tfrac{j(p-1)}{2}-1}{\tfrac{s-j}{2}}
Q_{(r+ps-pj)(p-1)}\b Q_{j(p-1)}\\
=&\vp(\tr_{g\oplus g}(\b_{g,g}))^{\tfrac{p-1}{2}}\left(\tfrac{p-1}{2}!\right)\sum_i
(-1)^{\tfrac{r-1-i+r(p-1)}{2}}
\binom{\tfrac{i(p-1)}{2}}{\tfrac{r-1+s(p-1)-i}{2}} \b Q_{(r-1+ps-pi)(p-1)}
Q_{i(p-1)}\\
&- q^{\tfrac{p-1}{2}}\sum_i (-1)^{\tfrac{r-1+s(p-1)-i}{2}}
\binom{\tfrac{i(p-1)}{2}-1}{\tfrac{r-1+s(p-1)-i}{2}} Q_{(r+ps-pi)(p-1)} \b
Q_{i(p-1)},
\end{align*}
where we write, for instance, $Q_{2k}\b Q_{2k'}$ for $\theta(e_{2k}\otimes
e_{2k'-1}^p)\otimes \iota^{p^2}.$
Again we note that for each composition of Dyer--Lashof operations to be
nonzero, we will have that $r$ and $s$ have opposite parities, 
and $i$ and $j$ have the same parities
as $s$. This ensures that all of the fractions in the above formula are
integers.

Note that the binomial coefficients in the second term on the right-hand side are zero
unless
$$\tfrac{r-1+s(p-1)+2}{p}\leq i\leq r+s(p-1)-1,$$
and the binomial coefficients on the left-hand side are zero unless
$$\tfrac{s+2}{p}\leq j\leq s.$$
The assumption that $r\geq s$ ensures that
$$s<s+\tfrac{1}{p}= \tfrac{ps+1}{p}\leq \tfrac{r+1+s(p-1)}{p},$$
so no term can appear on both sides of the equation. 

Now suppose $s=2p^t$ for some $t>0$. Replacing $j$ by $2j$ (since we may assume
that $j$ and $s$ have the same parity), the binomial coefficient on the
left-hand side becomes
$$\binom{j(p-1)-1}{p^t-j},$$
which is 0 unless $j=p^t$ by Lemma \ref{lem:binom}. 
Moreover, replacing $i$ by $2i$, we have that the binomial coefficient in the
first term on the right-hand side becomes
$$\binom{i(p-1)}{\tfrac{r-1}{2}+p^t(p-1)-i}.$$
Since $\b Q_{(r-1+ps-2pi)(p-1)}=0$ unless $r-1+ps-2pi>0$, this implies that
$$pi-\tfrac{r-1}{2}-p^t(p-1)<p^t,$$
so Lemma \ref{lem:binom} implies that
$$\binom{i(p-1)}{\tfrac{r-1}{2}+p^t(p-1)-i}\equiv
\binom{(i-p^t)(p-1)}{\tfrac{(r-1)}{2}-i} \bmod{p}$$
when $(r-1)/2-i\geq 0$; when $(r-1)/2-i<0$, an elementary argument using Lucas'
Theorem shows that the left-hand side must be 0 (and the right-hand side is
clearly 0).
A similar argument for the second term on the right-hand side shows that
$$\binom{i(p-1)-1}{\tfrac{(r-1)}{2}+p^t(p-1)-i}\equiv
\binom{i(p-1)-p^t(p-1)-1}{\tfrac{r-1}{2}-i}\bmod{p}$$
when $r+ps-pi\geq 0$, and when $r+ps-pi<0$, $Q_{(r+ps-pi)(p-1)}$ vanishes.
Thus, we have proven that when $s=2p^t$,
\begin{align*}
Q_{r(p-1)}\b
Q_{s(p-1)}=q_{g,r}\sum_i&
t_{r,s,i}\tfrac{(i-s)(p-1)}{pi-s(p-1)-r+1}
\b Q_{(r-1+ps-pi)(p-1)} Q_{i(p-1)}\\
&-q^{\tfrac{p-1}{2}}(-1)^{\tfrac{s(p-1)}{2}}\sum_i t_{r,s,i} Q_{(r+ps-pi)(p-1)} \b
Q_{i(p-1)},
\end{align*}
where
$$q_{g,r}=\vp(\tr_{g\oplus
g}(\b_{g,g}))^{\tfrac{p-1}{2}}\left(\tfrac{p-1}{2}!\right)(-1)^{\tfrac{r(p-1)}{2}},$$
$q=\vp(\tr_{g\oplus g}(\b_{g,g}))(-1)^n$, and
$$t_{r,s,i}=(-1)^{\tfrac{r-i-1}{2}}\binom{\tfrac{(i-s)(p-1)}{2}-1}{\tfrac{r-1-i}{2}}.$$

To prove this for a general $s$, let
\begin{align*}
\b Q_{r,s,g,n}\coloneqq &Q_{r(p-1)}\b Q_{s(p-1)}-q_{g,r}\sum_i t_{r,s,i}
\tfrac{(i-s)(p-1)}{pi-s(p-1)-r+1}
\b Q_{(r-1+ps-pi)(p-1)} Q_{i(p-1)}\\
&+q^{\tfrac{p-1}{2}}(-1)^{\tfrac{s(p-1)}{2}}\sum_i 
t_{r,s,i} Q_{(r+ps-pi)(p-1)} \b 
Q_{i(p-1)}.
\end{align*}
Once again, let $Y_k$ and $y$ be as in Lemma \ref{lem:sphere}. The external
Cartan formulas show that
$$\b Q_{r',s',g,n-k}(x\otimes y)\mapsto (-1)^{\tfrac{k(p-1)}{2}}\b Q_{r,s,g,n}(x)\otimes y$$
under the map induced by the appropriate morphism $((g\oplus \1_{\fG})^p)^p\to
(g^p)^p\oplus \1_{\fG}$
for $r'=r+k$, $s'=s+k$. Choosing $k>0$ so that $s'=2p^t$ makes $$\b
Q_{r',s',g,n-k}(x\otimes y)=0,$$ and therefore $\b Q_{r,s,g,n}(x)=0$ as well.
This proves the second pair of Adem relations.
\end{proof}

\begin{proof}[Proof (of Theorem \ref{thm:lower index}, Statements
\ref{item:susp1} and \ref{item:susp2})]
Finally, we prove statement \ref{item:susp1} Let $X=\S^{g,n+1}\Fvp$, and let $Y$ be as
in Construction \ref{con:zigzag}. As usual, we write $\iota_n$ for $1\in (\O
X)(g)_n$ and $\iota_{n+1}$ for $1\in X(g)_{n+1}$.
From the definition of $\s$, we see that to show $$\s
Q_{i(p-1)}=(-1)^{\tfrac{n(p-1)}{2}}\left(\tfrac{p-1}{2}\right)!\,Q_{(i-1)(p-1)},$$
it suffices to find a chain $y\in (W\otimes_{C_p} Y^{\dc p})(\gpower{g}{p})$ whose image in
$(W\otimes_{C_p} X^{\dc p})(\gpower{g}{p})$ is
$$(-1)^{\tfrac{n(p-1)}{2}}\left(\tfrac{p-1}{2}\right)!\, e_{(i-1)(p-1)}\otimes
\iota_{n+1}^p$$
and such that $d(y)=e_{i(p-1)}\otimes \iota_n^p$.
Such a $y$ may be constructed exactly as in the proof of Theorem 3.1 in
\cite{Mayalg}, where the classes May calls $a$ and $b$ are replaced by
$\iota_{n+1}$ and $\iota_n$, respectively, and $y$ is the class May denotes by $c$
(for $j=i(p-1)$.) 

Similarly, to show that $$\s \b
Q_{i(p-1)}=-(-1)^{\tfrac{n(p-1)}{2}}\left(\tfrac{p-1}{2}\right)! \b
Q_{(i-1)(p-1)},$$
we use the construction for what May calls $c'$ to construct a chain $$y'\in
(W\otimes_{C_p} Y^{\dc p})(\gpower{g}{p})$$ whose image in $(W\otimes_{C_p}
X^{\dc p})(\gpower{g}{p})$
is $$(-1)^{\tfrac{n(p-1)}{2}}\left(\tfrac{p-1}{2}\right)!\, e_{(i-1)(p-1)-1}\otimes
\iota_{n+1}^p$$ and such that $d(y')=-e_{i(p-1)-1}\otimes \iota_n^p$.

Statement \ref{item:susp2} follows directly from \ref{item:susp1} and
commutativity of diagram \ref{eq:susp diagram}.
\end{proof}

\subsection{Upper indexing}\label{subsection:upper}
It is more common to re-index the Dyer-Lashof operations so that the operations
raise degree. This allows us to define a graded algebra of $\Fvp$ operations that acts
on the $\Fvp$-homology of an $E_{\infty}$-algebra. Since our homology is graded
by $(\ob{\fG})\times \Z$, we will have a double grading on our upper-indexed operations
as well.

\begin{defn}
For $s\in \Z $ and $g\in \fG$, let
$$Q^s_g\colon H_{g,n}\to H_{\gpower{g}{p},n+2s(p-1)}$$
be given by
$$Q^s_g=(-1)^s v(n) Q_{(2s-n)(p-1)},$$
where 
$$v(n)=(-1)^{\tfrac{n(n-1)(p-1)}{4}} [\left(\tfrac{p-1}{2}\right)!]^n,$$ and
$Q_{(2s-n)(p-1)}$ is a lower-indexed operation as defined in Definition
\ref{defn:lowerindex}.
When $2s-n<0$, $Q_{(2s-n)(p-1)}$ is defined to be zero. Similarly, if $s\in \Z$, let 
$$\b Q^s_g\colon H_{g,n}\to H_{\gpower{g}{p},n+2s(p-1)-1}$$
be given by
$$\b Q^s_g=(-1)^s v(n)\b Q_{(2s-n)(p-1)},$$
where $\b Q_{(2s-n)(p-1)}=0$ if $2s-n\leq 0$. We refer to the operations $Q^s_g$
and $\b Q^s_g$ as the \emph{classical} or \emph{untwisted} Dyer--Lashof
operations.

For $s\in\Z$, let
$$Q_g^{s+\tfrac12}\colon H_{g,n} \to H_{\gpower{g}{p},n+(2s+1)(p-1)}$$
be defined by $$Q_g^{s+\tfrac12}=(-1)^{s}[\left(\tfrac{p-1}{2}\right)!]^3 v(n) Q_{(2s+1-n)(p-1)},$$
and let
$$\b Q_g^{s+\tfrac12}\colon H_{g,n}\to H_{\gpower{g}{p},n+(2s+1)(p-1)-1}$$
be defined by $$\b Q_g^{s+\tfrac12}=(-1)^s [\left(\tfrac{p-1}{2}\right)!]^3 v(n)\b Q_{(2s+1-n)(p-1)}.$$
We refer to $Q^{s+\tfrac12}_g$ and $\b Q^{s+\tfrac12}_g$ as the \emph{twisted}
Dyer--Lashof operations.
\end{defn}

\begin{rems} This definition merits a number of remarks.
\begin{enumerate}
\item By Theorem \ref{thm:lower index}, the untwisted operations $Q^s_g$ on
$H_{g,n}$ are zero unless
$\vp(\tr_{g\oplus g}(\b_{g,g}))=1$, and the twisted operations
$Q^{s+\tfrac12}_g$ on $H_{g,n}$ are zero unless
$\vp(\tr_{g\oplus g}(\b_{g,g}))=-1$,
but it is convenient to note that both types of operations are defined for all
$g$.

\item For working with these constants, it is useful to note that
$$[\left(\tfrac{p-1}{2}\right)!]^2\equiv (-1)^{\tfrac{p-1}{2}+1}\bmod{p},$$
so that $\left(\tfrac{p-1}{2}\right)!$ is either $-1$ or a primitive 4th root of
unity, depending on whether $p$ is 1 or 3 mod $4$.

\item We will often omit the fractional notation, choosing instead to write $Q^s_g$,
where $s$ is either an integer or a half integer. We will write $\Z+\tfrac12$ to
denote the set of half integers.

\item
The untwisted operations are defined in precisely the same way as the classical
Dyer--Lashof operations (though now
keeping track of the $\fG$ grading). When $\vp$ is the trivial homomorphism, the choice of constant in this definition
ensures that
\begin{itemize}
\item if $2s=n$ and $x\in H_{g,n}(X)$, then $Q_g^s(x)=x^p$
\item suspension commutes with $Q_g^s$ (without the constant in \ref{item:susp2})
\item the signs in the Cartan formula cancel.
\end{itemize}

 Our choice ensures that the first two conditions hold for the twisted operations;
however, although
the signs in the Cartan formulas do simplify with our choice of constant, we
still have a $-1$ in some cases. 
This choice is also consistent with the customary choice of constant for the
untwisted operations, in that we are able to state the relations for twisted and
untwisted operations simultaneously.
Presumably there are other logical choices of constant, however.
\end{enumerate}
\end{rems}

Now we may rewrite Theorem \ref{thm:lower index} for the upper-indexed
operations. This theorem should look very similar to the classical theorem
for $E_{\infty}$-spaces in \cite{IteratedLoops}, although it lacks a statement
of the Nishida relations, since we do not necessarily have Steenrod operations
in our algebraic setting. A version of the Kudo transgression theorem will be
stated in Section \ref{section:ss}.

\begin{thm}\label{thm:upper index}
The Dyer--Lashof operations on $\Alg_{E_{\infty}}(\Mod_{\Fvp})$ consist of
natural transformations
\begin{align*}
Q^s_g & \colon H_{g,n} \to H_{\gpower{g}{p},n+2s(p-1)}\\
\b Q^s_g & \colon H_{g,n}\to H_{\gpower{g}{p}, n+2s(p-1)-1}
\end{align*}
for every $s\in \Z\sqcup \Z+\tfrac12$ satisfying the following relations.
\begin{enumerate}
\item If $\vp(\tr_{g\oplus g}(\b_{g,g}))=1$, then $Q^s_g=0$ whenever $s\in \Z+\tfrac12$; if
$\vp(\tr_{g\oplus g}(\b_{g,g}))=-1$, then $Q^s_g=0$ whenever $s\in \Z$. \label{item:twisted}
\item (Frobenius linearity) For $x,y\in
H_{g,n}(X)$ and $t\in \F$,
$$Q_g^s(t x+y)=t^p Q_g^s(x)+Q_g^s(y)$$ 
and
$$\b Q_g^s(t x+y)=t^p \b Q_g^s(x)+\b Q_g^s(y).$$ 

\item $Q_g^s=0$ on $H_{g,n}$ if $2s<n$.
\item $\b Q_g^s=0$ on $H_{g,n}$ if $2s\leq n$.
\item If $x\in H_{g,n}(X)$ with $n=2s$, then
$Q_g^s(x)=x^p$. \label{item:power}
\item If $\mathbbm{1}\in H_{\1_{\fG},0}(X)$ is the unit, then
$Q_{\1_{\fG}}^s(\mathbbm{1})=0$ and $\b Q_{\1_{\fG}}^s(\mathbbm{1})=0$ if $s>0$.
\item (Internal and external Cartan formulas) If $x\in
H_{g,n}(X)$ and $y\in H_{h,m}(X)$, then for $s,j,k\in \Z\sqcup \Z+\tfrac12$, the
shuffle map
$$H_{\gpower{(g\oplus h)}{p},*}(X)\to H_{\gpower{g}{p}\oplus\gpower{h}{p},*}(X)$$ takes
$$Q^s_{g\oplus h}(xy)\mapsto \sum_{j+k=s}
(-1)^{2jk(p-1)}Q^j_g(x) Q^k_h(y)$$
and
\begin{align*}
\b Q^s_{g\oplus h}(xy)\mapsto \sum_{j+k=s} (-1)^{2jk(p-1)} \left(\b
Q^j_g(x) Q^k_h(y)+(-1)^n Q^j_g(x)\b Q^k_h(y)\right).
\end{align*}
Similarly, if $x\in H_{g,n}(X)$ and $y\in H_{h,m}(Y)$, then for $s,j,k\in
\Z\sqcup \Z+\tfrac12$, the shuffle map
$$H_{\gpower{g\oplus h}{p},*}(X\dc Y)\to
H_{\gpower{g}{p}\oplus\gpower{h}{p},*}(X\dc Y)$$ takes
$$Q^s_{g\oplus h}(x\otimes y)\mapsto \sum_{j+k=s}
(-1)^{2jk(p-1)}Q^j_g(x)\otimes Q^k_h(y)$$
and
\begin{align*}
\b Q^s_{g\oplus h}(x\otimes y)\mapsto \sum_{j+k=s} (-1)^{2jk(p-1)} \left(\b
Q^j_g(x)\otimes Q^k_h(y)+(-1)^n Q^j_g(x)\otimes \b Q^k_h(y)\right).
\end{align*}

\item  (Adem relations) If $r>ps$, then
$$Q_{\gpower{g}{p}}^rQ_g^s = \vp(\tr_{g\oplus g}(\b_{g,g}))^{\tfrac{p-1}{2}}\sum_i (-1)^{r-i} \binom{(i-s)(p-1)-1}{r-(p-1)s-i-1}
Q_{\gpower{g}{p}}^{r+s-i}Q_g^i$$
and
$$\b Q_{\gpower{g}{p}}^rQ_g^s = \vp(\tr_{g\oplus g}(\b_{g,g}))^{\tfrac{p-1}{2}}\sum_i (-1)^{r-i} \binom{(i-s)(p-1)-1}{r-(p-1)s-i-1} \b
Q_{\gpower{g}{p}}^{r+s-i}Q_g^i;$$
if $r\geq ps$, then
\begin{align*}
Q_{\gpower{g}{p}}^r\b Q_g^s=&\sum_i (-1)^{r-i}\binom{(i-s)(p-1)}{r-(p-1)s-i} \b
Q_{\gpower{g}{p}}^{r+s-i}Q_g^i\\
&-\vp(\tr_{g\oplus g}(\b_{g,g}))^{\tfrac{p-1}{2}}\sum_i
(-1)^{r-i}\binom{(i-s)(p-1)-1}{r-(p-1)s-i} Q_{\gpower{g}{p}}^{r+s-i} \b Q_g^i
\end{align*}
and
$$\b Q_{\gpower{g}{p}}^r\b Q_g^s=-\vp(\tr_{g\oplus g}(\b_{g,g}))^{\tfrac{p-1}{2}}\sum_i (-1)^{r-i}\binom{(i-s)(p-1)-1}{r-(p-1)s-i} \b
Q_{\gpower{g}{p}}^{r+s-i} \b
Q_g^i.$$

\item If $X,\O X\in \Alg_{E_{\infty}}(\Mod_{\Fvp})$ with compatible $E_{\infty}$
structures, in the sense that the lower half of diagram \ref{eq:susp diagram}
commutes, then the suspension isomorphism $\S\colon H_{*,*}(\O X)\to
H_{*,*+1}(X)$ satisfies $$\S Q_g^s(x)=Q_g^s(\S x)$$ and $$\S \b Q_g^s(x)=-\b Q_g^s(\S
x).$$ \label{item:upper susp}

\end{enumerate}
\end{thm}

\begin{proof}
This follows directly from Theorem \ref{thm:lower index} by re-indexing the
operations and carefully simplifying constants.
\end{proof}

\begin{rems}
From Remark \ref{rem:product}, we know that $x^p=0$ if $x$ has bidegree $(g,n)$
with $(-1)^n\vp(\tr_{g\oplus g}(\b_{g,g}))=-1$, and from Theorem \ref{thm:lower index}, we know that
$Q_{\1_{\fG}}(x)=0$ if $(-1)^n\vp(\tr_{g\oplus g}(\b_{g,g}))=-1$. Therefore, statement \ref{item:power} is
trivially true in this case.

From Equation \ref{eq:trhomo}, we know that
$$\vp(\tr_{\gpower{g}{p}\oplus
\gpower{g}{p}}(\b_{\gpower{g}{p},\gpower{g}{p}}))=\vp(\tr_{g\oplus
g}(\b_{g,g}))^p=\vp(\tr_{g\oplus g}(\b_{g,g})),$$
where the second equality follows from the fact that $\vp(\tr_{g\oplus
g}(\b_{g,g}))=\pm 1$.
Thus, when considering compositions of upper-indexed operations
$Q^r_{\gpower{g}{p}}Q^s_g$, we may assume
that $r,s\in \Z$ if $\vp(\tr_{g\oplus g}(\b_{g,g}))=1$ and $r,s\in \Z+\tfrac12$ if
$\vp(\tr_{g\oplus g}(\b_{g,g}))=-1$. That is, the composition of an untwisted operation with a
twisted operation is always zero.
\end{rems}

\section{A basis for the homology of free
$E_{\infty}$-algebras}\label{section:basis}

Now we would like to give a complete description of the $\Fvp$-homology of free
$E_{\infty}$-algebras in terms of the $\Fvp$ Dyer--Lashof operations. This is
certainly useful for computations, but, in light of Proposition
\ref{prop:yoneda}, it also
gives a classification of \emph{all} operations. In particular, our result will
show that the Dyer--Lashof operations together with the product generate all
homology operations.

In this section we will compute the homology as a vector space by giving an
explicit basis. As is common with such calculations, it is easier to compute the
entire homology $H_{*,*}(E_{\infty}(X))$ at once, rather than computing
$H_{g,n}(E_{\infty}(X))$ for each $(g,n)$ separately. Thus, we will actually
describe $H_{*,*}(E_{\infty}(X))$ as a free commutative monoid object in
$\Mod_{\Fvp}(\GrVect^{\fC}_{\F})$.

Before stating and proving this result, we will rigorously define bases and free commutative algebras in this
setting and construct an analogue of the Serre spectral sequence that we use in
the proof.

\subsection{Quotients}
In this section and the next we will make frequent use of \emph{quotients}, which in this
context we define using a universal property or a colimit. Let $\Set^{\ob{\fG}\times
\Z}$ denote the category of $(\ob{\fG}\times \Z)$-graded sets, which can equivalently
be described as the category of functors from $\ob{\fG}\times \Z$ to $\Set$, where
$\ob{\fG}\times \Z$ denotes the category with objects the elements of $\ob{\fG}\times \Z$ and
no nonidentity morphisms. There is a free-forgetful adjunction
$$\Set^{\ob{\fG}\times \Z}\rightleftarrows \Mod_{\Fvp}(\GrVect^{\fC})$$
that takes a graded set $S$ to
$$\Fvp\{S\}\coloneqq \bigoplus_{x\in S} \S^{g,n}\Fvp\{x\},$$
where $\S^{g,n}\Fvp\{x\}$ denotes $\S^{g,n}\Fvp$ when $x$ has bidegree $|x|=(g,n)$.

\begin{defn}
Suppose we have a free-forgetful adjunction
$$L\colon \Mod_{\Fvp}(\GrVect^{\fC}) \rightleftarrows
\mathsf{D}$$
to a (pointed) category $\mathsf{D}$, where the left adjoint is denoted $L$.
If $V\in \mathsf{D}$, and we have a map $L(\Fvp\{S\}) \to V$ (which
is equivalent to having a map $S\to V$ of graded sets), then we define
the \emph{quotient of $V$ by the set $S$} to be an object $V/S\in
\mathsf{D}$ together with a map $V\to V/S$ such that the composition
$$L(\Fvp\{S\})\to V\to V/S$$
is zero, and satisfying the following universal property: if $V\to V'$ is
another map such that $L(\Fvp\{S\})\to V\to V'$ is zero, then $V\to V'$ factors
uniquely as
$$
\xymatrix{
V \ar[r] \ar[dr] & V/S \ar[d]\\
& V'.
}
$$
\end{defn}
To see that quotients exist, note that the pushout of the diagram
$$
\xymatrix{
L(\Fvp\{S\}) \ar[r] \ar[d] & V\\
L(0) & 
}
$$
satisfies this universal property. Such pushouts exist any time $\mathsf{D}$ can
be described as algebras over a sifted monad (see Lemma 3.5 of \cite{SOS}), which is the case in all of our
examples.

\subsection{Commutative $\Fvp$-algebras}
We now make sense of the free commutative algebra on a graded set in
$\Set^{\ob{\fG}\times \Z}$ and give a description of a basis for its underlying
object in $\Mod_{\Fvp}(\GrVect^{\fG})$.

Let $\Alg_{\Fvp}(\GrVect^{\fC})$ denote the category of monoid objects in
$\Mod_{\Fvp}(\GrVect^{\fC})$.
Consider the free-forgetful adjunction
$$\Mod_{\Fvp}(\GrVect^{\fC})\rightleftarrows \Alg_{\Fvp}(\GrVect^{\fC}),$$
where we may take the free functor to be given by the tensor algebra
$$V\mapsto T(V)\coloneqq \bigoplus_{k\geq 0} V^{\dc k}.$$
Let $\CAlg_{\Fvp}(\GrVect^{\fC})$ denote the full subcategory of
$\Alg_{\Fvp}(\GrVect^{\fC})$ of commutative monoid objects in
$\Mod_{\Fvp}(\GrVect^{\fC})$.
Then there is a functor
$$\Alg_{\Fvp}(\GrVect^{\fC})\to \CAlg_{\Fvp}(\GrVect^{\fC})$$
that is left adjoint to the inclusion functor, and that is defined by taking an
algebra
$V\in \Alg_{\Fvp}(\GrVect^{\fC})$ to the quotient of $V$ by the $(\ob{\fG}\times
\Z)$-graded set $S$ containing
$$v_1v_2 - (-1)^{n_1n_2} \b_{g_1, g_2}^{-1}(v_2v_1)$$
for every $v_i\in V(g_i)_{n_i}$, where $(g_i,n_i)\in \ob{\fG}\times \Z$. Here, $V(g_i)_{n_i}$ denotes the $n_i$-graded
component of $V(g_i)$, $v_iv_j$ denotes the image of $v_i\otimes v_j$ under
the monoid multiplication map, and $\b_{g_1, g_2}^{-1}(v_2v_1)$ denotes the
image of $v_2v_1$ under the map induced by the inverse of the braiding morphism
$\b_{g_1,g_2}\colon g_1\oplus g_2\to g_2\oplus g_1.$

That the quotient map is a left adjoint follows
from the universal property of the quotient: if $V'\in
\CAlg_{\Fvp}(\GrVect^{\fC})$, then every map $V\to V'$ in $\Alg_{\Fvp}(\GrVect^{\fC})$
sends $S$ to zero, since $V'$ is commutative. Thus, there is a unique map
$V/S\to V'$ in $\Alg_{\Fvp}(\GrVect^{\fC})$, which is then a map in
$\CAlg_{\Fvp}(\GrVect^{\fC})$ since $V/S$ and $V'$ are both commutative. This
gives a map
$$\Hom_{\Alg}(V,V')\to \Hom_{\CAlg}(V/S,V')$$
that is an inverse to the map given by precomposing with the quotient map $V\to
V/S$. 

\begin{defn}
Given $V\in \Mod_{\Fvp}(\GrVect^{\fC})$, the \emph{free commutative algebra on
$V$} is the image of $V$ under the composition
$$\Mod_{\Fvp}(\GrVect^{\fC})\to \Alg_{\Fvp}(\GrVect^{\fC})\to
\CAlg_{\Fvp}(\GrVect^{\fC}).$$

If $U\in \Set^{\ob{\fG}\times\Z}$, then the \emph{free commutative algebra on $U$},
denoted $\mathcal{S}(U)$, is
the free commutative algebra on $\Fvp\{U\}$.
\end{defn}

We can extend the notion of a basis to modules and algebras over $\Fvp$.
\begin{defn}
If $V\in \Mod_{\Fvp}(\GrVect^{\fC})$, then a \emph{basis} for $V$ is a set $S\in
\Set^{\ob{\fG}\times\Z}$ together with an isomorphism $\Fvp\{S\}\to V$. A basis
for
$V\in\Alg_{\Fvp}(\GrVect^{\fG})$ is a basis for the image of $V$ under the
forgetful functor to $\Mod_{\Fvp}(\GrVect^{\fG})$.
\end{defn}

It will be useful
to have a description of a basis for the free commutative algebra on a set, and
thus, we give the following description of quotients in
$\Alg_{\Fvp}(\GrVect^{\fG}).$

\begin{prop}\label{prop:pushout}
Let $V\in \Alg_{\Fvp}(\GrVect^{\fC})$ and $X\in \Mod_{\Fvp}(\GrVect^{\fC})$ with
a map $X\to V$ in $\Mod_{\Fvp}(\GrVect^{\fC})$. Then the pushout of
$$
\xymatrix{
V\dc X\dc V \ar[r] \ar[d] & V\\
0 &
}
$$
in $\Mod_{\Fvp}(\GrVect^{\fC})$ may be endowed with a monoid structure making it
the pushout of
$$
\xymatrix{
T(X) \ar[r] \ar[d] & V\\
T(0) &
}
$$
in $\Alg_{\Fvp}(\GrVect^{\fC})$.
\end{prop}

This proposition is a special case of a more general construction found in the
proof of Lemma 6.2 in \cite{SchwedeShipley}, so we omit the proof here.

Note that $\Set^{\ob{\fG}\times\Z}$ may also be equipped with a Day Convolution
monoidal structure using the Cartesian product of sets. Using the symmetry in
$\Set$ together with the (trivial) braiding on $\ob{\fG}\times \Z$, we obtain a
symmetric monoidal structure. We again have
a free-forgetful adjunction
$$\Set^{\ob{\fG}\times\Z}\rightleftarrows \Mon(\Set^{\ob{\fG}\times\Z})
\rightleftarrows\CMon(\Set^{\ob{\fG}\times\Z}),$$
where $\Mon$ denotes monoid objects and $\CMon$ commutative monoid objects.
The free functor $\Set^{\ob{\fG}\times\Z}\to \Mod_{\Fvp}(\GrVect^{\fC})$ is
(strong) monoidal, so it induces a functor
$$\Mon(\Set^{\ob{\fG}\times\Z})\to \Alg_{\Fvp}(\GrVect^{\fC})$$
that is left adjoint to the forgetful functor. In particular, the free
 algebra on a set $S\in \Set^{\ob{\fG}\times\Z}$ is isomorphic to $\Fvp\{T(S)\}$,
where $T(S)\in \Mon(\Set^{\ob{\fG}\times\Z})$ is the free monoid on $S$.

We can use this, together with Proposition \ref{prop:pushout}, to give a
noncanonical basis for the free commutative algebra on a set $U$. Give $U$ a
total ordering, and let $C(U)$ denote the elements of $T(U)$ of the form $s_1\cdots s_j$ where $s_1\leq \cdots
\leq s_j,$ and we may only have $s_i=s_{i+1}$ if $|s_i|=(g,n)$ where
$(-1)^n\vp(\tr_{g\oplus g}(\b_{g,g}))=1$. 

\begin{lem}\label{lem:free comm alg}
The set $C(U)$ is a basis for the free commutative algebra $\cS(U)$ on $U$.
\end{lem}

To see this, note that since $C(U)$ is a subset of $T(U)$, we have a map
$$\Fvp\{C(U)\}\hookrightarrow \Fvp\{T(U)\}\to \cS(U),$$
where the latter map is a quotient map. It is not hard to show that this map is
an isomorphism.  

If $X\in \Alg_{E_{\infty}}(\Mod_{\Fvp}(\Ch^{\fC}))$, then the product defined
in Section \ref{section:product} endows $H_{*,*}(X)$ with a commutative
monoid structure in $\Mod_{\Fvp}(\GrVect^{\fC})$.

\subsection{An analogue of the Serre spectral sequence}\label{section:ss}
In Construction \ref{con:zigzag}, we defined a filtration on $E_{\infty}(Y)$ for
$Y$ the mapping cone of the identity map  $X\to X$ for any $X\in
\Mod_{\Fvp}$ (in fact there we used the identity map $\O X\to
\O X$, but of course we may instead work with $X$).
In the next
section, we will make
use of the spectral sequence associated to this filtration, so we here prove
several useful facts about it.

For the remainder of this section, we replace the category $\Ch_{\F}^{\fC}$ by
the equivalent category $\Ch(\Vect^{\fC}_{\F})$ of $\Z$-graded chain complexes in
the abelian category $\Vect^{\fC}_{\F}$. We again write $\Fvp$ for our
coefficient system in this category, and we consider the category
$$\Mod_{\Fvp}(\Ch(\Vect^{\fC}_{\F}))$$
of modules over $\Fvp$ in $\Ch(\Vect^{\fC}_{\F})$. The monoidal structure in
this category is again written as $\dc$. Given $X\in
\Mod_{\Fvp}(\Ch(\Vect^{\fC}_{\F}))$, we write $H_n(X)$ to mean the functor
$\fC\to \Vect_{\F}$ that when evaluated on $g\in \fC$ is given by
$H_n(X(g))=H_{g,n}(X)$. Similarly, we replace the category $\GrVect^{\fC}$ by
the equivalent category $(\Vect^{\fC})^{\Z}$.

Our filtration on $E_{\infty}(Y)$ actually consists of filtrations on
$\Sym^k(Y)$ for each $k$, and for now we consider each of these separately.
 The filtration on
$\Sym^{k}(Y)$ is bounded, and hence the associated spectral sequence converges
to
$H_*(\Sym^{k}(Y))$. From Lemma \ref{lem:preservewes}, we know that $H_*(\Sym^{k}(Y))=0$ when $k\geq 1$.
Taking the coproduct over $k$, we obtain a charge-graded
spectral sequence $E$ converging to $H_{*}(E_{\infty}(Y))$.
We write $\arity{E}{k}$ to denote the charge $k$ component of this spectral
sequence; that is, the spectral sequence obtained from the filtration on
$\Sym^k(Y)$.

\begin{lem}\label{lem:E1}
The $E^1$ page of the charge $k$ spectral sequence is isomorphic to
$$\arity{E^1_{t,q}}{k}\cong \bigoplus_r H_{t+q-r}(\Sym^t(\S X))\dc H_{r}(\Sym^{k-t}(X))$$
in $\Mod_{\Fvp}(\Vect_{\F}^{\fC})$.
\end{lem}

Note the unusual bigrading here: the first grading corresponds to the
\emph{charge} grading of $E_{\infty}(\S X)$ and the second to the homological
degree of the tensor product. See Figure \ref{figure:ss} for an idea of the
shape of this spectral sequence.

\begin{figure}\label{figure:ss}
\centering
\begin{tikzpicture}
\draw[->] (-1.5,0)--(4.5,0) node[below right] {$t$};
\draw[->] (0,-3.5)--(0,3.5) node[above left] {$q$};
\foreach \Point in {(0,-3), (0,-2), (0,-1), (0,0), (0,1), (0,2), (0,3), (1,-3),
(1,-2), (1,-1), (1,0), (1,1), (1,2),
(1,3), (2,-3), (2,-2), (2,-1), (2,0), (2,1), (2,2), (2,3), (3,-3),(3,-2), (3,-1), (3,0),
(3,1), (3,2), (3,3)}{\node at \Point {\textbullet};}
\end{tikzpicture}
\caption{Picture of $\arity{E^1_{t,q}}{3}$. Dots
indicate potentially nonzero classes. In particular, all nonzero classes occur
for $0\leq t\leq 3$. A dot at $(0,q)$ corresponds to $H_{q}(\Sym^3(X))$, and
a dot at $(3,q)$ corresponds to $H_{3+q}(\Sym^3(\S X))$.}
\end{figure}
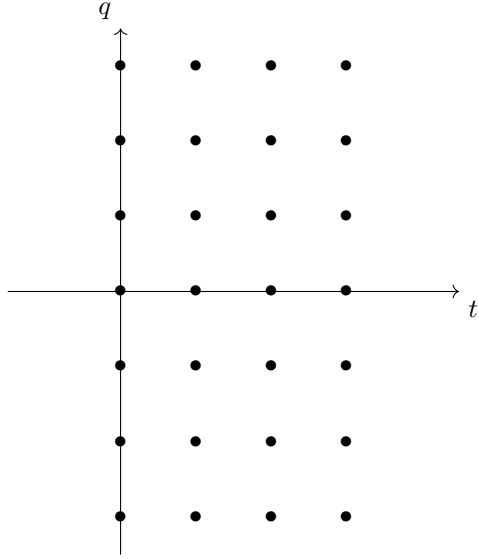

\begin{proof}
Note first that if $\gr$ denotes
associated graded, then  $\gr{Y}\cong X\oplus \S X$, so by Proposition
\ref{prop:sum to prod}, 
$$H_{*}(E_{\infty}(\gr{Y}))\cong H_{*}(E_{\infty}(X))\dc
H_{*}(E_{\infty}(\S X)).$$
Moreover,
$$\gr{\Sym^k(Y)}\cong \Sym^k(\gr{Y}),$$ 
where $F_i\Sym^k(Y)/F_{i-1}\Sym^k(Y)$ corresponds to the component of
$\Sym^k(\gr{Y})$ with the $F_1{Y}/F_0{Y}\cong \S X$ factor repeated $i$ times.
Taking homology, we obtain
\begin{align*}
\arity{E^1_{t,q}}{k}&\cong H_{t+q}(F_t\Sym^k(Y)/F_{t-1}\Sym^k(Y))\\
&\cong H_{t+q}(\Sym^t(\S X)\dc \Sym^{k-t}(X))\\
&\cong H_{t+q-*}(\Sym^t(\S X))\dc H_*(\Sym^{k-t}(X)),
\end{align*}
which gives the desired description of the $E^1$ page of the
spectral sequence.
\end{proof}

As usual, by
choosing a chain representative of the product, we obtain a map
$$E_{\infty}(Y)\dc E_{\infty}(Y)\to E_{\infty}(Y)$$
that in particular sends
$$F_i\Sym^k(Y) \dc F_j\Sym^{\ell}(Y)\to F_{i+j}\Sym^{k+\ell}(Y).$$
This gives the spectral sequence $E$ a multiplicative structure.
Our argument above shows that the associated graded $E^0$ is isomoprhic to
$E_{\infty}(\S X\oplus X)$, and, moreover, that the multiplication on $E^0$ is
compatible under this isomorphism with the chain-level multiplication on
$E_{\infty}(\S X\oplus X)$.
Since the weak equivalence $$E_{\infty}(\S X\oplus X)\to E_{\infty}(\S X)\dc
E_{\infty}(X)$$
of Proposition \ref{prop:sum to prod}
is a map of $E_{\infty}$-algebras, this map preserves the product on homology by naturality
of the operations. Recall that we may describe the product on the tensor product
$E_{\infty}(\S X)\dc E_{\infty}(X)$
as
$$
\begin{tikzcd}
H_*(E_{\infty}(\S X))\dc H_*(E_{\infty}(X))\dc H_*(E_{\infty}(\S X))\dc
H_*(E_{\infty}(X))\arrow{d}\\
H_*(E_{\infty}(\S X))\dc H_*(E_{\infty}(\S X))\dc H_*(E_{\infty}(X))\dc
H_*(E_{\infty}(X))\arrow{d}\\
H_*(E_{\infty}(\S X))\dc H_*(E_{\infty}(X)),
\end{tikzcd}
$$
where in the last step we applied the product to each factor separately.
Thus, the multiplicative structure on the $E^1$ page of our spectral sequence is
compatible with the above multiplication, and, in particular, this
multiplication satisfies the Leibniz rule.

We complete this section by giving a version of the Kudo
transgression theorem \cite{Kudo}. 
By Lemma \ref{lem:E1}, we know that
$$\arity{E^1_{k,*}}{k}\cong H_{k+*}(\Sym^k(\S X)),$$
and $\arity{E^1_{t,*}}{k}=0$ for $t>k$. 

\begin{defn}
A class $x\in H_{k+*}(\Sym^k(\S X))$ \emph{transgresses} to $y\in
H_{*+k-1}(\Sym^k X)$ if $d^r(x)=0$ when $r<k$ and $d^k(x)=y$.
\end{defn}

This makes sense, for $\arity{E^1_{0,*}}{k}\cong H_{*}(\Sym^k(X))$, and
$\arity{E^1_{t,*}}{k}=0$ if $t<0$, so $\arity{E^r_{0,*}}{k}$ is a quotient of
$H_*(\Sym^k(X))$ for each $r$. Since we are working over a field, the quotient
map splits canonically, and therefore we may think of $y$ as an element of
$H_{*+k-1}(\Sym^k X)$.
Now we may state our analogue of the transgression theorem. The proof is quite
similar to that of Theorem 3.4 in \cite{Mayalg}.

\begin{prop}\label{prop:Kudo}
The spectral sequence of Lemma \ref{lem:E1} has the following transgressive
classes.
\begin{enumerate}
\item If $x\in H_*(X)$, then $\S x$ transgresses to $x$. 
\item If $x\in H_*(\Sym^k(\S X))$ transgresses to $y$, then $Q^s x\in
H_*(\Sym^{pk}(\S X))$ transgresses to $Q^s y$, and $\b Q^s x$ transgresses to
$-\b Q^s y$ for any $s\in \Z\sqcup \Z+\tfrac12$.
\item If $x\in H_n(\Sym^k(\S X))(g)$ transgresses to 
$y$, and $\vp(\tr_{g\oplus g}(\b_{g,g}))(-1)^n=1$, then $x^{p-1}\otimes y$ ``transgresses'' to $-\b
Q^{\tfrac{n}{2}}_g(y)$. That is, $x^{p-1}\otimes y$ survives until the $E^{k(p-1)}$
page, and $d^{k(p-1)}(x^{p-1}\otimes y)=-\b Q^{\tfrac{n}{2}}_g(y)$.
\end{enumerate}
\end{prop}

\begin{proof}
By definition, the differentials in the spectral sequence come directly from the differential on
$E_{\infty}(Y)$. In particular, a transgression may be thought of as a
connecting homomorphism: as in Construction \ref{con:zigzag}, given $x\in
H_*(\Sym^k(\S X))$, we may choose a cycle representing it, choose a lift in
$\Sym^k(Y)$, and apply the differential to obtain a map
$$H_*(\Sym^k(\S X))\to H_{*-1}(F_{k-1}\Sym^k(Y)).$$ With this description, we
see that
$x$ transgresses precisely when there exists a lift $z\in \Sym^k(Y)$ of a cycle
representing $x$ such that $d(z)\in F_0\Sym^k(Y)\cong \Sym^k(X)$.

Thus, we can see from the definition of the suspension map from
\ref{con:zigzag}, $$\S\colon H_*(\Sym^k(X))\to
H_{*+1}(\Sym^k(\S X)),$$ that transgression is inverse to suspension when
restricted appropriately. That
is, the image of $\S$ is precisely the transgressive elements in
$H_{*+1}(\Sym^k(\S X))$, and
if $\S y=x$, then $x$ transgresses to the image of $y$ in
$\arity{E^k_{0,*}}{k}$. Similarly, if $x$ transgresses to $y$, then any lift
of $y\in H_*(\Sym^k(X))$ suspends to $x$.

It then follows that $\S x$ transgresses to $x$ for any $x\in
H_*(X)$. More generally, if $x$ transgresses to $y$, then by Theorem
\ref{thm:upper index} \ref{item:upper susp}, we have that
$$Q^s_g(x)=\S Q^s_g(y)$$
transgresses to $Q^s_g(y)$ and $\b Q^s_g(x)$ transgresses to $-\b Q^s_g(y)$.

The statement about $x^{p-1}\otimes y$ may be proven using a chain-level
construction as in the proof of Theorem 3.4 in \cite{Mayalg}. 
\end{proof}

\subsection{Admissible sequences}\label{section:admissible}
Now we wish to define certain compositions of
Dyer--Lashof operations that will give a basis for the homology of free
$E_{\infty}$-algebras as objects in $\CAlg_{\Fvp}(\GrVect^{\fC})$.

Let us consider tuples $$I=(\e_1,s_1,\dots,\e_k,s_k,g),$$
where $\e_i$ is either 0 or 1, $s_i\in \Z \sqcup (\Z+\tfrac12)$ (that is, $s_i$
is either an integer or a half integer), and $g\in \ob{\fG}$. Any such $I$ determines
a word in the Dyer--Lashof operations
$$Q^I_g\coloneqq \b^{\e_1} Q^{s_1}_{\gpower{g}{p^{k-1}}}\dots \b^{\e_k} Q^{s_k}_{g},$$
where by abuse of notation, $\b^0 Q^s_g$ denotes $Q^s_g$, and $\b^1 Q^s_g$
denotes $\b Q^s_g$.
We say such a sequence $I$ is \emph{admissible} if both of the following
conditions hold:
\begin{enumerate}
\item 
if $\vp(\tr_{g\oplus g}(\b_{g,g}))=1$, then $s_i\in\Z$ for all $1\leq i\leq k$; if $\vp(\tr_{g\oplus
g}(\b_{g,g}))=-1$, then $s_i\in
(\Z+\tfrac12)$ for all $1\leq i\leq k$
\item $ps_i-\e_i\geq s_{i-1}$ for all $1<i\leq k$.
\end{enumerate}
The empty sequence is also admissible, and it determines the identity operation.

Note that the first condition ensures that $Q^I_g$ is a composition of only
untwisted (respectively twisted) operations if $\vp(\tr_{g\oplus g}(\b_{g,g}))=1$ (respectively
$\vp(\tr_{g\oplus g}(\b_{g,g}))=-1$); since the composition of an untwisted and a twisted operation
is always zero, this condition avoids those $Q^I_g$ that are obviously zero for
this reason. The second condition ensures that $Q^I_g$ may not be rewritten
using an Adem relation.

Let the \emph{excess} of a sequence $I=(\e_1,s_1,\dots,\e_k,s_k,g)$ be the integer
$$e(I)\coloneqq 2s_1-\e_1-\sum_{i=2}^k(2s_i(p-1)-\e_i),$$
and the \emph{length} $$\ell(I)=k.$$
The associated operation $Q^I_g$ carries an $(\ob{\fG}\times\Z)$-grading, where
$$|Q^I_g|=(d_{\fG}(I),d_{\Z}(I))\coloneqq \left(\gpower{g}{(p^k-1)}, \sum_{i=1}^k (2s_i(p-1)-\e_i)\right)\in
\ob{\fG}\times \Z.$$
The empty sequence has excess $\infty$, length 0, and bidegree $(\1_{\fG},0)$.

If $I=(\e_1,s_1,\dots,\e_k,s_k,g)$, then we write
\begin{equation}\label{eq:subtuple}
I_{\geq j}=\begin{cases}
(\e_j,s_j,\dots,\e_k,s_k,g) & 1\leq j\leq k\\
\emptyset & j>k.
\end{cases}
\end{equation}
By definition, $I_{\geq j}$ is admissible if $I$ is.

The following elementary properties will be useful in manipulating this definition.

\begin{lem}\label{lem:admprops}
Let $I$ be admissible, and let $x$ be an element of an $(\ob{\fG}\times \Z)$-graded set $U$ with
bidegree $(g,n)$.
\begin{enumerate}
\item If $e(I)+\e_1>n$, then $d_{\Z}(Q^I_g x)\geq p^{\ell(I)}n,$ with equality
if and only if $I=\emptyset$. 
\label{list:bound}
\item For all $j\geq 1$, $e(I)+\e_1\leq e(I_{\geq j})+\e_j.$ 
\label{list:ugh}
\item If $e(I)+\e_1=n$, then
$$I=(\e_1,p^{i-1}s_i,0,p^{i-2}s_i,\dots,0,ps_{i},0,s_i,\e_{i+1},s_{i+1},\dots,\e_k,s_k,g),$$
where $e(I_{\geq i+1})+\e_{i+1}>n$ and $2s_i=n+d_{\Z}(I_{\geq i+1})$. Here
$I_{\geq i+1}$ may be empty, and if $i=1$, then the expression above becomes
$$I=(\e_1,s_1,\e_2,s_2,\dots,\e_k,s_k,g).$$ \label{list:powers}
\end{enumerate}
\end{lem}

\begin{proof}
We prove \ref{list:bound} by induction on the length of $I$. If $I=\emptyset$,
then $Q^I_gx=x$, so it is trivially true that $d_{\Z}(Q^I_gx)=d_{\Z}(x)$.
Now suppose that $d_{\Z}(Q^I_gx)\geq p^{\ell(I)}n$ for $0\leq \ell(I)\leq k-1$, and consider $I$
with $\ell(I)=k$. Note that
$$d_{\Z}(Q^I_g x)=n+\sum_{j=1}^k(2s_j(p-1)-\e_j)=2s_1(p-1)-\e_1+d_{\Z}(Q^{I_{\geq
2}}_g x).$$
The assumption that $e(I)+\e_1\geq n+1$ implies that
$$2s_1\geq n+1+\sum_{j=2}^k(2s_j(p-1)-\e_j)=d_{\Z}(Q^{I_{\geq 2}}_g x)+1.$$
Combining these inequalities, we have that
$$d_{\Z}(Q^I_g x)\geq p\, d_{\Z}(Q^{I_{\geq 2}}_g x)+p-1-\e_1>
p\,d_{\Z}(Q^{I_{\geq 2}}_g x).$$
By assumption, $d_{\Z}(Q^{I_{\geq
2}}_g x)\geq p^{k-1}n$, so this implies that $d_{\Z}(Q^I_gx)>p^kn$ for all
$k>0$.

To prove \ref{list:ugh}, we use the inequality $s_1\leq ps_2-\e_2$ from
admissibility of $I$ to see that
$$e(I)+\e_1=e(I_{\geq 2})+2s_1-2(ps_2-\e_2)\leq e(I_{\geq 2})\leq e(I_{\geq
2})+\e_2.$$ The statement then follows by induction using the same argument.

We prove \ref{list:powers} by induction on $\ell(I)$. If $k=1$, then
$e(I)+\e_1=n$ implies directly that $2s_1=n$, so there is nothing to prove.
In general, if $e(I_{\geq 2})+\e_2>n$, then $I$ has the desired form for $i=1$,
since $e(I)+\e_1=2s_1-d_{\Z}(I_{\ge 2})$ by definition of excess.
If $e(I_{\geq 2})+\e_2=n$, then 
$$n=e(I)+\e_1=e(I_{\geq 2})+2s_1-2(ps_2-\e_2)\leq e(I_{\geq 2})+\e_2=n,$$
so we must have that $$2s_1-2(ps_2-\e_2)=\e_2.$$ However, the left-hand side is
nonpositive by admissibility and $\e_2\in \{0,1\}$, so it follows that $\e_2=0$, and therefore
$s_1=ps_2$. Applying the induction hypothesis to $I_{\geq 2}$ gives the desired
result.
\end{proof}

\subsection{A basis for homology}\label{subsect:basis}

Suppose $\myset\in \Set^{\ob{\fG}\times\Z}$. We will make frequent use of the
$(\ob{\fG}\times\Z)$-graded set of formal symbols 
\begin{equation}\label{eq:Q}
Q(\myset)\coloneqq \{Q^I_g x\mid x\in \myset{(g,n)}, I\text{ admissible, }
e(I)+\e_1>n\}, 
\end{equation}
where $Q^I_g x$ has bidegree $$|Q^I_gx|=(\gpower{g}{p^{\ell(I)}},n+d_{\Z}(I)).$$

\begin{defn}
Suppose $\myset$ is a basis for $V\in \Mod_{\Fvp}(\GrVect^{\fC})$. Let
$\W{\myset}{V}$ in $\CAlg_{\Fvp}(\GrVect^{\fC})$ be the free commutative algebra on the
graded set $Q(\myset)$. In other words, $$\W{\myset}{V}=\cS(Q(\myset)).$$
\end{defn}

If $X\in \Mod_{\Fvp}(\Ch^{\fC})$ and $S$ is a choice of basis for
$H_{*,*}(X)$, then there is a map of sets
$$Q(S)\to H_{*,*}(E_{\infty}(X))$$
that sends $x\in S$ to $H_{*,*}(E_{\infty}(X))$ using the canonical monomorphism $$H_{*,*}(X)\to H_{*,*}(E_{\infty}(X)),$$
and
that sends $Q^I_g x$ to the corresponding Dyer--Lashof operation applied to $x$.
This induces a map
$$\Psi_X\colon  \W{S}{H_{*,*}(X)}\to H_{*,*}(E_{\infty}(X)),$$
which we claim is an isomorphism. 

\begin{thm}\label{thm:W}
If $X\in\Mod_{\Fvp}(\Ch^{\fC}_{\F})$, and if $S$ is any basis for $H_{*,*}(X)$,
then the map
$$\Psi_X\colon  \W{S}{H_{*,*}(X)}\to H_{*,*}(E_{\infty}(X))$$
is an isomorphism in $\CAlg_{\Fvp}(\GrVect^{\fC})$.
\end{thm}

Before beginning the proof, we require a few additional definitions. Recall that $H_{*,*}(E_{\infty}(X))$ carries an
extra $\N$-grading, the \emph{charge} grading.  We will endow $\W{S}{H_{*,*}(X)}$ with a
charge grading as well and make use of this in the proof.

In general, if $V\in \Mod_{\Fvp}(\GrVect^{\fC})$ carries a charge grading, then we may endow
the free commutative algebra on $V$ with a charge grading as follows. First note
that we may think of this extra grading as a functor
$$\arity{V}{-}\colon \N\to \Mod_{\Fvp}(\GrVect^{\fC}),$$
where $\N$ denotes the category with objects the natural numbers and no
nonidentity morphisms, and such that
$$V\cong\bigoplus_{k\in\N}\arity{V}{k}.$$
Such functors inherit a monoidal structure through Day Convolution, which
here explicitly means that
$$\arity{(V\dc V')}{k}\cong \bigoplus_{i+j=k} \arity{V}{i}\dc \arity{V'}{j}.$$
In particular, the tensor algebra
$T(V)$ inherits a charge grading
$$\arity{T(V)}{k}=\bigoplus_{j\geq 0}\arity{V^{\dc j}}{k}.$$
Here, we define $V^{\dc 0}$ to be concentrated in charge 0.
This is compatible with the monoid structure on $T(V)$, in the sense that we may
view $T(V)$ as a monoid object in the functor category
$\Mod_{\Fvp}(\GrVect^{\fC})^{\N}.$

The free commutative algebra on $V$ is formed by taking a quotient of $T(V)$ by
a set $U$. Note that if the set $U$ \emph{also} carries a charge grading, and if
this grading is compatible with the charge grading on $T(V)$, then the quotient
$T(V)/U$ inherits a charge grading.  
In the case of the free commutative algebra on $V$, every element of $U$ is of the form
$$[v_1,v_2]\coloneqq v_1v_2-(-1)^{n_1n_2}\b_{g_1,g_2}^{-1}(v_2v_1)\in (V\dc
V)(g_1\oplus g_2)_{n_1+n_2}.$$
If each $v_i\in V(g_i)_{n_i}$ is homogeneous of charge $k_i$, then $[v_1,v_2]$
is homogeneous of charge $k_1+k_2$, and hence we can give $U$ a charge grading
coming from that of $V$. Thus, the free commutative algebra on $V$ inherits a
charge grading.

In particular, this allows us
to endow $\W{S}{H_{*,*}(X)}$ with a charge grading: we say an element  $Q^I_g x$ has charge $p^{\ell(I)}$, so that $Q(S)$ has a charge grading, which
then induces one on the free commutative algebra $\W{S}{H_{*,*}(X)}$ on $Q(S)$. 

In general we use the notation $\arity{}{k}$ to denote charge $k$. In particular, we have
that
$$\arity{H_{*,*}(E_{\infty}(X))}{k}=H_{*,*}(\Sym^k(X)).$$

We require one additional technical lemma before the proof of Theorem
\ref{thm:W}.

If $y$ is an element of an $(\ob{\fG}\times\Z)$-graded set $\myset$ with
bigrading $(g,n)$, let
$$q(y)\coloneqq (-1)^n\vp(\tr_{g\oplus g}(\b_{g,g})).$$
We write 
$$Q(\myset)^{-}\coloneqq \{y\in Q(\myset)\mid q(y)=-1\} \text{ and }
Q(\myset)^{+}\coloneqq \{y\in Q(\myset)\mid q(y)=1\},$$
so that $$Q(\myset)=Q(\myset)^{-}\sqcup Q(\myset)^{+}.$$

\begin{lem}\label{lem:SSpieces} 
Let $X\in \Mod_{\Fvp}(\GrVect^{\fC})$, and let $S$ be a basis for $X$.
Then there is an isomorphism in $(\Vect_{\F}^{\fC})^{\Z}$ of
$$\W{\S S}{\S X}\dc \W{S}{X}$$ 
with
\begin{equation}\label{eq:basis1}
\left(\bigdc_{Q^I\S x\in Q(\S S)^{-}} \cS(Q^I\S x)\dc
\cS(Q^Ix)\right)\dc \bigdc_{Q^I\S x\in Q(\S S)^+}M,
\end{equation}
where 
$$M=\bigdc_{i\geq 0} 
\Fvp\{(Q^I\S x)^{jp^i}\}_{0\leq j\leq
p-1} \dc \cS(Q^{\op{i-1}{s}{g}{0}}Q^Ix)\dc
\cS(Q^{\op{i}{s}{g}{1}}Q^I x),$$
for
$$\op{i}{s}{g}{\e}\coloneqq \begin{cases}
(\e,p^is,0,p^{i-1}s,\dots,0,ps,0,s,g) & i\geq 0\\
\emptyset & i=-1,
\end{cases}$$
and $(g,2s)$ the bidegree of $Q^I\S x$.
\end{lem}

\begin{proof}
To see that such an isomorphism exists, note first that
$$\W{\S S}{\S X}\cong \bigdc_{y\in Q(\S S)^-} \cS(y)\dc \bigdc_{y\in Q(\S S)^+} \cS(y).$$
If $y\in Q(\S S)^+$, then $y^2\not=0$, and thus $\{y^i\}_{i\geq 0}$ is a basis for $\cS(y)$; that is, we
have an isomorphism in $(\Vect^{\fC})^{\Z}$ 
$$\Fvp\{y^i\}_{i\geq 0}\cong \cS(y).$$ Any power $y^i$ with $i>0$ may be written uniquely
as a distinct product of elements of
$$\{y,y^2,\dots,y^{p-1},y^p,y^{2p},\dots,y^{(p-1)p},\dots,y^{p^j},y^{2p^j},\dots,y^{(p-1)p^j}\}$$
(this essentially corresponds to writing $i$ in base $p$).
Hence, we have that
$$\bigdc_{i\geq 0} \Fvp\{y^{0p^i},y^{p^i},\dots,y^{(p-1)p^i}\} \cong
\Fvp\{y^i\}_{i\geq 0}$$
in $(\Vect^{\fC})^{\Z}$, and thus, we have that
$$\bigdc_{i\geq 0} \Fvp\{y^{0p^i},y^{p^i},\dots,y^{(p-1)p^i}\}\cong \cS(y)$$
in $(\Vect^{\fC})^{\Z}$. Thus, we have shown that
\begin{equation}\label{eq:basis2}
\W{\S S}{\S X}\cong \bigdc_{y\in Q(\S S)^-} \cS(y) \dc \bigdc_{y\in Q(\S S)^+}
\bigdc_{i\geq 0}\Fvp\{y^{0p^i},y^{p^i},\dots,y^{(p-1)p^i}\} 
\end{equation}
in $(\Vect^{\fC})^{\Z}$.

Now,
\begin{equation}\label{eq:basis3}
\W{S}{X}\cong \bigdc_{Q^I\S x\in Q(\S S)} \cS(Q^Ix)\dc \bigdc_{\substack{Q^Ix\in
Q(S)\\
e(I)+\e_1=d_{\Z}(\S x)}} \cS(Q^I x).
\end{equation}

We claim that if $Q^Ix\in Q(S)$ but $Q^I\S x\not\in Q(\S S)$, then  
$$Q^Ix=\b^{\e_1}Q^{p^js}_{\gpower{g}{p^j}}Q^{p^{j-1}s}_{\gpower{g}{p^{j-1}}}\dots
Q^{ps}_{\gpower{g}{p}}Q^s_g Q^{I'}x,$$
where $Q^{I'}\S x\in Q(\S S)^+$, $g=d_{\fG}(Q^{I'}\S x)$, and $2s=d_{\Z}(Q^{I'}\S x).$
Since $Q^I \S x\not\in Q(\S S)$, but $Q^I x\in Q(S)$, we must have that
$e(I)+\e_1=d_{
\Z}(\S x)$. By Lemma \ref{lem:admprops}, \ref{list:powers}, $I$ must have the
desired form. To see that $Q^{I'}\S x\in Q(\S S)^{+}$, note that
$2s=d_{\Z}(Q^{I'}\S x)$ is only possible if $q(Q^{I'}\S x)=1$, since
the sign of $\vp(\tr_{g\oplus g}(\b_{g,g}))$ determines whether or not $2s$ is
even or odd, and in particular, $\vp(\tr_{g\oplus g}(\b_{g,g}))$ and $(-1)^{2s}$
will always be equal.

We write
$$\op{i}{s}{g}{\e}\coloneqq
(\e,p^is,0,p^{i-1}s,\dots,0,ps,0,s,g)$$
so that we have shown
$$Q^Ix=Q^{\op{j}{s}{g}{\e_1}}Q^{I'}x,$$
where $2s=d_{\Z}(Q^{I'}\S x)$ and $g=d_{\fG}(Q^{I'}\S x)$.

Hence, any $Q^Ix$ with $e(I)+\e_1=d_{\Z}(\S x)$ may be
written uniquely as $Q^{\op{i}{s}{g}{\e}}Q^{I'}x$, where $Q^{I'}\S x\in Q(\S
S)^+$, and when $2s=d_{\Z}(Q^{I'}\S x)$,
$Q^{\op{i}{s}{g}{\e}}Q^{I'}x\in Q(S)$. Thus,
there is an isomorphism
$$\bigdc_{Q^I\S x\in Q(\S S)^+} \bigdc_{i\geq 0}
\cS(Q^{\op{i}{s}{g}{0}}Q^Ix)\dc \cS(Q^{\op{i}{s}{g}{1}}Q^I x)\to \bigdc_{\substack{Q^Ix\in
Q(S)\\
e(I)+\e_1=d_{\Z}(\S x)}} \cS(Q^I x).$$
Combining this with \ref{eq:basis2} and \ref{eq:basis3} proves \ref{eq:basis1}.

\end{proof}

\begin{rem} 
Note that here we choose
$\Fvp\{y^{j}\}_{0\leq j\leq p-1}$ rather than the quotient
$\cS(y)/y^p$ of $\cS(y)$ by the set $\{y^p\}$, since the multiplication on
$\cS(y)/y^p$ does not preserve charge.
\end{rem}

Now we are ready to prove Theorem \ref{thm:W}.

\begin{proof}
It is easy to see that the map
$\Psi_X$
preserves charge, since if $$x\in H_{g,*}(X)\cong H_{g,*}(\Sym^1(X)),$$ then
$Q^s_g(x)\in H_{*,*}(\Sym^p(X))$. More generally, $$Q^I_g(x)\in
H_{*,*}(\Sym^{p^{\ell(I)}}(X)).$$
Thus, it suffices to show that the maps
$$\arity{\Psi_X}{k}\colon \arity{\W{S}{H_{*,*}(X)}}{k}\to H_{*,*}(\Sym^k(X))$$
are isomorphisms for all $k\geq 0$. We will prove this by induction on $k$.
First we will establish the \hyperref[pfpt:basecases]{base cases} $k=0,1$. In the
\hyperref[pfpt:inductionstep]{induction step}, we introduce an intermediate statement,
Claim \ref{claim:susp}. We then show that the \hyperref[pfpt:fromclaim]{theorem
follows} from the claim
and finally give a \hyperref[pfpt:claimproof]{proof
of the claim}.

\subsubsection{Base cases}\label{pfpt:basecases}
Recall that $\W{S}{H_{*,*}(X)}$ is the quotient of the tensor algebra
$T(\Fvp\{Q(S)\})$ by the set of
$$[v_1,v_2]\coloneqq v_1v_2-(-1)^{n_1n_2}\b^{-1}_{g_1,g_2}(v_2v_1)$$
for $v_i\in T(\Fvp\{Q(S)\})(g_i)_{n_i}$.
Let $U$ be the set of all $[v_1,v_2]$. By Proposition \ref{prop:pushout}, we may
construct $\W{S}{H_{*,*}(X)}$ as a quotient in $\Mod_{\Fvp}(\GrVect^{\fC})$
(rather than in $\Alg_{\Fvp}(\GrVect^{\fC})$).
Explicitly,
$\arity{\W{S}{H_{*,*}(X)}}{k}$ is given by the following pushout square in $\Mod_{\Fvp}$:
$$
\begin{tikzcd}
\arity{T(\Fvp\{Q(S)\})\dc \Fvp\{U\} \dc
T(\Fvp\{Q(S)\})}{k} \ar[r] \ar[d] & \arity{T(\Fvp\{Q(S)\})}{k} \ar[d]\\
0 \ar[r] & \arity{\W{S}{H_{*,*}(X)}}{k}. 
\end{tikzcd}
$$

Every element of $Q(S)$ has charge at least 1, so elements of $U$ have charge
at least 2. 
Therefore, the charge 0 component of
$$T(\Fvp\{Q(S)\})\dc \Fvp\{U\} \dc T(\Fvp\{Q(S)\})$$ is 0, and
hence $$\arity{\W{S}{H_{*,*}(X)}}{0}\cong \arity{T(\Fvp\{Q(S)\})}{0}\cong
\Fvp.$$
The monoidal unit map $\Fvp\to \arity{T(\Fvp\{Q(S)\})}{0}\to
\arity{\W{S}{H_{*,*}(X)}}{0}$ is thus an
isomorphism.
Similarly, the monoidal unit map
$$\Fvp\to H_{*,*}(\Sym^0(X))$$
is also an isomorphism, so since $\Psi_X$ is a map of monoids, it follows that
$\arity{\Psi_X}{0}$ is an isomorphism.

Similar logic shows that
$$\arity{\W{S}{H_{*,*}(X)}}{1}\cong \arity{T(\Fvp\{Q(S)\})}{1}.$$
Since the only elements of $Q(S)$ with charge 1 are those of $S$, it follows
that 
$$\arity{T(\Fvp\{Q(S)\})}{1}\cong \Fvp\{S\}\cong H_{*,*}(X).$$
Thus, since $H_{*,*}(X)\cong H_{*,*}(\Sym^1(X))$, and $\Psi_X$ is defined on $S$
using the canonical map
$$H_{*,*}(X)\to H_{*,*}(E_{\infty}(X)),$$
we have that $\arity{\Psi_X}{1}$ is an isomorphism as well.

\subsubsection{Induction step}\label{pfpt:inductionstep}
Before proceeding by induction, we introduce some notational simplifications.
First, we will write $\W{S}{X}$ for $\W{S}{H_{*,*}(X)}$. Second, we replace the
category $\Ch^{\fC}_{\F}$ by the equivalent category $\Ch(\Vect_{\F}^{\fC})$ and
the category $\GrVect^{\fC}$ by the equivalent category
$(\Vect^{\fC})^{\Z}$, as described in Section \ref{section:ss}. We will also
write $\S^r X$ for $\S^{\1_{\fG},r}X$.

Suppose we have shown that for all $X$ and any basis, $\arity{\Psi_X}{\ell}$ is an isomorphism
whenever $\ell\leq k-1$. That $\arity{\Psi_X}{k}$ is an isomorphism will follow
from the following claim.
\begin{claim}\label{claim:susp}
If
$$\arity{\Psi_{\S X}}{k}\colon \arity{\W{\S S}{\S X}}{k}\to H_{*}(\Sym^k(\S
X))$$
is an isomoprhism in degrees $*<N$ and a surjection when $*=N$, then
$$\arity{\Psi_X}{k}\colon \arity{\W{S}{X}}{k}\to H_*(\Sym^k(X))$$
is an isomorphism in degrees $*<N-1$ and a surjection when $*=N-1$.
Here, $\S S$ denotes the set with elements $\S x$ for $x\in S$, where $\S x$ has
$\Z$-degree $m+1$ when $x$ has $\Z$-degree $m$. 
\end{claim}

\subsubsection{Proof of theorem assuming Claim
\ref{claim:susp}}\label{pfpt:fromclaim}
First we show that the theorem follows from the claim. For $n\in\Z$, we say that
$X$ is \emph{$n$-connected} if $H_m(X)=0$ for $m\leq n$. If $X$ is
$n$-connected, then $H_m(\Sym^k(X))$ is zero whenever $m<k(n+1)$ (to see this,
recall that $\Sym^k$ preserves quasi-isomorphisms, so we may replace $X$ by its
homology; furthermore, $\cC_{\infty}(k)$ is 0 in negative degrees, so the
minimal degree in which a nonzero class may occur is $m=k(n+1)$).
We claim that $\arity{\W{S}{X}}{k}$ is also 0 in degrees $m<k(n+1)$.
By construction $\W{S}{X}$ is spanned by products of elements in $Q(S)$. A product
$$Q^{I_1}_{g_1}x_1\otimes \dots \otimes Q^{I_j}_{g_j}x_j$$ 
has charge $\sum_{i=1}^j p^{\ell(I_i)}$. By Lemma \ref{lem:admprops}, the condition
$e(I)+\e_1>d_{\Z}(x)$ implies that $d_{\Z}(Q^I_g x)\geq p^{\ell(I)}d_{\Z}(x)$, so a product has degree
$$d_{\Z}(Q^{I_1}_{g_1}x_1\otimes \dots \otimes Q^{I_j}_{g_j}x_j)\geq \sum_{i=1}^j
p^{\ell(I_i)}d_{\Z}(x_i).$$
Since $d_{\Z}(x_i)\geq n+1$, we therefore have that a product with charge $k$ has
degree at least $k(n+1)$. Thus, $\arity{\Psi_X}{k}$ is trivially an isomorphism
in degrees $m<k(n+1)$, since both the source and target are
$(k(n+1)-1)$-connected.

Moreover, if $I$ is nonempty and admissible, then $$d_{\Z}(Q^I_g
x)>p^{\ell(I)}d_{\Z}(x)\geq
p^{\ell(I)}(n+1),$$ so the only way to have
$$d_{\Z}(Q^{I_1}_{g_1}x_1\otimes \dots \otimes Q^{I_j}_{g_j}x_j)=k(n+1)$$
while $\sum p^{\ell(I_i)}=k$ is if each $I_i$ is empty, $j=k$, and every
$d_{\Z}(x_i)=n+1$. That is, only $k$-fold products of elements of $H_{n+1}(X)$
may have degree $k(n+1)$ in $\arity{\W{S}{X}}{k}$. However, the map
$$(H_{n+1}(X)^{\dc k})_{\S_k}\to H_{k(n+1)}(\Sym^k(X))$$  is a
surjection. Thus, $\arity{\Psi_X}{k}$ is a surjection in degree
$m=k(n+1)$ when $X$ is $n$-connected.

If $X$ is $n$-connected, then $\S^r X$ is $(r+n)$-connected, so the above
argument implies that $\arity{\Psi_{\S^r X}}{k}$ is an isomorphism in degrees
$m<k(r+n+1)$ and a surjection when $m=k(r+n+1)$. By iterated applications of
Claim \ref{claim:susp}, this implies that $\arity{\Psi_X}{k}$ is an isomorphism
in degrees $m<k(n+1)+r(k-1).$ Since $k-1\geq 1$, $k(n+1)$ is fixed, and $r$ can
be made arbitrarily large, this implies that $\arity{\Psi_X}{k}$ is an
isomorphism in all degrees. This proves the theorem when $X$ is $n$-connected.

Note that any $X$ may be written as a filtered colimit of bounded-below objects
(i.e. objects that are $n$-connected for some $n$).
Explicitly, $X=\colim_{n\in \Z} X_{\geq n}$, where $X_{\geq n}$ is the
truncation of $X$ in the sense that $$(X_{\geq n})_m=\begin{cases} X_m & m\geq
n\\
0 & m<n.
\end{cases}$$
Since $\W{S}{H_*(X)}$ is determined only by $H_*(X)$, we may replace $X$ by its
homology; that is, we may assume that $X$ is a chain complex with trivial
differentials. Then we have maps
$$\dots\to X_{\geq n+1}\to X_{\geq n}\to X_{\geq n-1}\to \dots \to X$$
that are the identity on $(X_{\geq n})_m$ for all $m$, and, as claimed,
$X=\colim_{n\in \Z} X_{\geq n}.$
Let $$S_n=\{x\in S\mid d_{\Z}(x)\geq n\}.$$ Then $S_n\subset S_{n-1}$, and
$S=\cup_{n\in\Z}S_n.$ Similarly, $$Q(S_n)\subset Q(S_{n-1})\text{   and  }
Q(S)=\cup_{n\in\Z}Q(S_n).$$ The inclusion maps $Q(S_n)\to Q(S_{n-1})$ and
$Q(S_n)\to Q(S)$ induce
maps on the free commutative algebras
$$\W{S_n}{H_*(X_{\geq n})}\to \W{S_{n-1}}{H_*(X_{\geq n-1})}$$
and
$$\W{S_n}{H_*(X_{\geq n})}\to \W{S}{H_*(X)},$$
the latter of which is compatible with the former. Moreover, we have a
map
of graded sets
$$Q(S)\to \colim_{m\in\Z}\W{S_m}{H_*(X_{\geq m})}$$
defined using the maps
$$Q(S_n)\to \colim_{m\in\Z}\W{S_m}{H_*(X_{\geq m})}.$$
This defines a map
$$\W{S}{H_*(X)}\to \colim_{m\in\Z}\W{S_m}{H_*(X_{\geq m})}$$
of commutative algebras, which shows that $\W{S}{H_*(X)}$ satisfies the universal
property of the colimit, and hence this map is an isomorphism.
Now, $H_*(E_{\infty}(-))$ commutes with filtered colimits, so we also have that
$$H_*(E_{\infty}(X))\cong \colim_{n\in\Z} H_*(E_{\infty}(X_{\geq n})).$$
Since the diagrams
$$
\xymatrix{
\W{S_n}{H_*(X_{\geq n})} \ar[r] \ar[d] & \W{S_{n-1}}{H_*(X_{\geq n-1})} \ar[d]\\
H_*(E_{\infty}(X_{\geq n})) \ar[r] & H_*(E_{\infty}(X_{\geq n-1}))
}
$$
commute by naturality of the operations, it therefore suffices to prove the
theorem when $X$ is $n$-connected for some $n$.

\subsubsection{Proof of Claim \ref{claim:susp}}\label{pfpt:claimproof}
Thus, it suffices to prove Claim \ref{claim:susp}. To do this, we will use the
charge-graded spectral sequence described in Section \ref{section:ss} with $E^1$ page
$$\arity{E^1_{t,q}}{k}\cong \bigoplus_r H_{t+q-r}(\Sym^t(\S X))\dc
H_r(\Sym^{k-t}(X))$$
converging to 0 if $k>0$ and $\Fvp$ when
$k=0$.
We will construct a model spectral sequence relating $\W{S}{X}$ and $\W{\S S}{\S
X}$ in a
similar way, and the result
will then follow from a spectral sequence comparison theorem.

Now we build the model spectral sequence $\tilde{E}$ with a charge grading
that has
\begin{equation}\label{eq:grading}
\arity{\tilde{E}^1_{t,q}}{\ell}\cong \arity{\W{\S S}{\S X}_{t+q-*}}{t}\dc
\arity{\W{S}{X}_*}{\ell-t}.
\end{equation}

By Lemma \ref{lem:SSpieces}, we have that $\W{\S S}{\S X}\dc \W{S}{X}$ is isomorphic
to
$$
\left(\bigdc_{Q^I\S x\in Q(\S S)^{-}} \cS(Q^I\S x)\dc
\cS(Q^Ix)\right)\dc \bigdc_{Q^I\S x\in Q(\S S)^+}M,
$$ 
where 
$$M=\bigdc_{i\geq 0} 
\Fvp\{(Q^I\S x)^{jp^i}\}_{0\leq j\leq
p-1} \dc \cS(Q^{\op{i-1}{s}{g}{0}}Q^Ix)\dc
\cS(Q^{\op{i}{s}{g}{1}}Q^I x),$$
and $(g,2s)$ is the bidegree of $Q^I\S x$. We will construct an elementary
spectral sequence for each $Q^I\S x\in Q(\S S)^-$ and, separately, one for each
$Q^I\S x\in Q(\S S)^+$. Taking the product of these and using the Leibniz rule
will then define the differentials for $\tilde{E}$.

For every $Q^I\S x\in Q(\S S)^-$, we may define a charge-graded spectral sequence
with $E^1$ page
$$\cS(Q^I\S x)\dc \cS(Q^I x).$$
Explicitly, if $d_{\Z}(Q^I\S x)=n+1$, this sequence has 
$$\arity{E^r_{0,jn}}{jp^{\ell(I)}}=\Fvp\{1\otimes (Q^Ix)^j\}$$
if $r\leq p^{\ell(I)}$ and
$$\arity{E^r_{p^{\ell(I)},n(j+1)+1-p^{\ell(I)}}}{(j+1)p^{\ell(I)}}=\Fvp\{Q^I \S
x\otimes (Q^Ix)^j\}$$
if $r\leq p^{\ell(I)}$ and is otherwise zero. The differential is given by
$$d^{p^{\ell(I)}}(Q^I\S x\otimes 1)=1\otimes (-1)^{\e(I)}Q^I x,$$
where $\e(I)=\sum_{i=1}^{\ell(I)}\e_i$ (that is, $\e(I)$ is the number of times
an operation of the form $\b Q^s$ appears in $I$), 
and on the rest of the $p^{\ell(I)}$ page the differential is defined using the Leibniz rule.
Otherwise the differentials are all zero, and it is clear that the spectral
sequence converges to 0 except for a copy of the unit $\Fvp$ in charge 0.

For every $Q^I\S x\in Q(\S S)^+$ with $2s=d_{\Z}(Q^I\S x)$ and $g=d_{\fG}(Q^I \S
x)$, for each $i\geq 0$ we may similarly define a
charge-graded spectral sequence with $E^1$ page
$$ \Fvp\{(Q^I\S x)^{0 p^i},(Q^I\S x)^{p^i},\dots,(Q^I\S x)^{(p-1)p^i}\} \dc
\cS(Q^{\op{i-1}{s}{g}{0}}Q^Ix)\dc \cS(Q^{\op{i}{s}{g}{1}} Q^Ix).$$
Let $y=(Q^I\S x)^{p^i}$, $z=Q^{\op{i-1}{s}{g}{0}}Q^Ix$, and
$v=Q^{\op{i}{s}{g}{1}}Q^Ix$. In this spectral sequence, $y$ will trangress to
$(-1)^{\e(I)}z$, and $y^{p-1}\otimes z$ will transgress to $(-1)^{\e(I)+1}v$.
Explicitly, if $0\leq j_1\leq p-1$ and $j_2\geq 0$, then
$$\arity{E^r_{j_1p^{i+\ell(I)},q}}{p^{i+\ell(I)}(j_1+pj_2)} =
\Fvp\{y^{j_1}\otimes 1\otimes v^{j_2}\} \text{ for } \begin{cases}
r\leq p^{i+\ell(I)} & j_1>0\\
r\leq (p-1)p^{i+\ell(I)} & j_1=0,
\end{cases}$$
where $q=p^ij_1(2s-p^{\ell(I)})+2j_2(p^{i+1}s-1)$, and
$$\arity{E^r_{j_1p^{i+\ell(I)},q}}{p^{i+\ell(I)}(j_1+pj_2+1)}=\Fvp\{y^{j_1}\otimes
z\otimes v^{j_2}\} \text{ for } \begin{cases}
r\leq p^{i+\ell(I)} & j_1<p-1\\
r\leq (p-1)p^{i+\ell(I)} & j_1=p-1,
\end{cases}$$
where $q=p^ij_1(2s-p^{\ell(I)})+2s p^i-1+2j_2(p^{i+1}s-1).$ Otherwise,
$E^r_{t,q}=0$.
The differentials are defined by
$$d^r(y^{j_1}\otimes 1\otimes v^{j_2})=(-1)^{\e(I)}\vp(\tr_{g\oplus g}(\b_{g,g}))^{j_1-1}j_1 y^{j_1-1} \otimes z\otimes
v^{j_2}$$
for $r=p^{i+\ell(I)}$ if $j_1>0$, and
$$d^r(y^{p-1}\otimes z\otimes v^{j_2})=(-1)^{\e(I)+1}y^0\otimes 1\otimes v^{j_2+1}$$ for
$r=(p-1)p^{i+\ell(I)}$. All other differentials are zero. This spectral sequence converges to 0 except for a copy of $\Fvp$ in charge
0.

Thus, we may define a spectral sequence on the product $\tilde{E}$ of these elementary
spectral sequences using the Leibniz rule. By Proposition \ref{prop:Kudo}, the differentials in $\tilde{E}$ are constructed so that we
have an induced map $f\colon\tilde{E}\to E$
coming from the maps $\Psi_{\S X}$ and $\Psi_X$.

 Since both $\tilde{E}$ and $E$
converge to zero, aside from the unit $\Fvp$ in charge 0, on which the map $f$
is clearly an isomorphism, we have that $f^{\infty}$ is an isomorphism in all
degrees. By assumption
$\arity{\Psi_{\S X}}{\ell}$ is an isomorphism in all degrees when $\ell<k$, an
isomorphism in degrees $*<N$ and a surjection when $*=N$ for $\ell=k$,
so $\arity{f^1_{k,q}}{k}$ is an isomorphism when $q<N-k$ and a surjection when
$q=N-k$.
Thus,
by the spectral sequence comparison theorem \ref{thm:sscomp}, we have that
$\arity{f^1_{0,q}}{k}$ is an isomorphism for $q<N-1$ and a surjection for
$q=N-1$.
We therefore have that $\arity{\Psi_X}{k}$ is an isomorphism, respectively
surjection, in the appropriate degrees. This proves the claim.
\end{proof}

\section{Allowable Dyer--Lashof algebras}\label{section:allowable} 
In practice, the result of Theorem \ref{thm:W} is most useful for computations,
since it gives explicit generators for the homology of a free
$E_{\infty}$-algebra $H_{*,*}(E_{\infty}(X))$. However, our algebraic
description $\W{S}{H_{*,*}(X)}$ merely has the structure of a commutative monoid object in
$\Mod_{\Fvp}(\GrVect^{\fC})$; it does not come equipped with an action of the Dyer--Lashof
operations as $H_{*,*}(E_{\infty}(X))$ does. Moreover, $\W{S}{H_{*,*}(X)}$ was not
functorial in $X$, since it required a choice of basis. In this section we will describe
precisely the algebraic structure with which we have endowed
$H_{*,*}(E_{\infty}(X))$, and we will reformulate Theorem \ref{thm:W} to show that
$H_{*,*}(E_{\infty}(X))$ is a free object with respect to this structure.

\subsection{The Dyer--Lashof algebra}
To begin with, we will define the Dyer--Lashof algebra. Intuitively this should
be the free $\Fvp$-algebra on the set of symbols $\{\b^{\e} Q^s_g\}$ modulo some
relations. However, in order to accomodate the fact that over a general field
$\F$ of characteristic $p$, the Dyer--Lashof operations are Frobenius linear
rather than linear, we must instead define the Dyer--Lashof algebra in a
category of $\F$-$\F$-bimodules. This allows us to encode the relation $Q^s_g
\a=\a^p Q^s_g$ for a scalar $\a$ by defining different right and left vector
space structures.

In the case that $\F=\F_p$, having different left and right vector space
structures is not necessary, and the
Dyer--Lashof algebra can be constructed as a monoid object in
$\Mod_{\Fvp}(\GrVect^{\fC})$. This then allows one to consider left modules over
the Dyer--Lashof algebra in $\Mod_{\Fvp}(\GrVect^{\fC})$. The following
definitions using functors into a category of $\F$-$\F$-bimodules are designed
to carry out this same construction while allowing different left and right
vector space structures.

\subsubsection{The category $\Bimod_{\bifvp}$}
Consider the category  of functors from
$\fC^{\cop} \times \fC$ to the category of $\Z$-graded $\F$-$\F$-bimodules.
As usual, we may take the tensor product of bimodules to obtain a bimodule; we
denote this tensor product by $\tensor[_{\F}]{\otimes}{_{\F}}$ to distinguish it
from the tensor product $\otimes_{\F}$ of $\F$--vector spaces.
\begin{defn}
Let
$(\Bimod_{\F}^{\Z})^{\fC^{\cop}\times \fC}$
denote the category of functors from $\fC^{\cop} \times \fC$ to the category of $\Z$-graded $\F$-$\F$-bimodules.
We equip this functor category with a
monoidal structure, denoted $\bitimes$, as follows:
 for $F,G\in (\Bimod_{\F}^{\Z})^{\fC^{\cop}\times
\fC}$, $F\bitimes G$ is defined either as a coend, or, equivalently, as the coequalizer
$$F\bitimes G\coloneqq \coeq\left(\bigoplus_{\substack{g,g'\in \fG \\ f\colon
g\to g'}} F(g',-)\tensor[_{\F}]{\otimes}{_{\F}} G(-,g) \rightrightarrows
\bigoplus_{g\in \fG}F(g,-)\tensor[_{\F}]{\otimes}{_{\F}} G(-,g)\right),$$
where one arrow is obtained by $$F(g',-)\tensor[_{\F}]{\otimes}{_{\F}}
G(-,g)\xrightarrow{F(\id,\id)\otimes G(\id,f)}
F(g',-)\tensor[_{\F}]{\otimes}{_{\F}} G(-,g'),$$
and the other is obtained by
$$F(g',-)\tensor[_{\F}]{\otimes}{_{\F}} G(-,g)\xrightarrow{F(f,\id)\otimes
G(\id,\id)}
F(g,-)\tensor[_{\F}]{\otimes}{_{\F}} G(-,g).$$
\end{defn}

\begin{rem}
Such functors are sometimes called \emph{profunctors}, and the monoidal
structure we have defined is sometimes called the \emph{composition} of
profunctors, which motivates our choice of notation.
\end{rem}
Associators for this monoidal structure are defined using
the associators for tensor products of bimodules. The unit in this category is
$\Hom_{\F\fG}(-,-)$, which is considered as a bimodule with both left and right scalar
multiplication given by multiplication in $\F$.

We can also define an analogue $\bifvp$ of $\Fvp$ in
$(\Bimod_{\F}^{\Z})^{\fC^{\cop}\times
\fC}$ (here ``bm" stands for ``bimodule").

\begin{defn}
 Let $\bifvp(g,h)$ be the coequalizer of
$$\bigoplus_{\Aut_{\fG}(g)\times \Aut_{\fG}(h)}\Hom_{\F\fG}(g,h)\rightrightarrows
\Hom_{\F\fG}(g,h),$$
where one arrow precomposes by $b\colon g\to g$ and postcomposes by $b'\colon
h\to h$, and the other arrow multiplies by $\vp(\tr_h(b'))$ on the \emph{left}
and by $\vp(\tr_g(b))$ on the \emph{right}.
 Morphisms in $\fG^{\cop}\times \fG$ act as pre- and
postcomposition.
\end{defn}
Note that here the order of multiplication is
actually not important, as $\F$ is commutative, but later on we will need to
keep track of this.
Thus, $\bifvp$ is a quotient of the unit object $\Hom_{\F\fG}(-,-)$. It is not
hard to show that the unit map $\Hom_{\F\fG}(-,-)\bitimes \Hom_{\F\fG}(-,-)\to
\Hom_{\F\fG}(-,-)$ descends to a map $\bifvp\bitimes\bifvp\to\bifvp$ making
$\bifvp$ into a monoid object in $(\Bimod_{\F}^{\Z})^{\fG^{\cop}\times \fG}$.

\begin{defn}
Let $\Bimod_{\bifvp}$ denote the category of $\bifvp$-$\bifvp$-bimodules in
the functor category $(\Bimod_{\F}^{\Z})^{\fG^{\cop}\times \fG}$. This inherits a monoidal structure
$\bitimes_{\bifvp}$ in the usual way, which we will denote simply by $\bitimes$,
since it coincides with $\bitimes$.
\end{defn}
 
Since
$\bifvp$ is once again a quotient of the unit object, a similar argument to that
of Proposition \ref{prop:property} shows that 
being a bimodule over $\bifvp$ is a property rather than extra structure.

\begin{prop}
A functor $X\in (\Bimod_{\F}^{\Z})^{\fG^{\cop}\times \fG}$ may be equipped with
the structure of a bimodule over $\bifvp$ if and only if for every
$(g,h)\in\fG^{\cop}\times \fG$ and every morphism $(b,b')\in
\Aut_{\fG^{\cop}\times \fG}(g,h)$, the induced map
$$X(b,b')\colon X(g,h)\to X(g,h)$$
is given by \emph{left} multiplication by $\vp(\tr_h(b'))$ and
\emph{right} multiplication by $\vp(\tr_g(b^{\cop}))$.
Here $b^{\cop}$ denotes
the ``underlying'' morphism in $\fG$ (that is, a morphism $f\colon g_1\to g_2$ in
$\fG^{\cop}$ is by definition a morphism $g_2\to g_1$ in $\fG$; we denote this
morphism by $f^{\cop}$).
\end{prop}

There is also a functor $$-\biact-\colon \Bimod_{\bifvp}\times
\Mod_{\Fvp}(\GrVect^{\fC})\to \Mod_{\Fvp}(\GrVect^{\fC}),$$
where $M\biact X$ is defined either as a coend, or, equivalently, as the
coequalizer of
$$\bigoplus_{\substack{g,g'\in \fG \\ f\colon
g\to g'}} M(g',-)\tensor[_{\F}]{\otimes}{_{\F}} X(g)\rightrightarrows
\bigoplus_{g\in\fG} M(g,-)\tensor[_{\F}]{\otimes}{_{\F}} X(g),$$
where one arrow is $M(f,\id)\otimes X(\id)$, and the other is $M(\id,\id)\otimes
X(f).$ Note that the resulting object is indeed a module over $\Fvp$, since
an automorphism $f$ of $h$ acts on $M(g,h)$, and hence on $(M\biact X)(h)$, as left multiplication by
$\vp(\tr_h(f))$.

\begin{lem}\label{lem:monad}
If $M$ is a monoid object in $\Bimod_{\bifvp}$, then the endofunctor
$$M\biact-\colon \Mod_{\Fvp}(\GrVect^{\fC})\to
\Mod_{\Fvp}(\GrVect^{\fC})$$
inherits a monad structure.
\end{lem}

\begin{proof}
The obvious map
$$
\begin{tikzcd}
\bigoplus_{g\in \fG}
M(g,-)\tensor[_{\F}]{\otimes}{_{\F}}\left(\bigoplus_{g'\in \fG}
M(g',g)\tensor[_{\F}]{\otimes}{_{\F}} X(g')\right) \ar[d]\\
 \bigoplus_{g'\in \fG}
\left(\bigoplus_{g\in \fG}M(g,-)\tensor[_{\F}]{\otimes}{_{\F}}
M(g',g)\right) \tensor[_{\F}]{\otimes}{_{\F}} X(g')
\end{tikzcd}
$$
induces a map $M\biact (M\biact X)\to (M\circ M)\biact X$, which, when composed
with the monoidal product map, induces a map $M\biact (M\biact X)\to M\biact X$.
The monoid diagrams for $M$ are precisely what is needed to show that the
endofunctor
$M\biact-$
is a monad.
\end{proof}

In light of Lemma \ref{lem:monad}, it makes sense to consider objects in
$\Mod_{\Fvp}(\GrVect^{\fG})$ as algebras over a monoid object in
$\Bimod_{\bifvp}$.

\subsubsection{Construction of the Dyer--Lashof algebra}
Now we have the machinery to define the Dyer--Lashof algebra $R_{\vp}$ as a
monoid object in $\Bimod_{\bifvp}$.

If $V\in \Bimod_{\bifvp}$, then we can form the free monoid on $V$,
$T_{\F\text{-}\F}(V)$, 
defined by
$$T_{\F\text{-}\F}(V)\coloneqq \bigoplus_{n\geq 0} V^{\bitimes n}.$$
Defining a $T_{\F\text{-}\F}(V)$-algebra structure on $X\in
\Mod_{\Fvp}(\GrVect^{\fC})$, then, amounts to defining a map
$$V\biact X\to X.$$

Consider the set
$$Q=\{\b^{\e} Q^s_g\mid s\in \Z \text{ if } \vp(\tr_{g\oplus g}(\b_{g,g}))=1, s\in \Z+\tfrac12 \text{ if }
\vp(\tr_{g\oplus g}(\b_{g,g}))=-1\}$$
with $\Z$-grading 
$$d_{\Z}(Q^s_g)=2s(p-1) \text{ and } d_{\Z}(\b Q^s_g)=2s(p-1)-1.$$
Let $V(\b^{\e}Q^s_g)\in \Bimod_{\bifvp}$ be defined as follows: take the quotient of
$$\Hom_{\F\fG}(\gpower{g}{p},-)\otimes_{\F_p}\Hom_{\F\fG}(-,g)$$
by the set of $\a^p\otimes 1-1\otimes \a$ for each $\a\in \F$. That is, after
evaluating this tensor product on $(h,h')$, where $h\cong g$ and $h'\cong
\gpower{g}{p}$, we obtain $\F\otimes_{\F_p}\F$, and
we take the quotient of this. Note that this quotient may be described as $\F$ equipped
with the usual left multiplication as its left scalar multiplication, and right
multiplication after applying the Frobenius homomorphism as its right scalar
multiplication.
Now we take the levelwise tensor product of the resulting functor with
$\F[2s(p-1)-\e]$ to obtain a functor concentrated in $\Z$-degree $2s(p-1)-\e$,
and finally we compose on the right and left by $\bifvp$ to obtain an object in
$\Bimod_{\bifvp}$. 

Let 
$$V(Q)\coloneqq \bigoplus_{\b^{\e}Q^s_g\in Q} V(\b^{\e}Q^s_g)\in
\Bimod_{\bifvp}.$$

Finally, let $R_{\vp}$ be the quotient of $T_{\F\text{-}\F}(V(Q))$ by the Adem relations. That is,
we consider the following elements in
$T_{\F\text{-}\F}(V(Q))(g,\gpower{(\gpower{g}{p})}{p})$, where $g$ ranges over
all $g\in \fG$:
$$Q_{\gpower{g}{p}}^rQ_g^s - \vp(\tr_{g\oplus g}(\b_{g,g}))^{\tfrac{p-1}{2}}\sum_i (-1)^{r+i} \binom{(i-s)(p-1)-1}{r-(p-1)s-i-1}
Q_{\gpower{g}{p}}^{r+s-i}Q_g^i$$
and
$$\b Q_{\gpower{g}{p}}^rQ_g^s - \vp(\tr_{g\oplus g}(\b_{g,g}))^{\tfrac{p-1}{2}}\sum_i (-1)^{r+i} \binom{(i-s)(p-1)-1}{r-(p-1)s-i-1} \b
Q_{\gpower{g}{p}}^{r+s-i}Q_g^i$$
for all $r>ps$, and
for all $r\geq ps$, 
\begin{align*}
Q_{\gpower{g}{p}}^r\b Q_g^s-&\sum_i (-1)^{r+i}\binom{(i-s)(p-1)}{r-(p-1)s-i} \b
Q_{\gpower{g}{p}}^{r+s-i}Q_g^i\\
&+\vp(\tr_{g\oplus g}(\b_{g,g}))^{\tfrac{p-1}{2}}\sum_i
(-1)^{r+i}\binom{(i-s)(p-1)-1}{r-(p-1)s-i} Q_{\gpower{g}{p}}^{r+s-i} \b Q_g^i
\end{align*}
and
$$\b Q_{\gpower{g}{p}}^r\b Q_g^s+\vp(\tr_{g\oplus g}(\b_{g,g}))^{\tfrac{p-1}{2}}\sum_i (-1)^{r+i}\binom{(i-s)(p-1)-1}{r-(p-1)s-i} \b
Q_{\gpower{g}{p}}^{r+s-i} \b
Q_g^i,$$
where in all cases $r,s\in \Z$ if $\vp(\tr_{g\oplus g}(\b_{g,g}))=1$ and $r,s\in \Z+\tfrac12$ if
$\vp(\tr_{g\oplus g}(\b_{g,g}))=-1$.
Given an element $v$ in $T_{\F\text{-}\F}(V(Q))(g,\gpower{(\gpower{g}{p})}{p})$ with
$\Z$-degree $n$, we obtain a map
$$X_v\coloneqq\bifvp\bitimes
\S^n(\Hom_{\F\fG}(\gpower{(\gpower{g}{p})}{p},-)\otimes_{\F_p}\Hom_{\F\fG}(-,g))\bitimes
\bifvp \to T_{\F\text{-}\F}(V(Q)),$$
where $\S^n$ here denotes a levelwise tensor product with $\F[n]$.
Taking the coproduct over all Adem relations $v$, we then define $R_{\vp}$ to be
the pushout of the diagram
$$
\begin{tikzcd}
T_{\F\text{-}\F}(\oplus_v X_v) \ar[r] \ar[d] & T_{\F\text{-}\F}(V(Q)) \\
T_{\F\text{-}\F}(0)
\end{tikzcd}
$$
in $\Mon(\Bimod_{\bifvp})$. Thus,
$R_{\vp}$ is a monoid object in the category of $\bifvp$-bimodules.

\begin{defn}
The monoid $R_{\vp}$ is the \emph{$\Fvp$-Dyer--Lashof algebra}.
\end{defn}

The functor $$R_{\vp}\biact -\colon \Mod_{\Fvp}(\GrVect^{\fC}_{\F})\to
\Mod_{\Fvp}(\GrVect^{\fC}_{\F})$$
defines a monad since
$R_{\vp}$ is an associative monoid object. 

\subsection{Allowable Dyer--Lashof modules}
We define an intermediate algebraic structure that consists of objects with an
action of $R_{\vp}$ subject to a degree condition. In order to be consistent
with Peter May's terminology, and to reserve the term ``algebra'' for such
objects that also have a monoid structure, we choose to call these objects
Dyer--Lashof \emph{modules}. However, we remark that this terminology is
particularly confusing here, as our Dyer--Lashof modules are in fact
\emph{algebras over the monad} $R_{\vp}\biact -$.

\begin{defn}
A \emph{left $R_{\vp}$-module} is an algebra over the monad $R_{\vp}\biact -$.
A left $R_{\vp}$-module $X$ is \emph{allowable} if the map
$$\rho\colon R_{\vp}\biact X\to X$$
sends the image of $V(\b^{\e}Q^s_g)\biact X_m$ in $R_{\vp}\biact X$ to 0 when
$2s<m+\e$, where $X_m$ denotes the $\Z$-degree $m$ component of $X$. In
particular this means that
$$\rho(\b^{\e} Q^s_g \otimes x)=0$$
when $x\in X(g)_m$ for $2s<m+\e$.
\end{defn}
The allowable $R_{\vp}$-modules form a full subcategory
$\LMod_{R_{\vp}}^{\text{allow}}$ of the category of left
$R_{\vp}$-modules.

There is a free allowable $R_{\vp}$-module functor
$$D\colon \Mod_{\Fvp}(\GrVect^{\fC}_{\F})\to \LMod_{R_{\vp}}^{\text{allow}},$$
where $D(X)$ is the quotient of $R_{\vp}\biact X$ by the set consisting of the
elements of $$Q^I_g\tensor[_{\F}]{\otimes}{_{\F}} X(g)_m$$ when
$e(I)<m$.
 Note that $D$
factors
through the category of all left $R_{\vp}$-modules: 
$$\Mod_{\Fvp}(\GrVect^{\fC}_{\F})\xrightarrow{R_{\vp}\biact -} \LMod_{R_{\vp}}\to
\LMod_{R_{\vp}}^{\text{allow}},$$
where the second functor is given by taking the quotient defined above. It is
standard to see that $R_{\vp}\biact -$ is left adjoint to the forgetful
functor. As
we have seen before, the universal property of the quotient shows that the
quotient functor is left adjoint to the inclusion functor. Thus, $D$ is left
adjoint to the forgetful functor to $\Mod_{\Fvp}(\GrVect^{\fC})$.

In order to compare the functor $D$ to the basis-level constructions of the
previous section, it will be convenient to have an explicit description of a
basis for free allowable $R_{\vp}$-modules.
We refer the reader to Section \ref{section:admissible} for a reminder of the
definitions and notation of admissible sequences used in the following statement.

\begin{lem}\label{lem:Dbasis}
If $S$ is a basis for $X\in \Mod_{\Fvp}(\GrVect^{\fC})$, then 
$$\{Q^I_g\tensor[_{\F}]{\otimes}{_{\F}} x \mid x\in S(g), I \text{ admissible, }
e(I)\geq d_{\Z}(x)\}$$
is a basis for $D(X)$.
\end{lem}

\begin{proof}
First we show that $R_{\vp}$ is spanned by $\{Q^I_g\}$ where $I$ is admissible.
(Here we may write all linear combinations with constants on the left.) By
construction, $R_{\vp}$ is spanned by $\{Q^I_g\}$ for
\emph{all} $I$, so we show by induction on $\ell(I)$ that each such product may be written as a linear
combination of admissible compositions.

If $\ell(I)=0$ or 1, then there is nothing to prove. Suppose that we have shown
that $Q^I_g$ may be written as a linear combination of admissible compositions
when $1\leq \ell(I)\leq k-1$. To prove the statement for $\ell(I)=k$, we require
a further intermediate statement (see Lecture 6 of \cite{Lurie} for a similar
argument):
\begin{claim}\label{claim:admissible}
 Fix $m\in \Z$, and suppose $s_2\geq p^{k-2}
m+\e_2$ and $s_1\geq p^{k-1}m$. Then either $I$ is already admissible, or $Q^I_g$ may be
written as a linear combination of $Q^{I'}_g$, where $\ell(I')=k$, $I'$ is
admissible, $s_2'>p^{k-2}m+\e_2'$, and $s_1'>p^{k-1}m$.
\end{claim}

We prove Claim \ref{claim:admissible} by induction on
$s_1$.
First note that we may assume without loss of generality that $I_{\geq 2}$ is
admissible, since it has length $k-1$.
The base case for the second induction is when $s_1=p^{k-1}m$. Since $s_2\geq p^{k-2}m+\e_2$, we
have that $$ps_2-\e_2\geq p^{k-1}m+(p-1)\e_2\geq p^{k-1}m=s_1.$$ Thus, $I$ is
already admissible since $I_{\geq 2}$ is.

Now suppose $s_1>p^{k-1}m$. Once again, either $I$ is already admissible, in
which case there is nothing to prove, or $s_1>ps_2-\e_2$. In the latter case, we
apply an Adem relation. We split this into two cases.

First, if $\e_1=\e_2=1$ or $\e_2=0$, then the appropriate Adem relation expresses $\b^{\e_1}
Q^{s_1}_{\gpower{g}{p^{k-1}}}\b^{\e_2}Q^{s_2}_{\gpower{g}{p^{k-2}}}$ as a linear combination of
operations $\b^{\e_1} Q^{s_1+s_2-i}_{\gpower{g}{p^{k-1}}}\b^{\e_2}
Q^i_{\gpower{g}{p^{k-2}}},$ where
$i\leq s_1-s_2(p-1)+\e_2-1$ and $pi-\e_2\geq s_1$ (otherwise the binomial
coefficients are zero). Since $s_1>ps_2-\e_2$, the second inequality tells us
that $i>s_2$, and since $s_2\geq p^{k-2}m+\e_2$, we have that $i>p^{k-2}m+\e_2$.
The first inequality implies that
$$s_1+s_2-i\geq ps_2+1-\e_2\geq p^{k-1}m+(p-1)\e_2+1>p^{k-1}m.$$
Moreover, since $i>s_2$, $s_1+s_2-i<s_1$. Thus, $\b^{\e_1}
Q^{s_1+s_2-i}_{\gpower{g}{p^{k-1}}}\b^{\e_2} Q^i_{\gpower{g}{p^{k-2}}}
Q^{I_{\geq 3}}_g$ both satisfies the
assumptions of Claim \ref{claim:admissible} and has first index $s_1+s_2-i$ strictly smaller than $s_1$.
The induction hypothesis then proves that this composition may be rewritten as a sum
of admissible compositions.

Second, if $\e_1=0$ and $\e_2=1$, then the appropriate Adem relation expresses $
Q^{s_1}_{\gpower{g}{p^{k-1}}}\b Q^{s_2}_{\gpower{g}{p^{k-2}}}$ as a linear combination of terms of two
different compositions. The terms of the form
$Q^{s_1+s_2-i}_{\gpower{g}{p^{k-1}}} \b
Q^{i}_{\gpower{g}{p^{k-2}}}$ have nonzero coefficients only when $$i\leq
s_1-s_2(p-1)+\e_2-1$$ and
$pi-\e_2\geq s_1$, and the same argument as for when $\e_1=\e_2=1$ shows that
this term may be rewritten as a sum of admissibles. The other terms are of the form $\b
Q^{s_1+s_2-i}_{\gpower{g}{p^{k-1}}} Q^i_{\gpower{g}{p^{k-2}}}$ and are nonzero only when $i\leq
s_1-s_2(p-1)$ and $pi\geq s_1$. Since $s_1>ps_2-\e_2$, the latter inequality
implies that $pi>ps_2-1$, and hence $$i\geq s_2\geq p^{k-2}m +1>p^{k-2}m.$$
The former inequality implies that $$s_1+s_2-i\geq p s_2\geq
p^{k-1}m+p>p^{k-1}m.$$
Thus, $$\b
Q^{s_1+s_2-i}_{\gpower{g}{p^{k-1}}} Q^i_{\gpower{g}{p^{k-2}}} Q^{I_{\geq 3}}_g$$
satisfies the
conditions of Claim \ref{claim:admissible}, and the induction hypothesis may be applied to rewrite it as
a linear combination of admissibles when $s_1+s_2-i<s_1$. However, we must also consider the case when $i=s_2$, since this
forces $s_1+s_2-i=s_1$. This also forces $ps_2=pi=s_1$, though, so the
composition $\b Q^{s_1}_{\gpower{g}{p^{k-1}}} Q^{s_2}_{\gpower{g}{p^{k-2}}}
Q^{I_{\geq 3}}_g$ is
already admissible.

This proves Claim \ref{claim:admissible} for all $s_1\geq p^{k-1}m$. 
Since $m$ was arbitrary, and given that for any particular $I$, we may always
choose $m$ small enough so that $I$ satisfies the assumptions of Claim
\ref{claim:admissible}, this
completes the proof that $Q^I_g$ may be written as a linear combination of
admissible compositions.

Now note that $e(I)<n$ precisely when
$Q^I_g\tensor[_{\F}]{\otimes}{_{\F}} x$ is an element of the set by which we
quotient $R_{\vp}\biact X$ to obtain $D(X)$. Therefore, only those $I$
satisfying $e(I)\geq n$ are required to span $D(X)$.

Since $S$ is a basis for $X$, this proves that the canonical map
$$\Fvp\{Q^I_g \tensor[_{\F}]{\otimes}{_{\F}} x\mid x\in S(g),I \text{ admissible,}
\,e(I)\geq d_{\Z}(x)\}\to D(X)$$
is surjective.

By Theorem \ref{thm:upper index}, $H_{*,*}(E_{\infty}(X))$ has the structure of
an allowable $R_{\vp}$-module. Thus, the map $X\to E_{\infty}(X)$ in
$\GrVect^{\fC}_{\F}$ induces a map
$$D(X)\to H_{*,*}(E_{\infty}(X))$$
of allowable $R_{\vp}$-modules that sends $Q^I_g
\tensor[_{\F}]{\otimes}{_{\F}}x\mapsto Q^I_g(x)$. In Theorem \ref{thm:W}, we
gave a basis for $H_{*,*}(E_{\infty}(X))$ consisting of products of elements in
$Q(S)$;
if $I$ is admissible and $e(I)+\e_1>d_{\Z}(x)$ for $x\in S(g)$, then $Q^I_g(x)$
is one of these basis elements. If $\e_1=1$, then $e(I)+\e_1>d_{\Z}(x)$ is
equivalent to $e(I)\geq d_{\Z}(x)$. If $\e_1=0$, and $e(I)=d_{\Z}(x)$, then
Lemma \ref{lem:admprops} \ref{list:powers} implies that $Q^I_g(x)=(Q^{I_{\geq
j}}_g(x))^{p^{j-1}}$, where $e(I_{\geq j})+\e_j>d_{\Z}(x)$ and
$2s_{j-1}=d_{\Z}(Q^{I_{\geq j}}_g(x)).$ Thus, $Q^{I_{\geq j}}_g(x)\in Q(S)$, and
$$(-1)^{d_{\Z}(Q^{I_{\geq j}}_g(x))}\vp(\tr_{g\oplus
g}(\b_{g,g}))=(-1)^{2s_{j-1}}\vp(\tr_{g\oplus g}(\b_{g,g}))=1,$$
so $(Q^{I_{\geq
j}}_g(x))^{p^{j-1}}$ is also an element of the basis we constructed for
$H_{*,*}(E_{\infty}(X))$. Thus,
$$\{Q^I_g\tensor[_{\F}]{\otimes}{_{\F}} x \mid x\in S(g), I \text{ admissible},
\,e(I)\geq d_{\Z}(x)\}$$
must be linearly independent in $D(X)$, since its image in
$H_{*,*}(E_{\infty}(X))$ is linearly independent. 
\end{proof}

Now we will define a monoidal structure on the category
$\LMod_{R_{\vp}}^{\text{allow}}$.
Recall that $R_{\vp}$ was the quotient of the free monoid object
$T_{\F\text{-}\F}(V(Q))$
in $\bifvp$-bimodules. For any $X,Y\in \LMod_{R_{\vp}}^{\text{allow}}$, we
will define a left $R_{\vp}$-module structure on $X\dc Y$.
First we define a map
$$V(Q)\biact(X\dc Y)\to X\dc Y$$
using the external Cartan formulas.
We do this in two steps. If $f\colon h_1\oplus h_2\to g$ is a morphism in $\fG$,
then we define the composition
$$V(\b^{\e}Q^s_g)(g,\gpower{g}{p})\tensor[_{\F}]{\otimes}{_{\F}} (X(h_1)\otimes
Y(h_2))\to X(\gpower{h_1}{p})\otimes Y(\gpower{h_2}{p})\to (X\dc
Y)(\gpower{g}{p}),$$
where the first map sends
$$Q^s_{g}\tensor[_{\F}]{\otimes}{_{\F}} (x\otimes y)\mapsto \sum_{j+k=s}
(-1)^{2jk(p-1)}Q^j_{h_1}(x)\otimes Q^k_{h_2}(y)$$
if $\e=0$ and sends $\b Q^s_{g}\tensor[_{\F}]{\otimes}{_{\F}}(x\otimes y)$ to
\begin{align*}
 \sum_{j+k=s} (-1)^{2jk(p-1)} \left(\b
Q^j_{h_1}(x)\otimes Q^k_{h_2}(y)+(-1)^{d_{\Z}(x)} Q^j_{h_1}(x)\otimes \b
Q^k_{h_2}(y)\right).
\end{align*}
The second map applies the canonical map $X(\gpower{h_1}{p})\otimes
Y(\gpower{h_2}{p})\to (X\dc Y)(\gpower{h_1}{p}\oplus \gpower{h_2}{p})$ followed
by the inverse shuffle map $\gpower{h_1}{p}\oplus \gpower{h_2}{p}\to \gpower{(h_1\oplus
h_2)}{p}$ and finally $$\gpower{f}{p}\colon \gpower{(h_1\oplus h_2)}{p}\to
\gpower{g}{p}.$$
Note that these sums are finite since $X$ and $Y$ are allowable.
This map induces a map
$$V(\b^{\e}Q^s_g)(g,\gpower{g}{p})\tensor[_{\F}]{\otimes}{_{\F}} (X\dc Y)(g)\to
(X\dc Y)(\gpower{g}{p}).$$
Given $g'\cong g$ and $h\cong \gpower{g}{p}$, choose any morphisms $f_1\colon
g\to g'$ and $f_2\colon h\to \gpower{g}{p}$, and define the map
$$V(\b^{\e}Q^s_g)(g',h)\tensor[_{\F}]{\otimes}{_{\F}} (X\dc Y)(g') \to (X\dc
Y)(h)$$
by the diagram
$$
\begin{tikzcd}
V(\b^{\e}Q^s_g)(g',h)\tensor[_{\F}]{\otimes}{_{\F}} (X\dc Y)(g') \ar[r,dashed]
\ar{d}[swap]{(f_1,f_2)\otimes f_1^{-1}} & (X\dc Y)(h)\\
V(\b^{\e}Q^s_g)(g,\gpower{g}{p})\tensor[_{\F}]{\otimes}{_{\F}} (X\dc Y)(g)
\ar[r] & (X\dc Y)(\gpower{g}{p}) \ar{u}[swap]{f_2^{-1}}
\end{tikzcd}
$$
where the bottom arrow is the one we have just defined in terms of the Cartan
formulas. An elementary argument shows that this map is independent of the
choice of $f_1,f_2$, and, moreover, the collection of all such maps induces a
map $V(\b^{\e}Q^s_g)\biact (X\dc Y)\to X\dc Y$, the sum of which gives the
desired map
$$V(Q)\biact (X\dc Y)\to X\dc Y.$$
This then induces a $T_{\F\text{-}\F}(V(Q))$-algebra map
$$\mu\colon T_{\F\text{-}\F}(V(Q))\biact(X\dc Y)\to X\dc Y.$$

By our definition of this map, the action of a symbol
$\b^{\e}Q^s_g$ on $X\dc Y$ satisfies the external Cartan formulas, since the
inverse shuffle map cancels with the shuffle map. 

We claim that this map sends all Adem relations to 0, and thus defines a left
$R_{\vp}$-module structure on $X\dc Y$.

\begin{prop}\label{prop:prodallow}
For $X,Y\in\LMod_{R_{\vp}}^{\text{allow}}$, the map
$$\mu\colon T_{\F\text{-}\F}(V(Q))\biact(X\dc Y)\to X\dc Y$$
factors through a map
$$R_{\vp}\biact(X\dc Y)\to X\dc Y$$
that endows $X\dc Y$ with the structure of a left allowable $R_{\vp}$-module.
\end{prop}

\begin{proof}
For ease of notation, we write $V$ for $V(Q)$.
We have an induced map
$$R_{\vp}\biact (X\dc Y)\to X\dc Y$$ if every Adem relation in
$T_{\F\text{-}\F}(V)$ acts by zero. 
Explicitly, if $u\in (V\bitimes V)(g,\gpower{(\gpower{g}{p})}{p})$ is an Adem
relation, then we need to check that
$$\mu(u\tensor[_{\F}]{\otimes}{_{\F}} (x\otimes y))=0$$
whenever $x\otimes y\in X(g_1)\otimes Y(g_2)$ with $g_1\oplus g_2\cong g$. Choosing such an
$x\otimes y$ is equivalent to specifying maps
$$\S^{g_1,n_1}\Fvp\to X \text{ and } \S^{g_2,n_2}\Fvp\to Y$$
in $\Mod_{\Fvp}$,
where $d_{\Z}(x)=n_1$, $d_{\Z}(y)=n_2$, or, equivalently, specifying maps
$$D(\S^{g_1,n_1}\Fvp)\to X \text{ and }D(\S^{g_2,n_2}\Fvp)\to Y$$
of allowable $R_{\vp}$-modules. Since $T_{\F\text{-}\F}(V)\biact-$ is a functor,
it therefore suffices to assume that $X=D(\S^{g_1,n_1}\Fvp)$ and
$Y=D(\S^{g_2,n_2}\Fvp)$.

By Theorem \ref{thm:upper index}, $H_{*,*}(E_{\infty}(\S^{g_i,n_i}\Fvp))$ is
an allowable $R_{\vp}$-module, and thus we have a map
$$D(\S^{g_i,n_i}\Fvp)\to H_{*,*}(E_{\infty}(\S^{g_i,n_i}\Fvp))$$
of allowable $R_{\vp}$-modules. Moreover, 
$$H_{*,*}(E_{\infty}(\S^{g_1,n_1}\Fvp))\dc H_{*,*}(E_{\infty}(\S^{g_2,n_2}\Fvp))$$
is also an allowable $R_{\vp}$-module whose $R_{\vp}$-action satisfies the
external Cartan formulas. Thus, writing
$$X=D(\S^{g_1,n_1}\Fvp) \text{ and }Y=D(\S^{g_2.n_2}\Fvp),$$
we have that the diagram
$$
\begin{tikzcd}[cramped, column sep=small]
T_{\F\text{-}\F}(V)\biact (X\dc Y) \arrow{d}\arrow{r} 
& R_{\vp}\biact(H_{*,*}(E_{\infty}(\S^{g_1,n_1}\Fvp))\dc
H_{*,*}(E_{\infty}(\S^{g_2,n_2}\Fvp))) \arrow{d}\\
X\dc Y \arrow{r} & 
H_{*,*}(E_{\infty}(\S^{g_1,n_1}\Fvp))\dc H_{*,*}(E_{\infty}(\S^{g_2,n_2}\Fvp))
\end{tikzcd}
$$
commutes. Here the top arrow is the composition of the quotient map
$T_{\F\text{-}\F}(V)\to R_{\vp}$ with the canonical maps from $X$ and $Y$ to the
free $E_{\infty}$-algebras.
We know that $u\tensor[_{\F}]{\otimes}{_{\F}} (x\otimes y)$ maps to 0 in the
lower-right corner
$$H_{*,*}(E_{\infty}(\S^{g_1,n_1}\Fvp))\dc H_{*,*}(E_{\infty}(\S^{g_2,n_2}\Fvp)).$$
Thus, in order to see that it maps to zero in $D(\S^{g_1,n_1}\Fvp)\dc
D(\S^{g_2,n_2}\Fvp)$, it suffices to show that the map
$$D(\S^{g_1,n_1}\Fvp)\dc D(\S^{g_2,n_2}\Fvp)\to H_{*,*}(E_{\infty}(\S^{g_1,n_1}\Fvp))\dc
H_{*,*}(E_{\infty}(\S^{g_2,n_2}\Fvp))$$
is injective. Since we are working over a field, this in turn is equivalent to
showing that each map
$$D(\S^{g_i,n_i}\Fvp)\to H_{*,*}(E_{\infty}(\S^{g_i,n_i}\Fvp))$$
is injective.

To see this, we use the description of a basis
for $D(\S^{g_i,n_i}\Fvp)$ from Lemma \ref{lem:Dbasis}.
Explicitly, if $x$ is a singleton set in bidegree $(g,n)$, so that $\Fvp\{x\}\cong
\S^{g,n}\Fvp$, we have that
$$\{Q^I_g\tensor[_{\F}]{\otimes}{_{\F}} x \mid I \text{ admissible},
e(I)\geq n\}$$
is a basis for $D(\S^{g,n}\Fvp)$. By Theorem \ref{thm:W}, we know that
$H_{*,*}(E_{\infty}(\S^{g,n}\Fvp))$ is isomorphic to the free commutative algebra on
$$Q(\{x\})=\{Q^I_g(x)\mid  I\text{ admissible}, e(I)+\e_1>n\}.$$ From
Lemma \ref{lem:free comm alg}, we know that the free commutative algebra on
$Q(\{x\})$ has a basis given by (monotone increasing under a total ordering)
products of elements of $Q(\{x\})$. In particular, there is an injection
$$\Fvp\{Q(\{x\})\}\hookrightarrow H_{*,*}(E_{\infty}(\S^{g,n}\Fvp)).$$
Note that 
$$\{Q^I_g(x)\mid I \text{ admissible}, e(I)\geq n\}=Q(\{x\})\sqcup
\{Q^I_g(x)\mid I \text{ admissible}, \e_1=0, e(I)=n\}.$$
Moreover, the argument of Lemma \ref{lem:SSpieces} shows that an admissible $I$ satisfies $\e_1=0$, $e(I)=n$ precisely when
$$Q^I_g(x)=(Q^{I_{\geq 2}}_g(x))^p$$ 
and either $Q^{I_{\geq 2}}_g(x)\in Q(\{x\})$ or $\e_2=0$, $e(I_{\geq 2})=n$.
That is, $Q^I_g(x)=(Q^{I_{\geq i}}_g(x))^{p^{i-1}}$, where $Q^{I_{\geq i}}_g(x)\in
Q(\{x\})$ (here we allow $I_{\geq i}$ to be empty, in which case $i$ is taken to
be $k+1$.) Hence, $D(\S^{g,n}\Fvp)$ has a basis of elements that are either in
$Q(\{x\})$ or are iterated powers of elements of $Q(\{x\})$, which means there
is an injection
$$D(\S^{g,n}\Fvp)\hookrightarrow \mathcal{S}(Q(\{x\}))\cong
H_{*,*}(E_{\infty}(\S^{g,n}\Fvp)).$$
 Since we can choose this map to be consistent with the
canonical map $$D(\S^{g,n}\Fvp)\to H_{*,*}(E_{\infty}(\S^{g,n}\Fvp)),$$ it follows that this
canonical map is an injection.

Hence, by the above argument, we have a map
$$R_{\vp}\biact(X\dc Y)\to X\dc Y$$
for any allowable $R_{\vp}$-modules $X$ and $Y$ making $X\dc Y$ into a left
$R_{\vp}$-module. Moreover, that $X\dc Y$ is allowable follows from the
Cartan formulas: if $x\in X(g)_n$ and $y\in Y(h)_m$, and $2s<n+m$, then for any
$j,k$ with $j+k=s$, we must have either $2j<n$ or $2k<m$, and hence
$Q^j_g(x)\otimes Q^k_h(y)=0$ by allowability of $X$ and $Y$. This shows that
$Q^s_{g\oplus h}$ acts on  $x\otimes y$ by 0 in the appropriate range. A similar
argument shows that $\b Q^s_{g\oplus h}$ acts by zero in the appropriate range.
\end{proof}

We can equip $\Fvp$ with the structure of an allowable $R_{\vp}$-module as well,
by defining a map
$$V(Q)\biact \Fvp\to \Fvp$$
that is zero except on $V(Q_{\1_{\fG}}^0)\biact \Fvp$,
where it is determined by sending $Q^0_{\1_{\fG}}\otimes 1\mapsto 1$. 
Since no Adem relations contain
repeated instances of $Q^0_{\1_{\fG}}$, it is easy to see that this induces a map
$$R_{\vp}\biact \Fvp\to \Fvp$$
giving $\Fvp$ a left $R_{\vp}$-module structure. By construction, $\Fvp$ is
allowable.

Thus, using the product of Proposition \ref{prop:prodallow} and $\Fvp$ as the monoidal unit,
we may equip $\LMod_{R_{\vp}}^{\text{allow}}$ with a monoidal structure, and we
claim that the symmetry in $\Mod_{\Fvp}(\GrVect^{\fG})$ induces a symmetry on
this category as well.

\begin{prop}
The symmetric monoidal structure on $\Mod_{\Fvp}(\GrVect^{\fG})$ endows $\LMod_{R_{\vp}}^{\text{allow}}$ with the structure of a
symmetric monoidal category when products are given the left $R_{\vp}$-module
structure of Proposition \ref{prop:prodallow}.
\end{prop}

\begin{proof}
It remains to show that the symmetry map $X\dc Y\to Y\dc X$ in $\Mod_{\Fvp}(\GrVect^{\fC})$
is a map of left $R_{\vp}$-modules. Thus, we need to check commutativity of the
diagram
\begin{equation}\label{diagram:lmodsym}
\begin{tikzcd}
R_{\vp}\biact (X\dc Y) \arrow{r}\arrow{d} & R_{\vp}\biact (Y\dc X)
\arrow{d}\\
X\dc Y \arrow{r} & Y\dc X.
\end{tikzcd}
\end{equation}
To see this, recall from Equation \ref{eq:braiding} that the symmetry on $X(g_1)_{n_1}\otimes
Y(g_2)_{n_2}$ is given by applying the symmetry in $\GrVect$, which swaps the
factors and introduces a sign, and acting by the
braiding morphism $\b^{-1}_{g_1,g_2}$. Thus, applying the symmetry first and
then acting by $R_{\vp}$ introduces a sign $(-1)^{n_1n_2}$ and acts by
$\gpower{(\b_{g_1,g_2}^{-1})}{p}\circ \psi^{-1}$ (where as usual $\psi$ is the
shuffle map). On the other hand, acting by $Q_{g}^s$ and \emph{then} applying
the symmetry introduces a sign $(-1)^{(n_1+2j(p-1))(n_2+2(s-j)(p-1))}$ on each
term of the sum in the Cartan formula and acts by $\psi^{-1}\circ
\b^{-1}_{\gpower{g_1}{p},\gpower{g_2}{p}}.$ Since $2j(p-1)$ and $2(s-j)(p-1)$
are both even, these signs agree. Moreover,
$\gpower{(\b_{g_1,g_2}^{-1})}{p}\circ \psi^{-1}$ and
$\psi^{-1}\circ\b^{-1}_{\gpower{g_1}{p},\gpower{g_2}{p}}$ are both morphisms
$$\gpower{g_2}{p}\oplus \gpower{g_1}{p}\to \gpower{(g_1\oplus g_2)}{p}$$
defined using braidings and associators and that preserve the order of the
factors of $g_1$ and $g_2$. Hence, by the coherence axioms for symmetric
monoidal categories, they agree. Similarly, one can check that the signs agree
when acting by $\b Q_g^s$, and thus, Diagram \ref{diagram:lmodsym} commutes.
\end{proof}

\begin{rem}
For the classical Dyer--Lashof algebra (or similarly for the Steenrod algebra),
one may define this monoidal structure by first defining a coproduct map $R_{\vp}\to
R_{\vp}\otimes R_{\vp}$ making $R_{\vp}$ into a Hopf object. However, this does not work in our
situation for several reasons.

First, we allow $\Z$-graded operations rather
than merely nonnegatively-graded operations, so the sum
$$\sum Q^rQ^{s-r}$$
will not be finite. Even if we considered only nonnegatively-graded operations,
if $\ob{\fG}$ is infinite, then
$$\sum_{h_1\oplus h_2\cong h} Q^r_{h_1}Q^{r-s}_{h_2}$$
will be infinite as well. Thus, we cannot obviously define such a coproduct. It
was therefore important to restrict our attention to \emph{allowable}
$R_{\vp}$-modules in order to use this formula.

Moreover, if $\F$ is not equal to its prime subfield, then we must use the above
construction with bimodules, and the category of $\F$-$\F$-bimodules is not
braided monoidal. Hence, we could not do the usual construction
$$R_{\vp}\otimes X\otimes Y\to R_{\vp}\otimes R_{\vp}\otimes X\otimes Y\to
R_{\vp}\otimes X\otimes
R_{\vp}\otimes Y\to X\otimes Y$$
even if we had defined a coproduct on $R_{\vp}$,
since doing so requires braiding with a bimodule.
\end{rem}

\subsection{Allowable Dyer--Lashof algebras}
Now we are ready to define allowable Dyer--Lashof algebras, and we will show
that $H_{*,*}(E_{\infty}(X))$ is isomorphic to the free allowable Dyer--Lashof
algebra on $H_{*,*}(X)$.

Using the monoidal structure on $\LMod_{R_{\vp}}^{\text{allow}}$, the usual tensor algebra
construction defines a left adjoint $T$ to the forgetful functor
$$\LMod_{R_{\vp}}^{\text{allow}}\mathrel{\mathop{\rightleftarrows}^{T}}
\Mon(\LMod_{R_{\vp}}^{\text{allow}})$$
from the category of monoid objects in $\LMod_{R_{\vp}}^{\text{allow}}$.
 As usual, a quotient construction provides a left adjoint
$$\Mon(\LMod_{R_{\vp}}^{\text{allow}})\rightleftarrows
\CMon(\LMod_{R_{\vp}}^{\text{allow}})$$
to the forgetful functor from the category of commutative monoid objects in
$\LMod_{R_{\vp}}^{\text{allow}}$.

\begin{defn}
An \emph{allowable $R_{\vp}$-algebra} is a commutative monoid object $X$ in
$\LMod_{R_{\vp}}^{\text{allow}}$ with structure maps
$$\rho\colon R_{\vp}\biact X\to X \text{ in } \Mod_{\Fvp}(\GrVect^{\fC})$$ and
$$\mu\colon X\dc X\to X \text{ in } \LMod_{R_{\vp}}^{\text{allow}}$$ such that
if $x\in X(g)_m$ with $\vp(\tr_{g\oplus g}(\b_{g,g}))(-1)^m=1$, then when $2s=m$,
$$\rho(Q^s_g\otimes x)=x^{\otimes p},$$
where $x^{\otimes p}=\mu(x,\mu(x,\dots,\mu(x,x)))$ is the $p$-fold iterated
product.
\end{defn}
The allowable $R_{\vp}$-algebras form a full subcategory
$\Alg_{R_{\vp}}^{\text{allow}}$ of $\CMon(\LMod_{R_{\vp}}^{\text{allow}})$.

\begin{rem}
Note that by definition of the $R_{\vp}$ action on $X\dc X$ using the Cartan formulas, we
automatically have that the Cartan formulas hold in $X$. Moreover, the
assumption that $\vp(\tr_{g\oplus g}(\b_{g,g}))(-1)^m=1$ guarantees that
$Q^s_g\in Q$ if $2s=m$.
From the definition of the $R_{\vp}$-module structure on the monoidal unit
$\Fvp$, we automatically have that $Q^s_{\1_{\fG}}\otimes 1\mapsto 0$ when $s\not=0$.
\end{rem}

We have defined the category $\Alg_{R_{\vp}}^{\text{allow}}$ precisely so that
the operations on the homology of an $E_{\infty}$-algebra define an allowable
$R_{\vp}$-algebra structure. It is a direct consequence of Theorem
\ref{thm:upper index} that this is indeed true:

\begin{cor}
If $X\in \Alg_{E_{\infty}}(\Mod_{\Fvp})$, then $H_{*,*}(X)$ inherits the
structure of an allowable $R_{\vp}$-algebra.
\end{cor}

We define a free allowable $R_{\vp}$-algebra functor
$$\newW\colon \Mod_{\Fvp}(\GrVect^{\fC}_{\F})\to
\Alg_{R_{\vp}}^{\text{allow}}$$
 that factors as 
$$\Mod_{\Fvp}(\GrVect^{\fC})\xrightarrow{D} \LMod_{R_{\vp}}^{\text{allow}}\to
\CMon(\LMod_{R_{\vp}}^{\text{allow}})\to \Alg_{R_{\vp}}^{\text{allow}},$$
where we have shown that the first two functors are left adjoints, and we still
need to define the last functor.
To obtain an allowable $R_{\vp}$-algebra from an object $X$ in
$\CMon(\LMod_{R_{\vp}}^{\text{allow}})$ with $R_{\vp}$ action map $\rho\colon
R_{\vp}\biact
X\to X$, we take the quotient of $X$ by the set containing all elements of
$$\rho\left(Q^{\tfrac{m}{2}}_g \otimes x\right)- x^{p}$$ for $x\in X(g)_m$ with
$\vp(\tr_{g\oplus g}(\b_{g,g}))(-1)^m=1.$
Once again the universal property of the quotient proves that this quotient
functor is left adjoint to the inclusion functor. Thus, $\newW$ is defined as a
composition of left adjoints and is therefore itself left adjoint to the
forgetful functor.

In fact, we claim $\newW$ is simply a reformulation of
the construction $\W{S}{H_{*,*}(X)}$ from Section \ref{subsect:basis}; however,
$\newW$ comes
equipped with an action of the Dyer--Lashof operations, whereas
$\W{S}{H_{*,*}(X)}$ merely
has a product. We will now give a comparison of $\W{S}{H_{*,*}(X)}$ and
$\newW(H_{*,*}(X))$ in
$\Alg_{\Fvp}(\GrVect^{\fC})$.

Recall that $\W{S}{H_{*,*}(X)}$ is the free commutative algebra on a set $Q(S)$. Thus, in
order to construct a map $\W{S}{H_{*,*}(X)}\to \newW(H_{*,*}(X))$ in $\Alg_{\Fvp}$,
it suffices to construct a map
$$Q(S)\to \newW(H_{*,*}(X))$$ of graded sets. There is a map of sets
$Q(S)\to R_{\vp}\biact H_{*,*}(X)$ sending $Q^I_g x\mapsto Q^I_g\otimes x$. Thus,
we have a map
$$Q(S)\to R_{\vp}\biact H_{*,*}(X)\to D(H_{*,*}(X))\to T(D(H_{*,*}(X)))\to
\newW(H_{*,*}(X)),$$
where $R_{\vp}\biact -\to D(-)$ and $T(D(-))\to \newW(-)$ are quotient maps.
In other words, the map $R_{\vp}\biact H_{*,*}(X)\to \newW(H_{*,*}(X))$ in $\LMod_{R_{\vp}}$ is adjoint to the
identity map $\newW(H_{*,*}(X))\to \newW(H_{*,*}(X))$ in
$\Alg_{R_{\vp}}^{\text{allow}}$.
This induces a map
$$\W{S}{H_{*,*}(X)}\to \newW(H_{*,*}(X))$$
of algebras.

\begin{prop}\label{prop:same}
Let $X\in \Mod_{\Fvp}(\Ch^{\fC})$, and let $S$ be a basis for $H_{*,*}(X)$.
Then the map $$\W{S}{H_{*,*}(X)}\to \newW(H_{*,*}(X))$$
 in $\Alg_{\Fvp}(\GrVect^{\fC})$ is an isomorphism.
\end{prop}

Before proving this, we first show that this map factors the isomorphism
$$\W{S}{H_{*,*}(X)}\to H_{*,*}(E_{\infty}(X)).$$ 
The usual map
$$H_{*,*}(X)\to H_{*,*}(E_{\infty}(X))$$
in $\GrVect^{\fC}$ induces a map
$$\newW(H_{*,*}(X))\to H_{*,*}(E_{\infty}(X))$$
of allowable $R_{\vp}$-algebras that is natural in $X$. As the
isomorphism $\W{S}{H_{*,*}(X)}\to H_{*,*}(E_{\infty}(X))$ was defined using the map
$$S\to H_{*,*}(X)\to H_{*,*}(E_{\infty}(X)),$$
the maps $Q(S)\to H_{*,*}(E_{\infty}(X))$ and $Q(S)\to
\newW(H_{*,*}(X))\to H_{*,*}(E_{\infty}(X))$ agree. Since $\W{S}{H_{*,*}(X)}$ is the
free commutative algebra on $Q(S)$, it therefore follows that
the diagram
$$
\xymatrix{
\W{S}{H_{*,*}(X)} \ar[r]\ar[d]  & H_{*,*}(E_{\infty}(X))\\
\newW(H_{*,*}(X))\ar[ur] &
}
$$
commutes, where the diagonal arrow is considered a map in
$\Alg_{\Fvp}(\GrVect^{\fC})$ after forgetting the $R_{\vp}$-module structure.

\begin{proof}
Since the map $\W{S}{H_{*,*}(X)}\to H_{*,*}(E_{\infty}(X))$ is an isomorphism by Theorem
\ref{thm:W}, the map $\W{S}{H_{*,*}(X)}\to \newW(H_{*,*}(X))$ must be injective.

Now we show that this map is surjective as well.
From Lemma \ref{lem:Dbasis}, we know that 
$$Q'=\{Q^I_g\otimes x\mid I \text{ admissible},\, x\in S(g)_n,\,e(I)\geq n\}$$
is a basis for $D(H_{*,*}(X))$. 
Moreover,  we have that
$$Q'\cong Q(S)\sqcup \{Q^I_g\otimes x\mid I \text{ admissible},\, x\in S(g)_n,
\,\e_1=0,\,e(I)=n\}.$$
Thus, the underlying commutative algebra of $\newW(H_{*,*}(X))$ is a
quotient of the free commutative algebra $\mathcal{S}(Q')$ on $Q'$.
Choose a total ordering on $Q'$ so that a basis of $\mathcal{S}(Q')$ is given by
$\{q_1\cdots q_j\}$, where $q_i\in Q'$, $q_i\leq q_{i+1}$, and $q_i=q_{i+1}$ is
only allowed when $q_i$ has bidegree $(g,n)$ for $\vp(\tr_{g\oplus g}(\b_{g,g}))(-1)^n=1$. Here $j$
is any nonnegative integer, and $j=0$ corresponds to the empty product.

Restricting this total ordering to the subset $Q(S)$ of $Q'$, we obtain a basis
for $\W{S}{H_{*,*}(X)}$ of elements $\{q_1\cdots q_j\}$ as above, except that we require
$q_i\in Q(S)$. Thus, the obvious map $\W{S}{H_{*,*}(X)}\to \mathcal{S}(Q')$ is an
injection. 
The cokernel of this map has a basis of elements $\{q_1\cdots q_j\}$ as above,
but where at least one $q_i\in Q'\setminus Q(S)$. In other words,
$q_i=Q^I_g\otimes x$, where $\e_1=0$ and $e(I)=d_{\Z}(x)$.
By Lemma \ref{lem:admprops} \ref{list:powers}, the condition $\e_1=0$, $e(I)=n$ is equivalent to having that the image of
$Q^I_g\otimes x$ in $\newW(H_{*,*}(X))$ is $(Q^{I_{\geq i}}_g\otimes
x)^{p^{i-1}}$, where $Q^{I_{\geq i}}_g\otimes x\in Q(S)$ and $2\leq i\leq k+1$
is uniquely determined. Since $(Q^{I_{\geq i}}_g\otimes x)^{p^{i-1}}$ for
$Q^{I_{\geq i}}_g\otimes x\in Q(S)$ is in the
image of the composition
$$\W{S}{H_{*,*}(X)}\to \mathcal{S}(Q')\to \newW(H_{*,*}(X)),$$
it follows that this composition is surjective (since the latter map is
surjective, and every element in the cokernel of the first map is mapped to
something in the span of $\W{S}{H_{*,*}(X)}$).

\end{proof}

Putting this together with the result of Theorem \ref{thm:W}, we have that the
free functor $\newW$ describes the homology of free $E_{\infty}$-algebras.
That is, the homology of a free $E_{\infty}$-algebra is isomorphic to a free allowable
$R_{\vp}$-algebra.

\begin{cor}\label{cor:free}
There is a natural isomorphism
$$\newW(H_{*,*}(-))\to H_{*,*}(E_{\infty}(-))$$
of functors from $\Mod_{\Fvp}(\Ch^{\fC})$ to $\Alg_{R_{\vp}}^{\text{allow}}$.
\end{cor}

\appendix
\section{A spectral sequence comparison theorem}\label{app:SScomp}
The following is an analogue of the usual spectral sequence comparison theorem
of Zeeman
\cite{Zeeman}. Here we work with spectral sequences in the abelian category
$(\Vect_{\F}^{\fC})^{\N}$ of $\N$-graded objects of $\Vect^{\fC}_{\F}$. The $\N$
grading here corresponds to the \emph{charge} grading in the rest of the paper,
so we always write
$$\arity{F}{k}$$
for the summand of $F$ with grading $k$.

\begin{thm}\label{thm:sscomp}
Suppose $\tilde{E}$ and $E$ are spectral sequences in 
$(\Vect_{\F}^{\fC})^{\N}$ such that
$\arity{\tilde{E}^r_{p,q}}{k}=\arity{E^r_{p,q}}{k}=0$ if $p<0$ or $p>k$.
Moreover suppose that
$$\arity{E^1_{p,q}}{k}\cong \bigoplus_t \arity{E^1_{p,t}}{p}\dc
\arity{E^1_{0,q-t}}{k-p}$$
and
$$\arity{\tilde{E}^1_{p,q}}{k}\cong \bigoplus_t \arity{\tilde{E}^1_{p,t}}{p}\dc
\arity{\tilde{E}^1_{0,q-t}}{k-p}.$$
Let $\Psi\colon \tilde{E}\to E$ be a map of spectral sequences, and write
$\arity{\Psi^r_{p,q}}{k}$ to denote the corresponding map
$$\arity{\tilde{E}^r_{p,q}}{k}\to \arity{E^r_{p,q}}{k}.$$ If
\begin{itemize}
\item $\arity{\Psi^{\infty}_{p,q}}{k}$ is an isomorphism for all $p,q$;
\item $\arity{\Psi^1_{p,q}}{\ell}$ is an isomorphism for all $p,q$ if $\ell<k$;
\item $\arity{\Psi^1_{k,q}}{k}$ is an isomorphism for $q\leq M$ and a surjection for
$q=M+1$,
\end{itemize}
then
$$\arity{\Psi^1_{0,q}}{k}\colon \arity{\tilde{E}^1_{0,q}}{k}\to \arity{E^1_{0,q}}{k}$$
is an isomorphism for $q\leq M+k-1$ and a surjection for $q=M+k$.
\end{thm}

\begin{proof}
Write $\arity{Z^r_{p,q}}{\ell}$ (respectively $\arity{\tilde{Z}^r_{p,q}}{\ell}$) for the kernel
of $\arity{d^r_{p,q}}{\ell}$ ($\arity{\tilde{d}^r_{p,q}}{\ell}$) and
$\arity{B^r_{p,q}}{\ell}$
($\arity{\tilde{B}^r_{p,q}}{\ell}$) for the
image of $\arity{d^r_{p+r,q-r+1}}{\ell}$ ($\arity{\tilde{d}^r_{p+r,q-r+1}}{\ell}$) in
$\arity{E^r_{p,q}}{\ell}$. We write $\arity{z^r_{p,q}}{\ell}$ for the map
$\arity{\Psi^r_{p,q}}{\ell}$
restricted to $\arity{\tilde{Z}^r_{p,q}}{\ell}$ and $\arity{b^r_{p,q}}{\ell}$ for the
restriction to $\arity{\tilde{B}^r_{p,q}}{\ell}$. We make repeated use of the short
exact sequences
$$0\to \arity{B^r_{p,q}}{\ell}\to \arity{Z^r_{p,q}}{\ell}\to
\arity{E^{r+1}_{p,q}}{\ell}\to 0$$
and
$$0\to \arity{Z^r_{p,q}}{\ell}\to \arity{E^r_{p,q}}{\ell}\xrightarrow{d^r}
\arity{B^r_{p-r,q+r-1}}{\ell}\to 0,$$
as well as the same exact sequences in $\tilde{E}$.
In particular, these short exact sequences give the following implications:
\begin{enumerate}
\item If $\arity{\Psi^r_{p,q}}{\ell}$ is injective, then
$\arity{z^r_{p,q}}{\ell}$ and $\arity{b^r_{p,q}}{\ell}$ are both
injective.
\item If $\arity{\Psi^r_{p+r,q-r+1}}{\ell}$ is surjective, then $\arity{b^r_{p,q}}{\ell}$ is surjective.
\item If $\arity{\Psi^r_{p,q}}{\ell}$ is surjective and
$\arity{b^r_{p-r,q+r-1}}{\ell}$ is injective, then
$\arity{z^r_{p,q}}{\ell}$ is surjective.
\item If $\arity{b^r_{p,q}}{\ell}$ is surjective and $\arity{z^r_{p,q}}{\ell}$ is injective, then
$\arity{\Psi^{r+1}_{p,q}}{\ell}$ is injective.
\item If $\arity{z^r_{p,q}}{\ell}$ is surjective, then
$\arity{\Psi^{r+1}_{p,q}}{\ell}$ is surjective.
\end{enumerate}

Note that given the description of the $E^1$ pages of the spectral sequences as
products and the assumption that $\arity{\Psi_{p,q}^1}{\ell}$ is an isomorphism when
$\ell<k$, we see that $\arity{\Psi_{p,q}^1}{k}$ is always an isomorphism if $0<p<k$.

For the remainder of the proof, we use $\ell=k$, and therefore we omit this from
the notation. We begin by proving the following intermediate statement.
\begin{claim}
Suppose $1\leq r\leq k-1$. Then
\begin{align*}
b^r_{p,q} \text{ is }&\text{injective for all } q \text{ if } 1\leq p\leq k\\
 &\text{surjective }\begin{cases}
\text{ for all } q & 0\leq p\leq k-r-1 \text{ or } k-r<p\\
q\leq M+r & p=k-r
\end{cases}
\end{align*}
and
\begin{align*}
z^r_{p,q}\text{ is } &\text{injective } \begin{cases}
\text{for all }q & 1\leq p\leq k-r\\
q\leq M+k-p & k-r<p\leq k
\end{cases}\\
&\text{surjective } \begin{cases}
\text{for all }q & r+1\leq p<k\\
q\leq M+1 & p=k.
\end{cases}
\end{align*}
In particular, using statements (iv) and (v), this implies that
\begin{align*}
\Psi^{r+1}_{p,q} \text{ is }&\text{injective } \begin{cases}
\text{for all }q & 1\leq p<k-r\\
q\leq M+k-p & k-r\leq p\leq k
\end{cases}\\
&\text{surjective }\begin{cases}
\text{for all }q & r+1\leq p<k\\
q\leq M+1 & p=k.
\end{cases}
\end{align*}
\end{claim}
We prove this by induction on $r$. Note that for all $p>k-r$,
$B^r_{p,q}=\tilde{B}^r_{p,q}=0$, since our spectral sequences are zero for
$p>k$, so it suffices to consider the range $p\leq k-r$ when proving the 
above statements about $b^r_{p,q}$.
First consider $r=1$. Since $\Psi^1_{p,q}$ is injective for all $q$ when $1\leq
p\leq k-1$ and for $q\leq M$ when $p=k$, the injectivity statements for
$b^1_{p,q}$ and $z^1_{p,q}$ hold by statement (i). Since $\Psi^1_{p+1,q}$ is
surjective for all $q$ when $1\leq p+1\leq k-1$ and for $q\leq M+1$ when
$p+1=k$, the surjectivity statement for $b^1_{p,q}$ follows by statement (ii).
Finally, $\Psi^1_{p,q}$ is surjective for all $q$ when $1\leq p\leq k-1$ and for
$q\leq M+1$ when $p=k$, and $b^1_{p-1,q}$ is injective for all $q$ when $1\leq
p-1\leq k$, so by statement (iii) we have that $z^1_{p,q}$ is surjective on the
intersection of these ranges, namely for all $q$ when $2\leq p\leq k-1$ and for
$q\leq M+1$ when $p=k$. This proves the base case $r=1$.

Suppose we have the claim for $b^{r-1}_{p,q}$ and $z^{r-1}_{p,q}$, so that we
may assume that
\begin{align*}
\Psi^{r}_{p,q} \text{ is }&\text{injective } \begin{cases}
\text{for all }q & 1\leq p\leq k-r\\
q\leq M+k-p & k-r<p\leq k
\end{cases}\\
&\text{surjective }\begin{cases}
\text{for all }q & r\leq p<k\\
q\leq M+1 & p=k.
\end{cases}
\end{align*}
The injectivity ranges for $\Psi^r_{p,q}$ immediately give the injectivity
ranges for $z^r_{p,q}$ and $b^r_{p,q}$ by statement (i).
Since $\Psi^r_{p+r,q-r+1}$ is surjective for all $q$ when $r\leq p+r<k$ and for
$q-r+1\leq M+1$ when $p+r=k$, we have that $b^r_{p,q}$ is surjective for all $q$
when $0\leq p<k-r$ and for $q\leq M+r$ when $p=k-r$ by statement (ii).
Finally, $\Psi^r_{p,q}$ is surjective for all $q$ when $r\leq p<k$ and for
$q\leq M+1$ when $p=k$, and $b^r_{p-r,q+r-1}$ is injective for all $q$ if $1\leq
p-r\leq k$, so $z^r_{p,q}$ is surjective in the intersection of this range,
namely when $r+1\leq p\leq k-1$ for all $q$ and for $q\leq M+1$ if $p=k$.
This proves the claim.

For the remainder of the proof we will make a number of arguments using the
claim together with the shape of both spectral sequences. For ease of notation,
we only write most statements in $E$, but since $\tilde{E}$ has the same shape,
the same statements will hold for $\tilde{E}$ as well.

Note that $d^r\colon E^r_{0,q}\to 0$ is always zero, so $Z^r_{0,q}\cong E^r_{0,q}$.
Since $d^r\colon E^r_{r,q-r+1}\to E^r_{0,q}$ is zero if $r>k$, we therefore have that $$E^1_{0,q}\cong
E_{0,q}^{\infty}\oplus
\bigoplus_{r=1}^k B^r_{0,q},$$
and a similar statement holds for $\tilde{E}^1_{0,q}$. These splittings are
compatible with the maps $b^r_{0,q}$ and $\Psi^{\infty}_{0,q}$, and hence, to show that $\Psi^1_{0,q}$ is
injective/surjective, it suffices to show that all the $b^r_{0,q}$ are
injective/surjective (since we have assumed that $\Psi^{\infty}$ is always an
isomorphism).

From the claim we know that $b^r_{0,q}$ is surjective for all $q$ if $1\leq
r\leq k-1$. By definition, $d^k\colon E^k_{k,q-k+1}\to B^k_{0,q}$ is surjective, as is
$\tilde{d}^k$ in the same degree. From the claim,
we know that $\Psi^k_{k,q-k+1}$ is surjective when $q-k+1\leq M+1$. Putting this
all together, we have that all of the $b^r_{0,q}$ (and thus $\Psi^1_{0,q}$ as
well) are surjective when $q\leq M+k$. Note that this holds when $k=1$ as well,
since in this case we only need to consider $b_{0,q}^k$.

Now we investigate when the $b^r_{0,q}$ are injective by considering the short
exact sequences
$$0\to Z^r_{r,q-r+1}\to E^r_{r,q-r+1}\to B^r_{0,q}\to 0.$$ If $z^r_{r,q-r+1}$ is
surjective and $\Psi^r_{r,q-r+1}$ is injective, then $b^r_{0,q}$ is injective.

First suppose $k-r<r$.
Then $B^{r+t}_{r,q-r+1}=0$ for all $t\geq 0$ since our spectral sequences are
zero for $p>k$, so $Z^{r+t}_{r,q-r+1}\cong E^{r+t+1}_{r,q-r+1}$ for any $t\geq
0$. Moreover, $d^{r+t}_{r,q-r+1}=0$ for any $t>0$, so $Z^{r+t}_{r,q-r+1}\cong
E^{r+t}_{r,q-r+1}$ for any $t>0$. Thus, 
$Z^r_{r,q-r+1}\cong E^{\infty}_{r,q-r+1}$, so
$z_{r,q-r+1}^r$ is surjective (in fact an isomorphism) for all $q$.
Moreover, from the claim, we
know that $\Psi^r_{r,q-r+1}$ is injective when $q-r+1\leq M+k-r$, that is, when
$q\leq M+k-1$, so $b^r_{0,q}$ is injective when $q\leq M+k-1$. Note that this is
also true when $k=1$, since $\Psi^1_{k,q}$ is injective when $q\leq M$ by
assumption.

If $k-r=r$, then we have $E^{r+1}_{r,q-r+1}\cong E^{\infty}_{r,q-r+1}$, so there
is a short exact sequence
$$0\to B^r_{r,q-r+1}\to Z^r_{r,q-r+1}\to E^{\infty}_{r,q-r+1}\to 0.$$
We know that $\Psi^{\infty}$ is always surjective, and
 $b^r_{r,q-r+1}$ is surjective when
$q-r+1\leq M+r$. Thus, $z^r_{r,q-r+1}$ is surjective when $q\leq M+2r-1=M+k-1$. Moreover, $\Psi^r_{r,q-r+1}$
is injective for all $q$, so $b^r_{0,q}$ is injective for $q\leq M+k-1$.

Finally, if $k-r>r$, then $d^{r+t}_{r,*}=0$ for all $t>0$, and thus
$$Z^r_{r,q-r+1}\cong E^{\infty}_{r,q-r+1}\oplus \bigoplus_{t=0}^{k-2r} B^{r+t}_{r,q-r+1}.$$
We have that $b^{r+t}_{r,q-r+1}$ is surjective for all $q$ unless $r=k-r-t$, in
which case we need $q-r+1\leq M+r+t$,
or, equivalently $q\leq M+k-1$. Once again,
$\Psi^r_{r,q-r+1}$ is injective for all $q$, so $b^r_{0,q}$ is injective whenever
$q\leq M+k-1$.

Thus, we have shown that $\Psi^1_{0,q}$ is injective if $q\leq M+k-1$ and
surjective if $q\leq M+k$.
\end{proof}

\newpage
\bibliographystyle{alpha}
\bibliography{Einfty_ref.bib}

\end{document}